\documentclass[11pt]{article}

\RequirePackage[OT1]{fontenc}
\RequirePackage{amsthm,amsmath}
\RequirePackage[colorlinks,citecolor=blue,urlcolor=blue]{hyperref}

\usepackage{amssymb,latexsym}
\usepackage{dsfont, color}
\usepackage{verbatim}

\usepackage{makeidx}

\usepackage[T1]{fontenc}
\usepackage[english]{babel}

\usepackage[latin1]{inputenc}
\usepackage{amsfonts}
\usepackage{amsmath, color,hyperref}
\usepackage{epsfig}%pour inserer figures
\usepackage{fancybox}%pour inserer programmes
\usepackage{amsthm}% pour proofs
\usepackage{graphicx}
\usepackage{authblk}

\numberwithin{equation}{section}
\theoremstyle{plain}
\newtheorem{theorem}{Theorem}[section]
\newtheorem{proposition}{Proposition}[section]
\newtheorem{lemma}{Lemma}[section]
\newtheorem{remark}{Remark}[section]
\newtheorem{corollary}{Corollary}[section]

\newcommand{\GL}{\mathbb{G}_N^{\pi}}

\newcommand{\FHT}{\mathbb{F}_{N}^{\mathrm{HT}}}
\newcommand{\FHJ}{\mathbb{F}_N^{\mathrm{HJ}}}
\newcommand{\FN}{\mathbb{F}_N}

\newcommand{\SigmaHTFN}{\mathbf{\Sigma}_{\mathrm{HT}}}
\newcommand{\SigmaHTF}{\mathbf{\Sigma}^F_{\mathrm{HT}}}

\newcommand{\SigmaHJF}{\mathbf{\Sigma}^F_{\mathrm{HJ}}}

\newcommand{\tqed}{$\Box$}
\newcommand{\Log}{\mathrm{Log}}

\def\ssum{\mathop{\sum\sum}}
\def\sssum{\mathop{\sum\sum\sum}}
\def\ssssum{\mathop{\sum\sum\sum\sum}}

\author{H\'el\`ene Boistard}
\affil{Toulouse School of Economics}
\author{Hendrik P. Lopuha\"a} 
\affil{Delft University of Technology}
\author{Anne Ruiz-Gazen}
\affil{Toulouse School of Economics}
\title{Functional central limit theorems\\ for single-stage samplings designs}

\begin{document}
\date{}
\maketitle
%\begin{frontmatter}
%\runtitle{Functional CLTs for single-stage samplings designs}
%\thankstext{T1}{Footnote to the title with the ``thankstext'' command.}

%\begin{aug}
%\author{\fnms{H\'el\`ene} \snm{ Boistard}\thanksref{m1}\ead[label=e1]{helene@boistard.fr}},
%\author{\fnms{Hendrik P.} \snm{Lopuha\"a}\thanksref{m2}\ead[label=e2]{h.p.lopuhaa@tudelft.nl}}\\
%\and
%\author{\fnms{Anne} \snm{Ruiz-Gazen}\thanksref{m1}
%\ead[label=e3]{anne.ruiz-gazen@tse-fr.eu}}
%\begin{aug}
%\author{\fnms{H\'el\`ene} \snm{Boistard}\ead[label=e1]{helene@boistard.fr}},
%\author{\fnms{Hendrik P.} \snm{Lopuha\"a}\ead[label=e2]{h.p.lopuhaa@tudelft.nl}}\\
%\and
%\author{\fnms{Anne} \snm{Ruiz-Gazen}
%\ead[label=e3]{anne.ruiz-gazen@tse-fr.eu}}
%
%
%%\ead[label=u1,url]{http://www.foo.com}}
%
%%\thankstext{t1}{Some comment}
%%\thankstext{t2}{First supporter of the project}
%%\thankstext{t3}{Second supporter of the project}
%\runauthor{H. Boistard et al.}

%\affiliation{Toulouse School of Economics\thanksmark{m1} and Delft University of Technology\thanksmark{m2}}

%\address{H\'el\`ene Boistard and Anne Ruiz-Gazen\\
%Toulouse School of Economics\\
%21 all\'ee de Brienne\\
%31000 Toulouse, France\\
%\printead{e1}\\
%\phantom{E-mail:\ }\printead*{e3}}
%
%\address{Hendrik P. Lopuha\"a\\
%Delft Institute of Applied Mathematics\\
%Delft University of Technology\\
%Delft, The Netherlands\\
%\printead{e2}\\
%}
%\end{aug}

\begin{abstract}
For a joint model-based and design-based inference, we establish
functional central limit theorems for the Horvitz-Thompson empirical process and the
H\'ajek empirical process centered by their finite population mean as well as by their
super-population mean in a survey sampling framework.
The results apply to {single-stage unequal probability} sampling designs and essentially only require conditions on higher order correlations.
We apply our main results to a Hadamard differentiable statistical functional and illustrate its limit behavior
by means of a computer simulation.
\end{abstract}


%\begin{keyword}[class=MSC]
%\kwd[Primary ]{62D05}
%%\kwd{60K35}
%%\kwd[; secondary ]{60K35}
%\end{keyword}

%\begin{keyword}
%\kwd{design and model-based inference}
%\kwd{H\'ajek Process}
%\kwd{Horvitz-Thompson process}
%\kwd{rejective sampling}
%\kwd{Poisson sampling}
%\kwd{high entropy designs}
%\kwd{poverty rate}
%\end{keyword}
%
%\end{frontmatter}

\section{Introduction}
\label{sec:intro}
Functional central limit theorems are well established in statistics.
Much of the theory has been developed for empirical processes of independent summands.
In combination with the functional delta-method
they have become a very powerful tool for investigating the limit behavior for Hadamard or Fr\'echet differentiable statistical functionals
(e.g., see~\cite{van_1996} or~\cite{vandervaart1998} for a rigorous treatment with several applications).

In survey sampling, results on functional central limit theorems are far from complete.
At the same time there is a need for such results.
For instance, in~\cite{dell2008} the limit distribution of several statistical functionals is investigated,
under the assumption that such a limit theorem exists for a design-based empirical process,
whereas in~\cite{barrettdonald2009} the existence of a functional central limit theorem is assumed, to
perform model-based inference on several Gini indices.
Weak convergence of processes in combination with the delta method are treated in~\cite{bhattacharya2007}, \cite{davidson2009}, \cite{bhattacharyamazumder2011},
but these results are tailor made for specific statistical functionals, and do not apply to the empirical processes
that are typically considered in survey sampling.

Recently, functional central limit theorems for empirical processes in survey sampling have appeared in the literature.
Most of them are concerned with empirical processes indexed by a class of functions,
see~\cite{BreslowWellner_2007},\cite{Saegusa_2013}, and~\cite{Bertail_2013}.
{
Weak convergence under finite population two-phase stratified sampling, is established in~\cite{BreslowWellner_2007} and~\cite{Saegusa_2013}
for an empirical process indexed by a class of functions,
which is comparable to our Horvitz-Thompson empirical process in Theorem~\ref{th:FCLT HT F}.
Although their functional CLT allows general function classes, it only covers sampling designs
with equal inclusion probabilities within strata that assume exchangeability of the inclusion indicators,
such as simple random sampling and Bernoulli sampling.
Their approach uses results on exchangeable weighted bootstrap for empirical processes
from~\cite{Praestgaard_1993}, as incorporated in~\cite{van_1996}.
This approach, in particular the application of Theorem~3.6.13 in~\cite{van_1996},
seems difficult to extend to more complex sampling designs that go beyond exchangeable inclusion indicators.
In~\cite{Bertail_2013} a functional CLT is established, for a variance corrected Horvitz-Thompson empirical process under Poisson sampling.
In this case, one deals with a summation of independent terms, which allows the use of Theorem~2.11.1 from~\cite{van_1996}.
From their result a functional CLT under rejective sampling can then be established for the design-based Horvitz-Thompson process.
This is due to the close connection between Poisson sampling and rejective sampling.
For this reason, the approach used in~\cite{Bertail_2013} seems difficult to extend to other sampling designs.

Empirical processes indexed by a real valued parameter are considered in~\cite{Wang_2012}, \cite{conti2014}, and~\cite{conti2015}.
A functional CLT for the H\'ajek empirical c.d.f.~centered around the super-population mean is formulated in~\cite{Wang_2012},
and a similar result is implicitly conjectured for the Horvitz-Thompson empirical process.
Unfortunately, the paper seems to miss a number of assumptions  and the argument establishing Billingsley's tightness condition seems incomplete.
As a consequence, assumption~5 in~\cite{Wang_2012} differs somewhat from our conditions (C2)-(C4).
The remaining assumptions in~\cite{Wang_2012} are comparable to the ones needed for our Theorem~\ref{th:FCLT HJ F}.
\cite{conti2014} and~\cite{conti2015} consider high entropy designs, i.e., sampling designs that are close in Hellinger distance to the rejective
sampling design.
Functional CLT's are obtained for the Horvitz-Thompson (see~\cite{conti2014}) and H\'ajek (see~\cite{conti2015})
empirical c.d.f.'s both centered around the finite population mean.
}

The main purpose of the present paper is to establish functional central limit theorems for the
Horvitz-Thompson and the H\'ajek empirical distribution function that apply to general {{single-stage} unequal probability} sampling designs.
{In {the} context of weighted likelihood, the Horvitz-Thompson empirical process is a particular case of the inverse probability weighted empirical process
which is not necessarily the most efficient, see \cite{Robins_1994}. Its efficiency can be improved by using estimated weights, see \cite{Saegusa_2013}.
In the present paper we do not follow this path of the literature. We rather focus on the Horvitz-Thompson and the H\'ajek empirical processes that are related
to the Horvitz-Thompson and H\'ajek distribution function estimators as defined for example in \cite{Dorfman_2009}.}
For design-based inference about finite population parameters, these empirical distribution functions will be centered
around their population mean.
On the other hand, in many situations involving survey data, one is interested in the corresponding model parameters
(e.g., see~\cite{KornGraubard1998} and~\cite{Binder_2009}).
Recently, Rubin-Bleuer and Schiopu Kratina~\cite{Rubin-Bleuer_Kratina_2005} defined a mathematical framework for joint model-based and design-based inference through a probability product-space
and introduced a general and unified methodology for studying the asymptotic properties of model parameter estimators.
To incorporate both types of inferences, we consider the Horvitz-Thompson empirical process and the H\'ajek empirical process
under the super-population model described in~\cite{Rubin-Bleuer_Kratina_2005}, both centered around
their finite population mean as well as around their super-population mean.
Our main results are functional central limit theorems
for both empirical processes indexed by a real valued parameter
and apply to generic sampling schemes.
These results are established only requiring the usual standard assumptions that one encounters in asymptotic theory in survey sampling.
Our approach was inspired by an unpublished manuscript from Philippe Fevrier and Nicolas Ragache,
which was the outcome of an internship at INSEE in 2001.

The article is organized as follows.
Notations and assumptions are discussed in Section~\ref{sec:notation}.
In particular we briefly discuss the joint
model-based and design-based inference setting defined in~\cite{Rubin-Bleuer_Kratina_2005}.
In Sections~\ref{sec:HT} and~\ref{sec:Hajek}, we list the assumptions and state our main results.
Our assumptions essentially concern the inclusion probabilities of the sampling design up to the fourth order and a
central limit theorem (CLT)
for the Horvitz-Thompson estimator of a population total for i.i.d.~bounded random variables.
Our results allow random inclusion probabilities and are stated in terms of the design-based expected sample size,
but we also formulate more detailed results in case these quantities are deterministic.
{In Section~\ref{sec:examples} we discuss two specific examples: high entropy sampling designs and fixed size sampling designs with deterministic inclusion probabilities.
It turns out that in these cases the conditions used for general {single-stage} unequal probability sampling designs can be simplified.}

As an application of our results, in combination with the functional delta-method,
we obtain the limit distribution of the poverty rate in Section~\ref{sec:hadamard}.
%{Finally, this example is further investigated in Section~\ref{sec:simulation} by means of a simulation.}
This example is further investigated in Section~\ref{sec:simulation} by means of a simulation.
{Finally, in Section~\ref{sec:related} we discuss our results in relation to more complex designs.}
All proofs are deferred to Section~\ref{sec:proofs} and some tedious technicalities can be found in~\cite{Boistard_2015}.

\section{Notations and assumptions}
\label{sec:notation}
We adopt the super-population setup as described in~\cite{Rubin-Bleuer_Kratina_2005}.
Consider a sequence of finite populations $(\mathcal{U}^N)$, of sizes $N=1,2,\ldots$.
With each population we associate a set of indices $U_{N}=\{1,2,\ldots,N\}$.
Furthermore, for each index $i\in U_{N}$, we have a tuple $(y_i,z_i)\in \mathbb{R}\times \mathbb{R}_+^{q}$.
We denote $\mathbf{y}^N=(y_1,y_2,\ldots,y_{N})\in \mathbb{R}^{N}$
and $\mathbf{z}^N\in \mathbb{R}_+^{q\times N}$ similarly.
The vector~$\mathbf{y}^N$ contains the values of the variable of interest and
$\mathbf{z}^N$ contains information for the sampling design.
We assume that the values in each finite population are realizations of random variables
$(Y_i,Z_i)\in \mathbb{R}\times \mathbb{R}_+^{q}$, for $i=1,2,\ldots,N$,
on a common probability space $(\Omega,\mathfrak{F},\mathbb{P}_m)$.
Similarly, we denote $\mathbf{Y}^N=(Y_1,Y_2,\ldots,Y_{N})\in \mathbb{R}^{N}$
and $\mathbf{Z}^N\in \mathbb{R}_+^{q\times N}$.
To incorporate the sampling design, a product space is defined as follows.
For all $N=1,2,\ldots$, let
$\mathcal{S}_N=\left\{s: s\subset U_{N}\right\}$
be the collection of subsets of $U_{N}$ and let $\mathfrak{A}_N=\sigma(\mathcal{S}_N)$ be the $\sigma$-algebra generated
by~$\mathcal{S}_N$.
A sampling design associated to some sampling scheme is a function
${P}:\mathfrak{A}_N\times \mathbb{R}_+^{q\times N}\mapsto [0,1]$,
such that
\begin{itemize}
  \item[(i)]
  for all $s\in  \mathcal{S}_N$, $\mathbf{z}^N\mapsto {P}(s,\mathbf{z}^N)$ is a Borel-measurable function on $\mathbb{R}_+^{q\times N}$.
  \item[(ii)]
  for all $\mathbf{z}^N\in \mathbb{R}_+^{q\times N}$,
  $A\mapsto {P}(A,\mathbf{z}^N)$
is a probability measure on $\mathfrak{A}_N$.
\end{itemize}
Note that for each $\omega\in\Omega$, we can define a probability measure
$A\mapsto\mathbb{P}_d(A,\omega)=\sum_{s\in A}{P}(s,\mathbf{Z}^{N}(\omega))$
on the design space $(\mathcal{S}_N,\mathfrak{A}_N)$.
Corresponding expectations will be denoted by $\mathbb{E}_d(\cdot,\omega)$.
Next, we define a product probability space that includes the super-population
and the design space, under the premise that sample selection and the model characteristic
are independent given the design variables.
Let $(\mathcal{S}_N\times \Omega, \mathfrak{A}_N\times \mathfrak{F})$
be the product space with probability measure $\mathbb{P}_{d,m}$ defined on simple rectangles $\{s\}\times E\in \mathfrak{A}_N\times \mathfrak{F}$
by
\[
\mathbb{P}_{d,m}(\{s\}\times E)
=
\int_E
{P}(s,\mathbf{Z}^N(\omega))
\,\text{d}\mathbb{P}_m(\omega)
=
\int_E
\mathbb{P}_d(\{s\},\omega)
\,\text{d}\mathbb{P}_m(\omega).
\]
When taking expectations or computing probabilities,
we will emphasize whether this is with respect either to the measure
$\mathbb{P}_{d,m}$ associated with the product space $(\mathcal{S}_N\times\Omega,\mathfrak{A}_N\times \mathfrak{F})$,
or the measure $\mathbb{P}_d$ associated with the design space $(\mathcal{S}_N,\mathfrak{A}_N)$,
or the measure $\mathbb{P}_m$ associated with the super-population space $(\Omega,\mathfrak{F})$.

If $n_s$ denotes the size of sample $s$,
then this may depend on the specific sampling design including the values of
the design variables $Z_1(\omega),\ldots,Z_N(\omega)$.
Similarly, the inclusion probabilities may depend on the values of the design variables,
$\pi_i(\omega)=\mathbb{E}_d(\xi_i,\omega)=\sum_{s\ni i}{P}\left(s,\mathbf{Z}^N(\omega)\right)$,
where $\xi_i$ is the indicator $\mathds{1}_{\{s\ni i\}}$.
Instead of $n_s$, we will consider $n=\mathbb{E}_d[n_s(\omega)]=\sum_{i=1}^N \mathbb{E}_d(\xi_i,\omega)=\sum_{i=1}^N \pi_i(\omega)$.
This means that
the inclusion probabilities and the design-based expected sample size
may be random variables on $(\Omega,\mathfrak{F},\mathbb{P}_m)$.
{For instance~\cite{Bertail_2013} considers $\pi_i=\pi(Z_i)$, where the pairs $(Y_i,Z_i)$ are assumed to be i.i.d.~random vectors on $\Omega$,
and~\cite{conti2015} considers $\pi_i=nh(Z_i)/\sum_{j=1}^Nh(Z_j)$, for some positive function~$h$.}

We first consider the Horvitz-Thompson (HT) empirical processes,
obtained from the HT empirical c.d.f.:
\begin{equation}
\label{eq:HT cdf}
\FHT(t)
=
\frac{1}{N}
\sum_{i=1}^N
\frac{\xi_i\mathds{1}_{\{Y_i\leq t\}}}{\pi_i},
\quad
t\in \mathbb{R}.
\end{equation}
We will consider {the} HT empirical process $\sqrt{n}(\FHT-\FN)$, obtained by centering around the empirical c.d.f.~$\FN$ of $Y_1,\ldots,Y_N$,
as well as the HT empirical process $\sqrt{n}(\FHT-F)$, obtained by centering around the c.d.f.~$F$ of the $Y_i$'s.
A functional central limit theorem for both processes will be formulated in Section~\ref{sec:HT}.
In addition, we will consider the H\'ajek empirical c.d.f.:
\begin{equation}
\label{eq:Hajek cdf}
\FHJ(t)
=
\frac{1}{\widehat{N}}
\sum_{i=1}^N
\frac{\xi_i\mathds{1}_{\{Y_i\leq t\}}}{\pi_i},
\quad
t\in \mathbb{R},
\end{equation}
where $\widehat{N}=\sum_{i=1}^N \xi_i/\pi_i$ is the HT estimator for the population total $N$.
Functional central limit theorems for $\sqrt{n}(\FHJ-\FN)$ and $\sqrt{n}(\FHJ-F)$ will be provided in Section~\ref{sec:Hajek}.
The advantage of our results is that they allow general {{single-stage} unequal probability} sampling schemes
and that we primarily require bounds on the rate at which higher order correlations
tend to zero $\omega$-almost surely, under the design measure $\mathbb{P}_d$.
\section{FCLT's for the Horvitz-Thompson empirical processes}
\label{sec:HT}
A functional central limit theorem for
$\sqrt{n}(\FHT-\FN)$ and $\sqrt{n}(\FHT-F)$ is obtained
by proving weak convergence of all finite dimensional distributions and tightness.
In order to establish the latter for general {{single-stage} unequal probability} sampling schemes, we impose a number of conditions that involve the sets
\begin{equation}
\label{eq:def DtN}
\begin{split}
D_{\nu,N}=
\Big\{
(i_1,i_2,\ldots,i_\nu)\in \{1, 2, \dots, N\}^\nu:
&\,
i_1,i_2,\ldots,i_\nu\text{ all different}
\Big\},
\end{split}
\end{equation}
for the integers $1\leq \nu\leq 4$.
We assume the following conditions:
%\begin{quote}
\begin{itemize}
  \item[(C1)]
  there exist constants $K_1,K_2$, such that for all $i=1,2,\ldots,N$,
  \[
  0<K_1\leq \frac{N\pi_i}n\leq K_2<\infty,
  \quad
  \omega-\text{a.s.}
  \]
\end{itemize}
{
The upper bound in~(C1), which expresses the fact that the~$\pi_i$ may not be too large, is related to convergence of $n/N$.
The reason is that $N\pi_i/n\leq N/n$, so that an upper bound on~$N\pi_i/n$ is immediate if one requires $n/N\to\lambda>0$.
This {last} condition is imposed by many authors, e.g., see~\cite{Bertail_2013}, \cite{BreidtOpsomer_2000}, \cite{conti2014}, \cite{conti2015},
among others.
The upper bound in our condition~(C1) enables us to allow $n/N\to0$.
The lower bound in~(C1) expresses the fact that $\pi_i$ may not be too small.
Sometimes this is taken care of by imposing $\pi_i\geq \pi^*>0$, see for instance~\cite{Bertail_2013}, \cite{BreidtOpsomer_2000}.
It can be seen that conditions A3-A4 in~\cite{conti2015} imply the lower bound in~(C1).
Details~\label{page:C1} can be found in the supplement B~\cite{Boistard_2015}.
}

There exists a constant $K_3>0$, such that for all $N=1,2,\ldots$:
\begin{itemize}
  \item[(C2)]
  $\max_{(i,j)\in D_{2,N}} \Big|\mathbb{E}_d(\xi_i-\pi_i)(\xi_j-\pi_j)\Big|<K_3n/N^2$,
  \item[(C3)]
  $\max_{(i,j,k)\in D_{3,N}} \Big|\mathbb{E}_d(\xi_i-\pi_i)(\xi_j-\pi_j)(\xi_k-\pi_k)\Big|<K_3n^2/N^3$,
  \item[(C4)]
  $\max_{(i,j,k,l)\in D_{4,N}} \Big|\mathbb{E}_d(\xi_i-\pi_i)(\xi_j-\pi_j)(\xi_k-\pi_k)(\xi_l-\pi_l)\Big|<K_3n^2/N^4$,
\end{itemize}
$\omega$-almost surely.
%\end{quote}
These conditions on higher order correlations are commonly used in the literature on survey sampling in order to derive asymptotic properties of
estimators (e.g.,
see~\cite{BreidtOpsomer_2000},
and~\cite{Cardot_2010}).
\cite{BreidtOpsomer_2000} proved that they hold for simple random sampling
without replacement and stratified simple random sampling without replacement,
whereas~\cite{Boistard_2012} proved that they hold also for rejective sampling.
Lemma~2 from~\cite{Boistard_2012} allows us to reformulate the above conditions on higher order correlations
into conditions on higher order inclusion probabilities.

{Conditions~(C2)-(C4) are primarily used to establish tightness of the random processes involved.
{These conditions have been formulated as such, because they are compactly expressed in terms of higher order correlations.}
Nevertheless, as one of the referees pointed out, bounds on maximum correlations may be somewhat restrictive,
and bounds on the average correlation are perhaps more desirable.
For fixed size sampling designs with inclusion probabilities not depending on $\omega$, this can be accomplished
by adapting the tightness proof, see Section~\ref{subsec:fixed designs}.
Conditions (C2)-(C4) can be simplified enormously when we consider the class of high entropy sampling designs, see~\cite{berger1998a,berger1998b,conti2014,conti2015}.
In this case, conditions on the rate at which $\sum_{i=1}^N\pi_i(1-\pi_i)$ tends to infinity compared to~$N$ and~$n$ are sufficient
for (C2)-(C4), see Section~\ref{subsec:high entropy}.
}

To establish the convergence of finite dimensional distributions, for sequences of bounded i.i.d.~random variables $V_1,V_2,\ldots$ on $(\Omega,\mathfrak{\mathfrak{F}},\mathbb{P}_m)$,
we will need a CLT for the HT estimator in the design space, conditionally on the~$V_i$'s.
To this end, let $S_N^2$ be the (design-based) variance of the HT estimator of the population mean, i.e.,
\begin{equation}
\label{def:variance HT}
S_N^2
=
\frac{1}{N^2}
\sum_{i=1}^N\sum_{j=1}^N
\frac{\pi_{ij}-\pi_i\pi_j}{\pi_i\pi_j}
V_iV_j.
\end{equation}
We assume that
%\begin{quote}
\begin{itemize}
\item[(HT1)]
{Let $V_1,V_2,\ldots$ be a sequence of bounded i.i.d.~random variables, not identical to zero, and
such there exists an $M>0$, such that $|V_i|\leq M$ $\omega$-almost surely, for all $i=1,2,\ldots$.
Suppose that for $N$ sufficiently large, $S_N>0$ and}
\[
\frac1{S_N}
%\sqrt{n}
\left(
\frac{1}{N}
\sum_{i=1}^N
\frac{\xi_iV_i}{\pi_i}
-
\frac{1}{N}
\sum_{i=1}^N
V_i
\right)
%\stackrel{d}{\to}
\to
N(0,1),
\qquad
\omega-\text{a.s.},
\]
in distribution under $\mathbb{P}_d$.
\end{itemize}
Note that (HT1) holds for simple random sampling without replacement
if $n(N-n)/N$ tends to infinity when $N$ tends to infinity (see \cite{Thompson_1997}),
as well as for Poisson sampling under some conditions on the first order inclusion probabilities (e.g., see~\cite{Fuller_2009}).
{
For rejective sampling, \cite{Hajek_1964} gives a somewhat technical condition that is sufficient and necessary for (HT1).
Other references are~\cite{visek1979}, \cite{praskovasen2009}, among others.
In~\cite{berger1998b} the CLT is extended to high entropy sampling designs.
For this class of sampling designs, simple conditions can be formulated that are sufficient for (HT1),
see Proposition~\ref{prop:HT1} in Section~\ref{subsec:high entropy}.
}

We also need that $nS^2_N$ converges for the particular case where the $V_i$'s are random vectors
consisting of indicators $\mathds{1}_{\{Y_j\leq t\}}$.
\begin{itemize}
\item[(HT2)]
For $k\in\{1,2,\ldots\}$, $i=1,2,\ldots,k$ and $t_1,t_2,\ldots,t_k\in \mathbb{R}$,
define
$\mathbf{Y}_{ik}^t=\left(\mathds{1}_{\{Y_i\leq t_1\}},\ldots,\mathds{1}_{\{Y_i\leq t_k\}}\right)$.
There exists a deterministic matrix~$\mathbf{\Sigma}_k^{\mathrm{HT}}$, such that
\begin{equation}
\label{eq:HT2 alternative}
\lim_{N\to\infty}
\frac{n}{N^2}
\sum_{i=1}^N\sum_{j=1}^N
\frac{\pi_{ij}-\pi_i\pi_j}{\pi_i\pi_j}
\mathbf{Y}_{ik}\mathbf{Y}_{jk}^t
=
\mathbf{\Sigma}_k^{\mathrm{HT}},
\quad
\omega-\text{a.s.}
\end{equation}
\end{itemize}
This kind of assumption is quite standard in the literature on survey sampling
and is usually imposed for general random vectors (see, for example
\cite{Deville_1992}, p.379,
\cite{Francisco_1991}, condition~3 on page 457,
or~\cite{Krewski_1981}, condition C4 on page 1014).
It suffices to require~\eqref{eq:HT2 alternative} for $\mathbf{Y}_{ik}^t=\left(\mathds{1}_{\{Y_i\leq t_1\}},\ldots,\mathds{1}_{\{Y_i\leq t_k\}}\right)$.
Moreover, if (C1)-(C2) hold, then the sequence in~\eqref{eq:HT2 alternative} is bounded,
so that by dominated convergence it follows that
\begin{equation}
\label{eq:HT2 alternative moment}
\mathbf{\Sigma}_k^{\mathrm{HT}}
=
\lim_{N\to\infty}
\frac{1}{N^2}
\sum_{i=1}^N\sum_{j=1}^N
\mathbb{E}_m
\left[
n
\frac{\pi_{ij}-\pi_i\pi_j}{\pi_i\pi_j}
\mathbf{Y}_{ik}\mathbf{Y}_{jk}^t
\right].
\end{equation}
This might help to get a more tractable expression for $\mathbf{\Sigma}_k^{\mathrm{HT}}$.

We are now able to formulate our first main result.
Let $D(\mathbb{R})$ be the space of c\`adl\`ag functions on $\mathbb{R}$ equipped with the Skorohod topology.
\begin{theorem}
\label{th:FCLT HT FN}
Let $Y_1,\ldots,Y_N$ be i.i.d.~random variables with c.d.f.~$F$ and empirical c.d.f.~$\FN$
and let $\FHT$ be defined in~\eqref{eq:HT cdf}.
Suppose that conditions (C1)-(C4) and (HT1)-(HT2) hold.
Then $\sqrt{n}(\FHT-\FN)$ converges weakly in $D(\mathbb{R})$ to
a mean zero Gaussian process $\mathbb{G}^{\mathrm{HT}}$ with covariance function
\[
\mathbb{E}_{m}\mathbb{G}^{\mathrm{HT}}(s)\mathbb{G}^{\mathrm{HT}}(t)
=
\lim_{N\to\infty}
\frac{1}{N^2}
\sum_{i=1}^N\sum_{j=1}^N
\mathbb{E}_{m}
\left[
n
\frac{\pi_{ij}-\pi_i\pi_j}{\pi_i\pi_j}
\mathds{1}_{\{Y_i\leq s\}}\mathds{1}_{\{Y_j\leq t\}}
\right]
\]
for $s,t\in \mathbb{R}$.
\end{theorem}
Note that Theorem~\ref{th:FCLT HT FN} allows a random (design-based) expected sample size $n$ and random inclusion probabilities.
The expression of the covariance function of the limiting Gaussian process is somewhat unsatisfactory.
When $n$ and the inclusion probabilities are deterministic, we can obtain a functional CLT with
a more precise expression
for $\mathbb{E}_{m}\mathbb{G}^{\mathrm{HT}}(s)\mathbb{G}^{\mathrm{HT}}(t)$ under slightly weaker conditions.
This is formulated in the proposition below.
Note that with imposing conditions (i)-(ii) in Proposition~\ref{prop:FCLT HT FN} instead of~\eqref{eq:HT2 alternative},
convergence of $nS_N^2$ is not necessarily guaranteed.
However, this is established in Lemma~\ref{lem:variance HT}
in~\cite{Boistard_2015} under (C1) and (C2).
Finally, we like to emphasize that if we would have imposed (HT2) for \emph{any} sequence $\mathbf{Y}_1,\mathbf{Y}_2,\ldots$ of bounded random vectors,
then (HT2) would have implied conditions (i)-(ii) in the deterministic setup of Proposition~\ref{prop:FCLT HT FN}.
\begin{proposition}
\label{prop:FCLT HT FN}
Consider the setting of Theorem~\ref{th:FCLT HT FN}, where $n$ and $\pi_i,\pi_{ij}$, for $i,j=1,2,\ldots,N$, are deterministic.
Suppose that (C1)-(C4) and~(HT1) hold, but instead of (HT2) assume that
there exist constants $\mu_{\pi1}$, $\mu_{\pi2} \in \mathbb{R}$ such that
\begin{itemize}
\item[(i)]
$\displaystyle
\lim_{N\to\infty}
\frac{n}{N^2}
\sum_{i=1}^N
\left(
\frac{1}{\pi_i}-1
\right)
=
\mu_{\pi1}$,
\item[(ii)]
$\displaystyle
\lim_{N\to\infty}
\frac{n}{N^2}
\ssum_{i\ne j}
\frac{\pi_{ij}-\pi_i\pi_j}{\pi_i\pi_j}
=
\mu_{\pi2}$.
\end{itemize}
Then $\sqrt{n}(\FHT-\FN)$ converges weakly in $D(\mathbb{R})$ to
a mean zero Gaussian process $\mathbb{G}^{\mathrm{HT}}$ with covariance function
$%\mathbb{E}_{m}\mathbb{G}^{\mathrm{HT}}(s)\mathbb{G}^{\mathrm{HT}}(t)=
\mu_{\pi_1}F(s\wedge t)+\mu_{\pi2}F(s)F(t)$,
for
$s,t\in \mathbb{R}$.
\end{proposition}
{Conditions (i)-(ii) ensure that $nS_N^2$ converges to a finite limit (see Lemma~\ref{lem:variance HT} in~\cite{Boistard_2015}),
from which the limiting covariance structure in Proposition~\ref{prop:FCLT HT FN} can be derived.
Condition~(i) also appears in~\cite{conti2014}.
Conditions similar to~(ii) appear in~\cite{isakifuller1982}, \cite{bergerskinner2005}, and~\cite{escobarberger2013}.}
When $n/N\to\lambda\in[0,1]$, then conditions~(i)-(ii) hold with $\mu_{\pi1}=1-\lambda$ and
$\mu_{\pi2}=\lambda-1$ for simple random sampling without replacement.
For Poisson sampling, (ii) holds trivially because the trials are independent.
For rejective sampling, (i)-(ii) together with $n/N\to\lambda\in[0,1]$,
can be deduced from the associated Poisson sampling design.
Indeed, suppose that (i) holds for Poisson sampling with first order inclusion probabilities $p_1,\ldots,p_N$,
such that $\sum_{i=1}^Np_i=n$.
Then, from Theorem~1 in~\cite{Boistard_2012} it follows that if $d=\sum_{i=1}^N p_i(1-p_i)$ tends to infinity,
assumption (i) holds for rejective sampling.
Furthermore, if $n/N\to\lambda\in[0,1]$ and $N/d$ has a finite limit, then also (ii) holds for rejective sampling.

Weak convergence of the process $\sqrt{n}(\FHT-F)$, where we center with $F$ instead of $\FN$,
requires a CLT in the super-population space for
\begin{equation}
\label{eq:HT estimator standardized}
\sqrt{n}
\left(
\frac1N
\sum_{i=1}^N
\frac{\xi_iV_i}{\pi_i}
-
\mu_V
\right),
\qquad
\text{where }\mu_V=\mathbb{E}_m(V_i),
\end{equation}
for sequences of bounded i.i.d.~random variables $V_1,V_2,\ldots$ on $(\Omega,\mathfrak{\mathfrak{F}},\mathbb{P}_m)$.
Our approach to establish asymptotic normality of~\eqref{eq:HT estimator standardized} is then to decompose as follows
\begin{equation}
\label{eq:decomposition CLT}
\begin{split}
&
\sqrt{n}
\left(
\frac{1}{N}
\sum_{i=1}^N
\frac{\xi_iV_i}{\pi_i}
-
\mu_{V}
\right)\\
&=
\sqrt{n}
\left(
\frac{1}{N}
\sum_{i=1}^N
\frac{\xi_iV_i}{\pi_i}
-
\frac{1}{N}
\sum_{i=1}^N
V_i
\right)
+
\frac{\sqrt{n}}{\sqrt{N}}
\times
\sqrt{N}
\left(
\frac{1}{N}
\sum_{i=1}^N
V_i
-
\mu_{V}
\right).
\end{split}
\end{equation}
Since the $V_i$'s are i.i.d.~and bounded, for the second term on the right hand side, by the traditional CLT we immediately obtain
\begin{equation}
\label{eq:CLT}
\sqrt{N}
\left(
\frac{1}{N}
\sum_{i=1}^N
V_i
-
\mu_{V}
\right)
%\stackrel{\mathcal{D}_m}{\to}%
\to
N(0,\sigma^2_V),
\end{equation}
in distribution under $\mathbb{P}_m$,
where $\sigma_V^2$ denotes the variance of the $V_i$'s,
whereas the first term on the right hand side can be handled with (HT1).
\cite{BreslowWellner_2007} and~\cite{Saegusa_2013} use a decomposition similar to the one in~\eqref{eq:decomposition CLT}.
Their approach assumes exchangeable $\xi_i$'s and equal inclusion probabilities $n/N$, which allows the use of results
on exchangeable weighted bootstrap to handle the first term on the right hand side of~\eqref{eq:decomposition CLT}.
Instead, we only require conditions (C2)-(C4) on higher order correlations for the $\xi_i$'s and allow the $\pi_i$'s
to vary within certain bounds as described in~(C1).
To combine the two separate limits in~\eqref{eq:CLT} and~(HT1), we will need
\begin{itemize}
\item[(HT3)]
$n/N\to\lambda\in[0,1]$, $\omega$-a.s.
\end{itemize}
{
One often assumes $\lambda\in(0,1)$ (e.g., see~\cite{Bertail_2013}, \cite{BreidtOpsomer_2000}, \cite{conti2014}, \cite{conti2015},
among others).
We like to emphasize that convergence of $n/N$ was not needed so far in our setup, because condition (C1) is used to control terms $1/\pi_i$.
To determine the precise limit for~\eqref{eq:decomposition CLT} we do need (HT3), but we allow $\lambda=0$ or $\lambda=1$.}

{Next, we will use} Theorem~5.1(iii) from~\cite{Rubin-Bleuer_Kratina_2005}.
The finite dimensional projections of the processes involved turn out to be related to a particular HT estimator.
In order to have the corresponding design-based variance converging to a strictly positive constant, we need the following condition.
\begin{itemize}
\item[(HT4)]
For all $k\in \{1,2,\ldots\}$ and $t_1,t_2,\ldots,t_k\in \mathbb{R}$,
the matrix $\mathbf{\Sigma}_k^{\mathrm{HT}}$ in~\eqref{eq:HT2 alternative} is positive definite.
\end{itemize}
We are now able to formulate our second main result.
\begin{theorem}
\label{th:FCLT HT F}
Let $Y_1,\ldots,Y_N$ be i.i.d.~random variables met c.d.f.~$F$ and let $\FHT$ be defined in~\eqref{eq:HT cdf}.
Suppose that conditions (C1)-(C4) and (HT1)-(HT4) hold.
Then $\sqrt{n}(\FHT-F)$ converges weakly in $D(\mathbb{R})$ to
a mean zero Gaussian process $\mathbb{G}_F^{\mathrm{HT}}$ with covariance function
$\mathbb{E}_{d,m}\mathbb{G}_F^{\mathrm{HT}}(s)\mathbb{G}_F^{\mathrm{HT}}(t)$ given by
\[
\lim_{N\to\infty}
\frac{1}{N^2}
\sum_{i=1}^N\sum_{j=1}^N
\mathbb{E}_{m}
\left[
n
\frac{\pi_{ij}-\pi_i\pi_j}{\pi_i\pi_j}
\mathds{1}_{\{Y_i\leq s\}}\mathds{1}_{\{Y_j\leq t\}}
\right]
+
\lambda
\big\{
F(s\wedge t)
-
F(s)F(t)
\big\},
\]
for $s,t\in \mathbb{R}$.
\end{theorem}
Theorem~\ref{th:FCLT HT F} allows random $n$ and inclusion probabilities.

As before, when the sample size $n$ and inclusion probabilities are deterministic
we can obtain a functional CLT under a simpler condition than (HT4) and with a more detailed description of the covariance function
of the limiting process.
\begin{proposition}
\label{prop:FCLT HT F}
Consider the setting of Theorem~\ref{th:FCLT HT F}, where $n$ and $\pi_i,\pi_{ij}$, for $i,j=1,2,\ldots,N$, are deterministic.
Suppose that (C1)-(C4), (HT1)and~(HT3) hold, but instead of (HT2) and (HT4) assume that
there exist constants $\mu_{\pi1}$, $\mu_{\pi2} \in \mathbb{R}$ such that
\begin{itemize}
\item[(i)]
$\displaystyle
\lim_{N\to\infty}
\frac{n}{N^2}
\sum_{i=1}^N
\left(
\frac{1}{\pi_i}-1
\right)
=
\mu_{\pi1}>0$,
\item[(ii)]
$\displaystyle
\lim_{N\to\infty}
\frac{n}{N^2}
\ssum_{i\ne j}
\frac{\pi_{ij}-\pi_i\pi_j}{\pi_i\pi_j}
=
\mu_{\pi2}$.
\end{itemize}
Then $\sqrt{n}(\FHT-F)$ converges weakly in $D(\mathbb{R})$ to
a mean zero Gaussian process $\mathbb{G}^{\mathrm{HT}}$ with covariance function
$%\mathbb{E}_{d,m}\mathbb{G}_F^{\mathrm{HT}}(s)\mathbb{G}_F^{\mathrm{HT}}(t)=
(\mu_{\pi_1}+\lambda)F(s\wedge t)+(\mu_{\pi2}-\lambda)F(s)F(t)$,
for $s,t\in \mathbb{R}$.
\end{proposition}
Since $1/\pi_i\geq 1$, we will always have $\mu_{\pi1}\geq0$ in condition~(i) in Proposition~\ref{prop:FCLT HT F}.
This means that (i) is not very restrictive.
For simple random sampling without replacement, condition (i) requires $\lambda$ to be strictly smaller than one.

{
\begin{remark}[High entropy designs]
\label{rem:high entropy}
Theorems~\ref{th:FCLT HT FN} and~\ref{th:FCLT HT F} include high entropy sampling designs with random inclusion probabilities,
which are considered for instance in~\cite{Bertail_2013} and~\cite{conti2015},
whereas Propositions~\ref{prop:FCLT HT FN} and~\ref{prop:FCLT HT F} include high entropy designs
with deterministic inclusion probabilities, for instance considered in~\cite{conti2014}.
For such designs, the conditions can be simplified considerably, in particular (C2)-(C4),
see Corollary~\ref{cor:high entropy}(i)-(ii) and Corollary~\ref{cor:high entropy deterministic}(i)-(ii) in Section~\ref{subsec:high entropy}.
\end{remark}
}

\section{FCLT's for the H\'ajek empirical processes}
\label{sec:Hajek}
To determine the behavior of the process $\sqrt{n}(\FHJ-\FN)$, it is useful to relate {it}
to the process
\begin{equation}
\label{eq:HT Lumley}
\GL(t)
=
\frac{\sqrt{n}}{N}
\sum_{i=1}^N
\frac{\xi_i}{\pi_i}
\Big(
\mathds{1}_{\{Y_i\leq t\}}-F(t)
\Big).
\end{equation}
We can then write
\begin{equation}
\label{eq:decompose HJ EP}
\sqrt{n}
\left\{
\FHJ(t)-\FN(t)
\right\}
=
\mathbb{Y}_N(t)
+
\left(
\frac{N}{\widehat{N}}-1
\right)\GL(t),
\end{equation}
where
\begin{equation}
\label{def:Y_N}
\mathbb{Y}_N(t)
=
\frac{\sqrt{n}}{N}
\sum_{i=1}^N
\left(
\frac{\xi_i}{\pi_i}-1
\right)
\left(
\mathds{1}_{\{Y_i\leq t\}}-F(t)
\right).
\end{equation}
As intermediate results we will first show that the process $\GL$ converges weakly
to a mean zero Gaussian process
and that $\widehat{N}/N\to 1$ in probability.
As a consequence, the limiting behavior of $\sqrt{n}(\FHJ-\FN)$ will be the same as that of $\mathbb{Y}_N$,
which is an easier process to handle.
Instead of (HT2) and~(HT4) we now need
\begin{itemize}
\item[(HJ2)]
For $k\in\{1,2,\ldots\}$, $i=1,2,\ldots,k$ and $t_1,t_2,\ldots,t_k\in \mathbb{R}$,
define
$\widetilde{\mathbf{Y}}_{ik}^t=\left(\mathds{1}_{\{Y_i\leq t_1\}}-F(t_1),\ldots,\mathds{1}_{\{Y_i\leq t_k\}}-F(t_k)\right)$.
There exists a deterministic matrix~$\mathbf{\Sigma}_k^{\mathrm{HJ}}$, such that
\begin{equation}
\label{eq:HJ2 alternative}
\lim_{N\to\infty}
\frac{n}{N^2}
\sum_{i=1}^N\sum_{j=1}^N
\frac{\pi_{ij}-\pi_i\pi_j}{\pi_i\pi_j}
\widetilde{\mathbf{Y}}_{ik}\widetilde{\mathbf{Y}}_{jk}^t
=
\mathbf{\Sigma}_k^{\mathrm{HJ}},
\quad
\omega-\text{a.s.}
\end{equation}
\end{itemize}
and
\begin{itemize}
\item[(HJ4)]
For all $k\in \{1,2,\ldots\}$ and $t_1,t_2,\ldots,t_k\in \mathbb{R}$,
the matrix $\mathbf{\Sigma}_k^{\mathrm{HJ}}$ in~\eqref{eq:HJ2 alternative} is positive definite.
\end{itemize}
As in the case of~\eqref{eq:HT2 alternative moment}, if (C1)-(C2) hold, then (HJ2) implies
\begin{equation}
\label{eq:HJ2 alternative moment}
\mathbf{\Sigma}_k^{\mathrm{HJ}}
=
\lim_{N\to\infty}
\frac{1}{N^2}
\sum_{i=1}^N\sum_{j=1}^N
\mathbb{E}_m
\left[n
\frac{\pi_{ij}-\pi_i\pi_j}{\pi_i\pi_j}
\widetilde{\mathbf{Y}}_{ik}\widetilde{\mathbf{Y}}_{jk}^t
\right].
\end{equation}
\begin{theorem}
\label{th:FCLT Lumley}
Let $\GL$ be defined in~\eqref{eq:HT Lumley} and let $\widehat{N}=\sum_{i=1}^N \xi_i/\pi_i$.
Suppose $n\to\infty$, $\omega$-a.s., and that there exists $\sigma_{\pi}^2\geq 0$, such that
\begin{equation}
\label{eq:cond CLT Nhat}
\frac{n}{N^2}
\sum_{i=1}^N\sum_{i=1}^N
\frac{\pi_{ij}-\pi_i\pi_j}{\pi_i\pi_j}
\to
\sigma_{\pi}^2,
\quad
\omega-\text{a.s.}
\end{equation}
If in addition,
\begin{itemize}
\item[(i)]
(HT1) hold, then $\widehat{N}/N\to 1$ in $\mathbb{P}_{d,m}$ probability.
\item[(ii)]
(C1)-C(4), (HT1), (HT3), (HJ2) and (HJ4) hold, then~$\GL$ converges weakly in $D(\mathbb{R})$ to a
mean zero Gaussian process~$\mathbb{G}^{\mathrm{\pi}}$ with covariance function
$\mathbb{E}_{d,m}\mathbb{G}^{\mathrm{\pi}}(s)\mathbb{G}^{\mathrm{\pi}}(t)$
given by
\[
\begin{split}
\lim_{N\to\infty}
\frac{1}{N^2}
\sum_{i=1}^N\sum_{j=1}^N
\mathbb{E}_m
&\left[n
\frac{\pi_{ij}-\pi_i\pi_j}{\pi_i\pi_j}
\left(\mathds{1}_{\{Y_i\leq s\}}-F(s)\right)
\left(\mathds{1}_{\{Y_i\leq t\}}-F(t)\right)
\right]\\
&+
\lambda
\left(
F(s\wedge t)-F(s)F(t)
\right),
\quad
s,t\in \mathbb{R}.
\end{split}
\]
\end{itemize}
\end{theorem}
Note that in view of condition (HT3), the condition $n\to\infty$ is immediate, if $\lambda>0$.
\begin{comment}
Theorem~\eqref{th:FCLT Lumley}(ii) is similar to Theorem~3 in~\cite{Lumley_2014}.
The latter result holds more generally for empirical processes indexed by a class of functions.
However, the conditions imposed in~\cite{Lumley_2014} seem very difficult to verify,
whereas Theorem~\eqref{th:FCLT Lumley}(ii) is obtained under relatively easy to check assumptions on inclusion probabilities,
especially when these are deterministic.
\end{comment}
We proceed by establishing weak convergence of $\sqrt{n}(\FHJ-\FN)$.
\begin{theorem}
\label{th:FCLT HJ FN}
Let $Y_1,\ldots,Y_N$ be i.i.d.~random variables with c.d.f.~$F$ and empirical c.d.f.~$\FN$ and
let $\FHJ$ be defined in~\eqref{eq:Hajek cdf}.
Suppose $n\to\infty$, $\omega$-a.s., and that (C1)-C(4), (HT1), (HT3), and (HJ2) hold,
as well as condition~\eqref{eq:cond CLT Nhat}.
Then $\sqrt{n}(\FHJ-\FN)$ converges weakly in $D(\mathbb{R})$ to
a mean zero Gaussian process~$\mathbb{G}^{\mathrm{HJ}}$ with covariance function
$\mathbb{E}_{d,m}\mathbb{G}^{\mathrm{HJ}}(s)\mathbb{G}^{\mathrm{HJ}}(t)$ given by
\[
\lim_{N\to\infty}
\frac{1}{N^2}
\sum_{i=1}^N\sum_{j=1}^N
\mathbb{E}_m
\left[n
\frac{\pi_{ij}-\pi_i\pi_j}{\pi_i\pi_j}
\left(\mathds{1}_{\{Y_i\leq s\}}-F(s)\right)
\left(\mathds{1}_{\{Y_i\leq t\}}-F(t)\right)
\right],
\]
for $s,t\in \mathbb{R}$.
\end{theorem}
Note that we do not need condition (HJ4) in Theorem~\ref{th:FCLT HJ FN}.
This condition is only needed in Theorem~\ref{th:FCLT Lumley} to establish the limit distribution of the finite dimensional projections
of the process $\GL$.
For Theorem~\ref{th:FCLT HJ FN} we only need that $\GL$ is tight.

As before, below we obtain a functional CLT for $\sqrt{n}(\FHJ-\FN)$ in the case that
$n$ and the inclusion probabilities are deterministic.
Similar to the remark we made after Theorem~\ref{th:FCLT HT FN}, note that if we would have imposed~(HJ2) for any sequence of bounded random vectors,
then this would imply conditions (i)-(ii) of Proposition~\ref{prop:FCLT HT FN}, which can then be left out in Theorem~\ref{th:FCLT Lumley}.
\begin{proposition}
\label{prop:FCLT HJ FN}
Consider the setting of Theorem~\ref{th:FCLT HJ FN}, where $n$ and $\pi_i,\pi_{ij}$, for $i,j=1,2,\ldots,N$, are deterministic.
Suppose $n\to\infty$ and that (C1)-(C4), (HT1) and~(HT3) hold, as well as conditions~(i)-(ii) from Proposition~\ref{prop:FCLT HT FN}.
Then $\sqrt{n}(\FHJ-\FN)$ converges weakly in $D(\mathbb{R})$ to
a mean zero Gaussian process $\mathbb{G}^{\mathrm{HT}}$ with covariance function
$%\mathbb{E}_{d,m}\mathbb{G}_F^{\mathrm{HT}}(s)\mathbb{G}_F^{\mathrm{HT}}(t)=
\mu_{\pi_1}
\left(
F(s\wedge t)-F(s)F(t)
\right)$,
for $s,t\in \mathbb{R}$.
\end{proposition}
Finally, we consider $\sqrt{n}(\FHJ-F)$.
Again, we relate this process to~\eqref{eq:HT Lumley} and write
\begin{equation}
\label{eq:relation GL and GHJ}
\begin{split}
\sqrt{n}\left(\FHJ(t)-F(t)\right)
&=
%\frac{N}{\widehat{N}}
%\frac{\sqrt{n}}N
%\left(
%\sum_{i=1}^N
%\frac{\xi_i\mathds{1}_{\{Y_i\leq t\}}}{\pi_i}
%-
%\sum_{i=1}^N
%\frac{\xi_i}{\pi_i}
%F(t)
%\right)\\
%&=
\frac{N}{\widehat{N}}
\GL(t).
\end{split}
\end{equation}
Since $\widehat{N}/N\to1$ in probability, this implies that
$\sqrt{n}(\FHJ-F)$ has the same limiting behavior as $\GL$.
\begin{theorem}
\label{th:FCLT HJ F}
Let $Y_1,\ldots,Y_N$ be i.i.d.~random variables with c.d.f.~$F$ and let $\FHJ$ be defined in~\eqref{eq:Hajek cdf}.
Suppose $n\to\infty$, $\omega$-a.s., and that (C1)-C(4), (HT1), (HT3), (HJ2) and (HJ4) hold,
as well as condition~\eqref{eq:cond CLT Nhat}.
Then $\sqrt{n}(\FHJ-F)$ converges weakly in~$D(\mathbb{R})$ to
a mean zero Gaussian process $\mathbb{G}^{\mathrm{HJ}}_F$ with covariance function
$\mathbb{E}_{d,m}\mathbb{G}^{\mathrm{\pi}}(s)\mathbb{G}^{\mathrm{\pi}}(t)$
given by
\[
\begin{split}
\lim_{N\to\infty}
\frac{1}{N^2}
\sum_{i=1}^N\sum_{j=1}^N
\mathbb{E}_m
&\left[n
\frac{\pi_{ij}-\pi_i\pi_j}{\pi_i\pi_j}
\left(\mathds{1}_{\{Y_i\leq s\}}-F(s)\right)
\left(\mathds{1}_{\{Y_i\leq t\}}-F(t)\right)
\right]\\
&+
\lambda
\left(
F(s\wedge t)-F(s)F(t)
\right),
\quad
s,t\in \mathbb{R}.
\end{split}
\]
\end{theorem}
With Theorem~\ref{th:FCLT HJ F} we recover Theorem~1 in~\cite{Wang_2012}.
Our assumptions are comparable to those in~\cite{Wang_2012}, although this paper seems to miss
a condition on the convergence of the variance, such as our condition~(HJ2).

We conclude this section by establishing a functional CLT for $\sqrt{n}(\FHJ-F)$ in the case
of deterministic $n$ and inclusion probabilities.
\begin{proposition}
\label{prop:FCLT HJ F}
Consider the setting of Theorem~\ref{th:FCLT HJ F}, where $n$ and $\pi_i,\pi_{ij}$, for $i,j=1,2,\ldots,N$, are deterministic.
Suppose $n\to\infty$ and that (C1)-(C4), (HT1) and~(HT3) hold, as well as conditions~(i)-(ii) from Proposition~\ref{prop:FCLT HT F}.
Then $\sqrt{n}(\FHJ-F)$ converges weakly in $D(\mathbb{R})$ to
a mean zero Gaussian process $\mathbb{G}^{\mathrm{HJ}}$ with covariance function
$%\mathbb{E}_{d,m}\mathbb{G}_F^{\mathrm{HJ}}(s)\mathbb{G}_F^{\mathrm{HJ}}(t)=
\left(
\mu_{\pi_1}+\lambda
\right)
\left(
F(s\wedge t)-F(s)F(t)
\right)$,
for $s,t\in \mathbb{R}$.
\end{proposition}

{
\begin{remark}[High entropy designs]
\label{rem:high entropy HJ}
Remark~\ref{rem:high entropy} about simplifying the conditions for the Horvitz-Thompson empirical process in the case of high entropy designs,
also holds for the H\'ajek empirical process.
See Corollary~\ref{cor:high entropy}(iii)-(iv) and Corollary~\ref{cor:high entropy deterministic}(iii)(iv)
in Section~\ref{subsec:high entropy}.
\end{remark}
}

{
\section{Examples}
\label{sec:examples}

\subsection{High entropy designs}
\label{subsec:high entropy}
For the sake of brevity, let us suppress the possible dependence of a sampling design on $\mathbf{Z}^N$
and write $P(\cdot)=P(\cdot,\mathbf{Z}^N)$.
The entropy of a sampling design $P$ is defined as
\[
H(P)=-\sum_{s\in \mathcal{S}_N} P(s)\Log[P(s)]
\]
where $\Log$ denotes the Napierian logarithm, and define $0\Log[0]=0$.
The entropy $H(P)$ represents the average amount of information contained in design $P$
(e.g., see~\cite{berger1998b}).
Given inclusion probabilities $\pi_1,\ldots,\pi_N$, the rejective sampling design,
denoted by $R$ (see~\cite{hajek1959,Hajek_1964}), is known to maximize the entropy
among all fixed size sampling designs subject to the constraint that the first order inclusion probabilities are equal to $\pi_1,\ldots, \pi_N$.
This sampling design is defined by
\[
R(s)=\theta\prod_{i\in s}\alpha_i,
\quad
\text{with }
\alpha_i=\eta \frac{p_i}{1-p_i}
\]
where $\theta$ is such that $\sum_{s\in \mathcal{S}_N}R(s)=1$, $\eta$ is such that $\sum_{i=1}^N\alpha_i=1$,
and the $0<p_i<1$ are such that $\sum_{i=1}^Np_i=n$ and are chosen to produce the first order inclusion probabilities $\pi_i$.
It is shown in~\cite{dupacova1979} that for any given set of inclusion probabilities $\pi_1,\ldots,\pi_N$, there always exists
a unique set of~$p_i$'s such that the first order inclusion probabilities corresponding to $R$ are exactly equal to the $\pi_i$'s.

An important class is formed by sampling designs $P$ that are close to a rejective sampling design~$R$.
Berger~\cite{berger1998b} considers such a class where the divergence of $P$ from $R$ is measured by
\begin{equation}
\label{def:divergence}
D(P\| R)=\sum_{s\in \mathcal{S}_N} P(s)\Log\left[\frac{P(s)}{R(s)}\right].
\end{equation}
In this subsection we will consider high entropy designs $P$, i.e., sampling designs $P$ for which
{there exists a rejective sampling design $R$ such that}
\begin{itemize}
\item[(A1)]
$D(P\| R)\to 0$, as $N\to\infty$.
\end{itemize}
A similar class is considered in~\cite{conti2014, conti2015}, where the Hellinger distance between~$P$ and~$R$
is used instead of~\eqref{def:divergence}.
Sampling designs satisfying~(A1) are investigated in~\cite{berger1998b}.
Examples are Rao-Sampford sampling and successive sampling, see Theorems~6 and~7 in~\cite{berger1998b}.

For high entropy designs $P$ satisfying~(A1), the conditions imposed in
Sections~\ref{sec:HT} and~\ref{sec:Hajek} can be simplified considerably.
Essentially, the results in these sections can be obtained by conditions on the rate at which
\begin{equation}
\label{def:dN}
d_N=\sum_{i=1}^N \pi_i(1-\pi_i)
\end{equation}
tends to infinity,
compared to $N$ and~$n$.
First of all condition (HT1) can be established under mild conditions.
\begin{proposition}
\label{prop:HT1}
Let $P$ be a high entropy design satisfying (A1) with inclusion probabilities $\pi_1,\ldots,\pi_N$.
Let $d_N$ and $S_N^2$ be defined by~\eqref{def:dN} and~\eqref{def:variance HT}.
Suppose that (C1) holds and that the following conditions hold $\omega$-almost surely
\begin{itemize}
\item[(B1)] $n/d_N=O(1)$, as $N\to\infty$;
\item[(B2)] $N/d_N^2\to0$, as $N\to\infty$;
\item[(B3)] $n^2S_N^2\to\infty$, as $N\to\infty$.
\end{itemize}
Then (HT1) is satisfied.
\end{proposition}
Conditions (B1)-(B2) are immediate, if $d_N/N\to d>0$ and $n/N\to\lambda>0$.
Moreover, $nS_N^2$ typically converges almost surely to some $\sigma^2\geq 0$,
so that~(B3) is immediate as soon as $\sigma^2>0$ and (B1) holds.

The following corollary covers the results from Sections~\ref{sec:HT} and~\ref{sec:Hajek} for
high entropy designs with inclusion probabilities that possibly depend on $\omega$.
Such designs are considered for instance in~\cite{Bertail_2013} and~\cite{conti2015}.
\begin{corollary}
\label{cor:high entropy}
Let $P$ be a high entropy design satisfying (A1) with inclusion probabilities $\pi_1,\ldots,\pi_N$,
and let $d_N$ be defined by~\eqref{def:dN}.
Suppose that conditions (C1) and (HT1) hold.
Furthermore, suppose that the following conditions hold $\omega$-almost surely:
\begin{itemize}
\item[(A2)] $d_N\to\infty$, as $N\to\infty$;
\item[(A3)] $n/d_N=O(1)$, as $N\to\infty$;
\item[(A4)] $N^2/(nd_N)=O(1)$, as $N\to\infty$.
\end{itemize}
Then
\begin{itemize}
\item[(i)] if (HT2) is satisfied, then the conclusion of Theorem~\ref{th:FCLT HT FN} holds;
\item[(ii)] if (HT2)-(HT4) are satisfied, then the conclusion of Theorem~\ref{th:FCLT HT F} holds.
\item[(iii)] if (HT3), (HJ2) are satisfied, and $\omega$-almost surely,
\begin{itemize}
\item[(A5)] $n(N-n)^2/(N^2d_N)\to \alpha$, as $N\to\infty$,
\end{itemize}
then the conclusion of Theorem~\ref{th:FCLT HJ FN} holds;
\item[(iv)]
if (HT3), (HJ2), (HJ4), and (A5) are satisfied, then the conclusion of Theorem~\ref{th:FCLT HJ F} holds.
\end{itemize}
\end{corollary}
As it turns out, for the particular setting of high entropy designs, conditions (A2)-(A4) together with (C1) are sufficient for (C2)-(C4),
whereas~(A5) implies condition~\eqref{eq:cond CLT Nhat}.
The conditions in Corollary~\ref{cor:high entropy} have been formulated as weakly as possible.
They are implied by the usual conditions that one finds in the literature.
For instance, when $N/d_N=O(1)$ (e.g., see~\cite{Boistard_2012}) and $n/N\to\lambda>0$, then (A2)-(A4) are immediate.
Part~(iii) in Corollary~\ref{cor:high entropy} is similar to Proposition~1 in~\cite{conti2015}, where the Hellinger distance between~$P$ and~$R$ is used instead of~\eqref{def:divergence}.
It can be seen that the conditions in~\cite{conti2015} are sufficient for our conditions (B1)-(B2) in Proposition~\ref{prop:HT1}, (C1), (A1)-(A5), (HT3) and the existence of the almost sure limits in (HT2) and~(HJ2).

Things become even easier when the high entropy design has inclusion probabilities that do not depend on~$\omega$.
\begin{corollary}
\label{cor:high entropy deterministic}
Let $P$ be a high entropy design satisfying (A1)-(A5) with deterministic inclusion probabilities $\pi_1,\ldots,\pi_N$.
Suppose that conditions~(C1), (HT1), and
$\lim_{N\to\infty}
(n/N^2)
\sum_{i=1}^N
\left(
\pi_i^{-1}-1
\right)
=
\mu_{\pi1}.
$
hold.
Then
\begin{itemize}
\item[(i)] the conclusion of Proposition~\ref{prop:FCLT HT FN} holds;
\item[(ii)] if (HT3) is satisfied and $\mu_{\pi1}>0$, then the conclusion of Proposition~\ref{prop:FCLT HT F} holds;
\item[(iii)] if (HT3) is satisfied, then the conclusion of Proposition~\ref{prop:FCLT HJ FN} holds;
\item[(iv)] if (HT3) is satisfied and $\mu_{\pi1}>0$, then the conclusion of Proposition~\ref{prop:FCLT HJ F} holds.
\end{itemize}
\end{corollary}
As before, conditions (A2)-(A4) together with (C1) are sufficient for (C2)-(C4),
whereas (A5) implies condition~(ii) of Propositions~\ref{prop:FCLT HT FN} and~\ref{prop:FCLT HT F}.
Part~(i) in Corollary~\ref{cor:high entropy deterministic} is similar to Proposition 1 in~\cite{conti2014},
where the Hellinger distance between $P$ and $R$ is used instead of~\eqref{def:divergence}.
It can be seen that the conditions in~\cite{conti2014} are sufficient for (B1)-(B2) in Proposition~\ref{prop:HT1}, (A1)-(A5), (HT3) and (i).

\subsection{Fixed size sampling designs with deterministic inclusion probabilities}
\label{subsec:fixed designs}
Conditions (C2)-(C4) put bounds on maximum correlations.
This is somewhat restrictive, and bounds on the average correlation may be more suitable for applications.
This can indeed be accomplished to some extent for fixed size sampling designs~$P$,
with inclusion probabilities $\pi_i$ that do not depend on~$\omega$.

Suppose there exists a $K>0$, such that for all $N=1,2,\ldots$,
\begin{itemize}
\item[(C2$^*$)]
for all $j=1,2,\ldots,N$:
$\displaystyle
\frac{n}{N}
\sum_{i\ne j}
\left|
\frac{\pi_{ij}-\pi_i\pi_j}{\pi_i\pi_j}
\right|
\leq K$,
\item[(C3$^*$)]
$\displaystyle\frac{n}{N^3}
\sssum_{(i,j,k)\in D_{3,N}}
\left|
\frac{\pi_{ijk}-\pi_i\pi_j\pi_k}{\pi_i\pi_j\pi_k}
\right|\leq K$.
\item[(C4$^*$)]
$
\displaystyle\frac{n^2}{N^4}
\left|
\sum_{(i,j,k,l)\in D_{4,N}}
\frac{\mathbb{E}_{d}\left[(\xi_i-\pi_i)(\xi_j-\pi_j)(\xi_k-\pi_k)(\xi_l-\pi_l)\right]}{\pi_i\pi_j\pi_k\pi_l}
\right|
\leq K$.
\end{itemize}
The summation in (C2$^*$) has a number of terms of the order $N$.
This means that typically the summands must decrease at rate $1/N$.
This is comparable to condition~(ii) in Proposition~\ref{prop:FCLT HT FN}.
Similarly for summands in the summation in~(C3$^*$).
The summands in (C4$^*$) have to overcome a factor of the order~$N^2$,
which will typically not be the case for general sampling designs.
However, according to Lemma~2 in~\cite{Boistard_2012}, the fourth order correlation
can be decomposed in terms of the type
\[
(-1)^{4-m}\frac{\pi_{i_1\cdots i_m}-\pi_{i_1}\cdots \pi_{i_m}}{\pi_{i_1}\cdots \pi_{i_m}},
\quad
m=2,3,4.
\]
Because these terms can be both negative and positive,
they may cancel each other in such a way that~(C4$^*$) does hold.
This is for instance the case for simple random sampling, e.g., see the discussion in Remarks~(iii) and~(iv) in~\cite{BreidtOpsomer_2000},
or for rejective sampling, see Proposition 1 in~\cite{Boistard_2012}.

By using Lemma 2 in~\cite{Boistard_2012} it follows that conditions (C2$^*$)-(C4$^*$)
are implied by (C2)-(C4).
The following corollary covers the results from Sections~\ref{sec:HT} and~\ref{sec:Hajek} under the weaker conditions (C2$^*$)-(C4$^*$),
for fixed size sampling designs with deterministic inclusion probabilities.

\begin{corollary}
\label{cor:fixed size}
Let $P$ be a fixed size sampling design with deterministic inclusion probabilities.
Suppose that (C1), (C2$^*$)-(C4$^*$), (HT1), hold, as well as conditions~(i) and~(ii) from Proposition~\ref{prop:FCLT HT FN}.
Then
\begin{itemize}
\item[(i)]
the conclusion of Proposition~\ref{prop:FCLT HT FN} holds;
\item[(ii)]
if (HT3) is satisfied and $\mu_{\pi1}>0$, then the conclusion of Proposition~\ref{prop:FCLT HT F} holds;
\item[(iii)]
if (HT3) is satisfied, then the conclusions of Propositions~\ref{prop:FCLT HJ FN} and~\ref{prop:FCLT HJ F} hold.
\end{itemize}
\end{corollary}
}

\section{Hadamard-differentiable functionals}
\label{sec:hadamard}
Theorem~\ref{th:FCLT HJ F} provides an elegant means to study
the limit behavior of estimators that can be described as $\phi(\FHJ)$,
where $\phi$ is a Hadamard-differentiable functional.
Given such a $\phi$, the functional delta-method, e.g., see Theorems~3.9.4 and~3.9.5 in~\cite{van_1996}
or Theorem~20.8 in~\cite{vandervaart1998},
enables one to establish the limit distribution of~$\phi(\FHJ)$.
Similarly, this holds for Theorems~\ref{th:FCLT HT FN}, \ref{th:FCLT HT F}, and~\ref{th:FCLT HJ FN},
or Propositions~\ref{prop:FCLT HT FN}, \ref{prop:FCLT HT F}, \ref{prop:FCLT HJ FN}, and~\ref{prop:FCLT HJ F}
in the special case of deterministic $n$  and inclusion probabilities.

We illustrate this by discussing the poverty rate.
This indicator has recently been revisited by~\cite{graftille2014} and~\cite{oguzalperberger2015}.
This example has also been discussed by~\cite{dell2008},
but under the assumption of weak convergence of $\sqrt{n}(\FHJ-\FN)$ to some centered continuous Gaussian process.
Note that this assumption is now covered by our Theorem~\ref{th:FCLT HJ FN} and Proposition~\ref{prop:FCLT HJ FN}.
Let $\mathbb{D}_{\phi}\subset D(\mathbb{R})$ consist of $F\in D(\mathbb{R})$ that are non-decreasing.
Then for $F\in \mathbb{D}_{\phi}$, the poverty rate is defined as
\begin{equation}
\label{def:poverty rate}
\phi(F)
=
F\left(\beta F^{-1}(\alpha)\right)
\end{equation}
for fixed $0<\alpha,\beta<1$, where
$F^{-1}(\alpha)=\inf\left\{t:F(t)\geq \alpha\right\}$.
Typical choices are $\alpha=0.5$ and $\beta=0.5$ (INSEE) or $\beta=0.6$ (EUROSTAT).
Its Hadamard derivative is given by
\begin{equation}
\label{eq:derivative phi}
\phi_F'(h)
=
-\beta\frac{f(\beta F^{-1}(\alpha))}{f(F^{-1}(\alpha))}
h(F^{-1}(\alpha))
+
h(\beta F^{-1}(\alpha)).
\end{equation}
See the supplement B in~\cite{Boistard_2015} for details.
We then have the following corollaries for the Horvitz-Thompson estimator~$\phi(\FHT)$
and the H\'ajek estimator~$\phi(\FHJ)$
for the poverty rate $\phi(F)$.
\begin{corollary}
\label{cor:poverty rate HT}
Let $\phi$ be defined by~\eqref{def:poverty rate}
and suppose that the conditions of Proposition~\ref{prop:FCLT HT F} hold.
Then, if $F$ is differentiable at $F^{-1}(\alpha)$,
the random variable $\sqrt{n}(\phi(\FHT)-\phi(F))$ converges in distribution to
a mean zero normal random variable with variance
\begin{equation}
\label{def:lim var poverty rate HT}
\begin{split}
\sigma_{\mathrm{HT},\alpha,\beta}^2
&=
\beta^2\frac{f(\beta F^{-1}(\alpha))^2}{f(F^{-1}(\alpha))^2}
\left(\gamma_{\pi1}\alpha+\gamma_{\pi2}\alpha^2\right)\\
&\quad+
\gamma_{\pi1}\phi(F)
+
\gamma_{\pi2}\phi(F)^2
-2\beta\frac{f(\beta F^{-1}(\alpha))}{f(F^{-1}(\alpha))}
\phi(F)
\big(
\gamma_{\pi1}+\gamma_{\pi2}\alpha
\big),
\end{split}
\end{equation}
where $\gamma_{\pi1}=\mu_{\pi1}+\lambda$ and $\gamma_{\pi2}=\mu_{\pi2}-\lambda$.
If in addition $n/N\to0$,
then $\sqrt{n}(\phi(\FHT)-\phi(\FN))$ converges in distribution to a mean zero normal random variable with variance $\sigma^2_{\mathrm{HT},\alpha,\beta}$.
\end{corollary}
\begin{corollary}
\label{cor:poverty rate HJ}
Let $\phi$ be defined by~\eqref{def:poverty rate}.
and suppose that the conditions of Proposition~\ref{prop:FCLT HJ F} hold.
Then, if $F$ is differentiable at $F^{-1}(\alpha)$,
the random variable $\sqrt{n}(\phi(\FHJ)-\phi(F))$ converges in distribution to
a mean zero normal random variable with variance
\begin{equation}
\label{def:lim var poverty rate HJ}
\begin{split}
\sigma_{\mathrm{HJ},\alpha,\beta}^2
&=
\beta^2\frac{f(\beta F^{-1}(\alpha))^2}{f(F^{-1}(\alpha))^2}\gamma_{\pi1}\alpha(1-\alpha)\\
&\quad+
\gamma_{\pi1}
\phi(F)
\big(
1-\phi(F))
\big)
-2\beta\frac{f(\beta F^{-1}(\alpha))}{f(F^{-1}(\alpha))}\phi(F)
\gamma_{\pi1}(1-\alpha),
\end{split}
\end{equation}
where $\gamma_{\pi1}=\mu_{\pi1}+\lambda$.
If in addition $n/N\to0$,
then $\sqrt{n}(\phi(\FHJ)-\phi(\FN))$ converges in distribution to a mean zero normal random variable with variance $\sigma^2_{\mathrm{HJ},\alpha,\beta}$.
\end{corollary}
\section{Simulation study}
\label{sec:simulation}
The objective of this simulation study is to investigate the performance of the Horvitz-Thompson (HT)
and the H\'ajek (HJ) estimators for the poverty rate, as defined in~\eqref{def:poverty rate},
at the finite population level and at the super-population level.
The asymptotic results from Corollary~\ref{cor:poverty rate HT} and~\ref{cor:poverty rate HJ}
are used to obtain variance estimators whose performance is also assessed in this small study.

Six simulation schemes are implemented with different population sizes and (design-based) expected sample sizes,
namely $N=10\,000$ and $1000$ and $n=500$, $100$, and $50$.
The samples
are drawn according to three different sampling designs.
The first one is simple random sampling without replacement~(SI) with size~$n$.
The second design is Bernoulli sampling~(BE) with parameter~$n/N$.
The third one is Poisson sampling~(PO) with first order inclusion probabilities equal to $0.4n/N$
for the first half of the population and equal to $1.6n/N$ for the other half of the population,
where the population is randomly ordered.
The first order inclusion probabilities are deterministic for the three designs and the sample size $n_s$
is fixed for the SI design, while it is random with respect to the design for the BE and PO designs.
Moreover, the SI and BE designs are equal probability designs, while PO is an unequal probability design.
The results are obtained by replicating $N_R=1000$ populations. For each population, $n_R=1000$ samples are drawn according to the different designs.
The variable of interest $Y$ is generated for each population according to an exponential distribution with rate parameter equal to one.
For this distribution and given $\alpha$ and $\beta$, the poverty rate has an explicit expression $\phi(F)= 1-\exp(\beta\ln(1-\alpha))$.
In what follows, $\alpha=0.5$ and $\beta=0.6$ and $\phi(F)\simeq 0.34$.
These are the same values for $\alpha$ and $\beta$ as considered in~\cite{dell2008}.

The Horvitz-Thompson estimator and H\'ajek estimator for $\phi(F)$ or $\phi(\FN)$ are denoted by~$\widehat{\phi}_\mathrm{HT}$ and~$\widehat{\phi}_\mathrm{HJ}$, respectively.
They are obtained by plugging in the empirical c.d.f.'s $\FHT$ and $\FHJ$, respectively, for $F$ in expression~\eqref{def:poverty rate}.
The empirical quantiles are calculated by using the function \texttt{wtd.quantile}
from the R package \texttt{Hmisc} for the H\'ajek estimator and by adapting the function for the Horvitz-Thompson estimator.
For the SI sampling design, the two estimators are the same.
The performance of the estimators for the parameters $\phi(F)$ and $\phi(\FN)$
is evaluated using some Monte-Carlo relative bias (RB).
This is reported in Table~\ref{RB_povr}.
\begin{table}[t]
\caption{RB (in \%) of the HT and the HJ estimators for the finite population $\phi(\FN)$ and the super-population $\phi(F)$ poverty rate parameter}
\label{RB_povr}
\begin{tabular}{cccrrrrrrrrr}
\hline
 \multicolumn{3}{c}{\mbox{}} & \multicolumn{3}{c}{$N=10\,000$} & \multicolumn{3}{c}{$N=1000$}\\
 \multicolumn{3}{c}{\mbox{}} & $n=500$ & $n=100$ & $n=50$ & $n=500$ & $n=100$ & $n=50$\\
\hline
SI  & HT-HJ & $\phi(\FN)$ & $-$0.17 & $-$0.89 & $-$1.82 & $-$0.05 & $-$0.84 & $-$1.62 \\
    &       & $\phi(F)$   & $-$0.20 & $-$0.91 & $-$1.86 & $-$0.18 & $-$0.72 & $-$1.85 \\
\hline
    & HT    & $\phi(\FN)$ & $-$0.12 & $-$0.66 & $-$1.29 &  0.01 & $-$0.65 & $-$1.12\\
BE  &       & $\phi(F)$   & $-$0.15 & $-$0.68 & $-$1.34 & $-$0.12 & $-$0.54 & $-$1.36\\
    & HJ    & $\phi(\FN)$ & $-$0.17 & $-$0.92 & $-$1.87 & $-$0.04 & $-$0.88 & $-$1.68\\
    &       & $\phi(F)$   & $-$0.20 & $-$0.93 & $-$1.92 & $-$0.17 & $-$0.76 & $-$1.91\\
\hline
    & HT    & $\phi(\FN)$ & $-$0.05 & $-$1.05 & $-$2.06 & $-$0.06 & $-$0.30 & $-$0.37\\
PO  &       & $\phi(F)$   & $-$0.08 & $-$1.07 & $-$2.11 & $-$0.19 & $-$0.19 & $-$0.63\\
    & HJ    & $\phi(\FN)$ & $-$0.20 & $-$1.27 & $-$2.95 & $-$0.04 & $-$1.08 & $-$1.99\\
    &       & $\phi(F)$   & $-$0.23 & $-$1.28 & $-$3.00 & $-$0.17 & $-$0.97 & $-$2.23\\
\hline
\end{tabular}
\end{table}
When estimating the super-population parameter $\phi(F)$, if  $\widehat\phi_{ij}$ denotes the estimate
(either $\widehat{\phi}_\mathrm{HT}$ or $\widehat{\phi}_\mathrm{HJ}$) for the $i$th generated population and the $j$th drawn sample,
the Monte Carlo relative bias of $\widehat\phi$ in percentages has the following expression
\[
\text{RB}_F(\widehat\phi)
=
\frac{100}{N_R\, n_R}
\sum_{i=1}^{N_R}
\sum_{j=1}^{n_R}
\frac{\widehat\phi_{ij}-\phi(F)}{\phi(F)}.
\]
When estimating the finite population parameter $\phi(\FN)$,
the parameter depends on the generated population $N_i$, for each $i=1,\ldots,N_R$,
and will be denoted by $\phi(\mathbb{F}_{N_i})$.
The Monte Carlo relative bias of $\widehat\phi$ is then computed by
replacing $F$ by $\mathbb{F}_{N_i}$ in the above expression.
Concerning the relative biases reported in Table \ref{RB_povr}, the values are small and never exceed 3\%.
As expected, these values increase when~$n$ decreases.
When the centering is relative to $\phi(\FN)$, the relative bias is in general somewhat smaller than when centering with $\phi(F)$.
This behavior is most prominent when $N=1000$ and $n=500$, which suggests that the estimates are typically closer to the population poverty rate
$\phi(\FN)$ than to the model parameter $\phi(F)$.
The H\'ajek estimator has a larger relative bias than the Horvitz-Thompson estimator in all situations but in particular for the
Poisson sampling design when the size of the population is 1000.
Note that all values in Table~\ref{RB_povr} are negative, which illustrates the fact that the estimators typically
underestimate the population and model poverty rates.

In Table~\ref{RB_var_povr},
\begin{table}[t]
\caption{RB (in \%) for the variance estimator of the HT and the HJ estimators for the poverty rate parameter}
\label{RB_var_povr}
\begin{tabular}{ccrrrrrrrrr}
\hline
 \multicolumn{2}{c}{\mbox{}} & \multicolumn{3}{c}{$N=10\,000$} & \multicolumn{3}{c}{$N=1000$}\\
 \multicolumn{2}{c}{\mbox{}} & $n=500$ & $n=100$ & $n=50$ & $n=500$ & $n=100$ & $n=50$\\
\hline
SI  & HT-HJ & $-$2.21 & $-$3.08 & $-$2.97 & $-$2.25 & $-$3.26 & $-$3.00 \\
\hline
BE  & HT    & $-$4.15 & $-$5.11 & $-$4.21 & $-$3.31 & $-$5.11 & $-$4.19 \\
    & HJ    & $-$2.22 & $-$3.06 & $-$3.03 & $-$2.26 & $-$3.24 & $-$3.03 \\
\hline
PO  & HT    & $-$4.43 & $-$4.96 & $-$3.45 & $-$3.74 & $-$5.72 & $-$4.59 \\
    & HJ    & $-$2.36 & $-$3.43 & $-$3.36 & $-$2.44 & $-$3.75 & $-$4.13 \\
\hline
\end{tabular}
\end{table}
the estimators of  the variance of $\widehat{\phi}_\mathrm{HT}$ and~$\widehat{\phi}_\mathrm{HJ}$ are obtained by plugging in
the empirical c.d.f.'s $\FHT$ and $\FHJ$, respectively, for $F$ in the expressions~\eqref{def:lim var poverty rate HT} and~\eqref{def:lim var poverty rate HJ}.
To estimate $f$ in the variance of~$\widehat{\phi}_\mathrm{HJ}$, we follow~\cite{bergerskinner2003},
who propose a H\'ajek type kernel estimator with a Gaussian kernel function.
For the variance of~$\widehat{\phi}_\mathrm{HT}$, we use a corresponding Horvitz-Thompson estimator by replacing $\widehat{N}$ by $N$.
Based on~\cite{silverman1986}, pages 45-47, we choose $b=0.79Rn_s^{-1/5}$, where $R$ denotes the interquartile range.
This differs from~\cite{bergerskinner2003}, who propose a similar bandwidth of the order $N^{-1/5}$.
However, this severely underestimates the optimal bandwidth, leading to large variances of the kernel estimator.
Usual bias variance trade-off computations show that the optimal bandwidth is of the order $n_s^{-1/5}$.

For the SI sampling design, \eqref{def:lim var poverty rate HT} and~\eqref{def:lim var poverty rate HJ} are identical
and can be calculated in an explicit way using the fact that $\mu_{\pi1}+\lambda=1$ and $\mu_{\pi2}-\lambda=-1$.
For the BE design, $\mu_{\pi1}+\lambda=1$, whereas for Poisson sampling,
the value $(n/N^{2})\sum_{i=1}^N 1/\pi_i$ is taken for $\mu_{\pi1}+\lambda$.
For these designs, $\mu_{\pi2}-\lambda=-\lambda$, where we take $n/N$ as the value of $\lambda$.

In order to compute the relative bias of the variance estimates, the asymptotic variance is taken as reference.
This asymptotic variance $\mbox{AV}(\widehat\phi)$ of the estimator $\widehat\phi$
(either $\widehat{\phi}_\mathrm{HT}$ or $\widehat{\phi}_\mathrm{HJ}$)
is computed from~\eqref{def:lim var poverty rate HT} and~\eqref{def:lim var poverty rate HJ}.
The expressions $f(\beta F^{-1}(\alpha))$ and $f(F^{-1}(\alpha))$ are explicit in the case of an exponential distribution.
Furthermore, for $\mu_{\pi1}+\lambda$ and $\mu_{\pi2}-\lambda$ we use the same expressions as mentioned above.
The Monte Carlo relative bias of the variance estimator $\widehat{\text{AV}}(\widehat\phi)$ in percentages,
is defined by
\[
\text{RB}(\widehat{\text{AV}}(\widehat\phi))
=
\frac{100}{N_R\, n_R}
\sum_{i=1}^{N_R}
\sum_{j=1}^{n_R}
\frac{\widehat{\text{AV}}(\widehat\phi_{ij})-\text{AV}(\widehat\phi)}{\text{AV}(\widehat\phi)},
\]
where $\widehat{\text{AV}}(\widehat\phi_{ij})$ denotes the variance estimate for the $i$th generated population and the $j$th drawn sample.

Table~\ref{CR_povr}
\begin{table}[t]
\caption{Coverage probabilities (in \%) for 95\% confidence intervals of the HT and the HJ estimators for the finite population $\phi(\FN)$ and the super-population $\phi(F)$ poverty rate parameter}
\label{CR_povr}
\begin{tabular}{cccrrrrrrrr}
\hline
 \multicolumn{3}{c}{\mbox{}} & \multicolumn{3}{c}{$N=10\,000$} & \multicolumn{3}{c}{$N=1000$}\\
 \multicolumn{3}{c}{\mbox{}} & $n=500$ & $n=100$ & $n=50$ & $n=500$ & $n=100$ & $n=50$\\
\hline
SI  & HT-HJ & $\phi(\FN)$ & 95.2 & 94.4 & 93.5 & 98.8 & 95.1 & 94.6\\
    &       & $\phi(F)$   & 94.6 & 93.2 & 92.2 & 94.7 & 93.2 & 92.0\\
\hline
    & HT    & $\phi(\FN)$ & 94.9 & 94.3 & 94.6 & 98.4 & 94.8 & 94.6\\
BE  &       & $\phi(F)$   & 94.4 & 93.7 & 94.9 & 94.6 & 93.6 & 94.7\\
    & HJ    & $\phi(\FN)$ & 95.1 & 94.3 & 93.9 & 98.7 & 94.9 & 94.2\\
    &       & $\phi(F)$   & 94.7 & 94.2 & 93.9 & 94.7 & 94.2 & 93.9\\
\hline
    & HT    & $\phi(\FN)$ & 94.5 & 94.2 & 94.3 & 96.8 & 94.0 & 93.6\\
PO  &       & $\phi(F)$   & 94.5 & 94.0 & 94.3 & 94.6 & 93.6 & 93.5\\
    & HJ    & $\phi(\FN)$ & 94.8 & 93.9 & 93.6 & 97.2 & 94.2 & 93.3\\
    &       & $\phi(F)$   & 94.6 & 93.9 & 93.6 & 94.6 & 93.9 & 93.2\\
\hline
\end{tabular}
\end{table}
gives the Monte-Carlo coverage probabilities for a nominal coverage probability of 95\% for the two parameters $\phi(\FN)$ and $\phi(F)$, the Horvitz-Thompson and the H\'ajek estimators and the different simulation schemes.
In general the coverage probabilities are somewhat smaller than 95\%, which is due to the underestimation of the asymptotic variance,
as can be seen from Table~\ref{RB_var_povr}.
The case $N=1000$ and $n=500$ for $\widehat{\phi}_\mathrm{HJ}$ forms an exception,
which is probably due to the fact that in this case $\lambda=n/N$ is far from zero, so that
the limit distribution of $\sqrt{n}(\phi(\FHT)-\phi(\FN))$ and $\sqrt{n}(\phi(\FHJ)-\phi(\FN))$ has a larger variance than the ones reported in Corollaries~\ref{cor:poverty rate HT}
and~\ref{cor:poverty rate HJ}.
When looking at Table~\ref{RB_var_povr}, the relative biases are smaller than 5\% when $n$ is 500.
The biases are larger for the Horvitz-Thompson estimator than for the H\'ajek estimator.
Again all relative biases are negative, which illustrates the fact that the asymptotic variance is typically underestimated.

\section{Discussion}
\label{sec:related}
{
In the appendix of~\cite{lin2000} the author remarks ``To our knowledge there does not exist a general theory on conditions required for the tightness and weak convergence
of Horvitz-Thompson processes."
One purpose of this paper has been to obtain these type of results in such a way that they are potentially applicable to a large class of {single-stage} unequal probability sampling designs.
Conditions (C2)-(C4) play a crucial role in this, as they establish the tightness of the processes involved.
The main motivation for the way they are formulated is to incorporate {single-stage} sampling designs which allow the sample size and/or the inclusion probabilities to depend on $\omega$,
which will be the case if they depend on the auxiliary variables $Z_i$.
These conditions trivially hold for simple sampling designs, but also for rejective sampling, which enables us to obtain weak convergence of the H\'ajek and Horvitz-Thompson processes under
high entropy designs.
Further extensions to more complex designs are beyond the scope of the present investigation,
but we believe that results similar to those described in Sections~\ref{sec:HT}, \ref{sec:Hajek}, and~\ref{sec:examples}, would continue to hold under reasonable assumptions.

%multistage
For instance multistage sampling designs deserve attention.
The recent paper~\cite{Chauvet_2015} gives some asymptotic results in the case of simple random sampling without replacement at the first stage and
with arbitrary designs at further stages.
\cite{escobarberger2013} gives also some consistency results for a particular two-stage fixed sample size design.
The clusters are drawn using {sampling without replacement with a probability proportional to the size design and the secondary units are drawn using a
simple random sampling without replacement} within each sampled cluster.
This leads to a self-weighted design.
Similar designs would be worth considering in order to generalize our functional limit theorems to multistage sampling.

Stratified sampling is also of importance.
Asymptotics in the case of stratified simple random sampling without replacement is studied in~\cite{Bickel_1984},
when the number of strata is bounded and~in \cite{Krewski_1981} when the number of strata tends to infinity.
More recently, consistency results are obtained in~\cite{Berger_2011} for large entropy designs when the number of strata is bounded.
It would be of particular interest to generalize our functional asymptotic results to such stratified designs.

Our results rely on the assumption that the sample selection process and the super-population model characteristic are independent given the design variables.
It means that the sampling is non-informative~\cite{Pfeffermann_2009}.
Our results do not directly generalize to informative sampling and further research is needed for such sampling designs.
Also functional CLT's for processes corresponding to other estimators,
such as regression and calibration estimators~(\cite{Deville_1992}) {deserve} attention.
}

\section{Proofs}
\label{sec:proofs}
We will use Theorem 13.5 from~\cite{Billingsley_1999}, which requires convergence of finite dimensional
distributions and a tightness condition (see (13.14) in~\cite{Billingsley_1999}.
To obtain weak convergence of the finite dimensional distributions, we use condition (HT1) in combination with the
Cr\'amer-Wold device, see Lemmas~\ref{lem:fidis HT FN}, \ref{lem:fidis HT F}, and~\ref{lem:fidis HT F deterministic}.
Details of their proofs can be found in the supplement A in~\cite{Boistard_2015}.

We will now establish the tightness condition, as stated in the following lemma.
\begin{lemma}
\label{lem:tightness HT FN}
Let $Y_1,\ldots,Y_N$ be i.i.d.~random variables with c.d.f.~$F$ and empirical c.d.f.~$\FN$ and let $\FHT$ be defined according to~\eqref{eq:HT cdf}.
Let $\mathbb{X}_N=\sqrt{n}(\FHT-\FN)$ and suppose that (C1)-(C4) hold.
Then there exists a constant $K>0$ independent of~$N$, such that for any $t_1$, $t_2$ and $-\infty<t_1\leq t\leq t_2<\infty$,
\[
\mathbb{E}_{d,m}
\left[
\left(
\mathbb{X}_N(t)-\mathbb{X}_N(t_1)
\right)^2
\left(
\mathbb{X}_N(t_2)-\mathbb{X}_N(t)
\right)^2
\right]
\leq
K\Big(F(t_2)-F(t_1)\Big)^2.
\]
%for $-\infty<t_1\leq t\leq t_2<\infty$.
\end{lemma}
\begin{proof}
First note that
\[
\mathbb{X}_N(t)
=
\frac{\sqrt{n}}{N}
\sum_{i=1}^N
\left(
\frac{\xi_i}{\pi_i}
-1
\right)
\mathds{1}_{\{Y_i\leq t\}}.
\]
For the sake of brevity, for $-\infty<t_1\leq t\leq t_2<\infty$,
and $i=1,2,\ldots,N$, we define
$p_1= F(t)-F(t_1)$, $p_2= F(t_2)-F(t)$,
$A_i=\mathds{1}_{\{t_1< Y_i\leq t\}}$,
and $B_i=\mathds{1}_{\{t< Y_i\leq t_2\}}$.
Furthermore, let $\alpha_i=(\xi_i-\pi_i)A_i/\pi_i$ and $\beta_i=(\xi_i-\pi_i)B_i/\pi_i$.
Then, according to the fact that
$p_1p_2\leq (F(t_2)-F(t_1))^2$, due to the monotonicity of $F$, it suffices to show
\begin{equation}
\label{eq:toprove2}
\frac{1}{N^4}
\mathbb{E}_{d,m}
\left[
n^2
\left(\sum_{i=1}^N\alpha_i\right)^2
\left(\sum_{j=1}^N\beta_j\right)^2
\right]
\leq
K
p_1p_2.
\end{equation}
The expectation on the left hand side can be decomposed as follows
\begin{equation}
\label{eq:decomposition}
\begin{split}
&\sum_{i=1}^N\sum_{k=1}^N
\mathbb{E}_{d,m}\left[n^2\alpha_i^2\beta_k^2\right]
+
\sum_{i=1}^N\sum_{j\ne i}\sum_{k=1}^N
\mathbb{E}_{d,m}\left[n^2\alpha_i\alpha_j\beta_k^2\right]\\
&+
\sum_{k=1}^N\sum_{l\ne k}\sum_{i=1}^N
\mathbb{E}_{d,m}\left[n^2\alpha_i^2\beta_k\beta_l\right]
+
\sum_{i=1}^N\sum_{j\neq i}\sum_{k=1}^N\sum_{l\neq k}
\mathbb{E}_{d,m}\left[n^2\alpha_i\alpha_j\beta_k\beta_l\right].
\end{split}
\end{equation}
Note that by symmetry, sums two and three on the right hand side can be handled similarly,
so that essentially we have to deal with three summations.
We consider them one by one.

First note that,
since $\mathds{1}_{\{t_1<Y_i\leq t\}}\mathds{1}_{\{t<Y_i\leq t_2\}}=0$, we will only have non-zero expectations when
$\{i,j\}$ and $\{k,l\}$ are disjoint.
With (C1), we find
\begin{equation}
\label{eq:bound sum1}
\begin{split}
&
\frac{1}{N^4}
\sum_{i=1}^N\sum_{k=1}^N
\mathbb{E}_{d,m}\left[n^2\alpha_i^2\beta_k^2\right]
=
\frac{1}{N^4}
\ssum_{(i,k)\in D_{2,N}}
\mathbb{E}_{d,m}\left[n^2\alpha_i^2\beta_k^2\right]\\
&=
\frac{1}{N^4}
\ssum_{(i,k)\in D_{2,N}}
\mathbb{E}_{m}
\left[
n^2
\frac{A_iB_k}{\pi_i^2\pi_k^2}
\mathbb{E}_{d}
(\xi_i-\pi_i)^2(\xi_k-\pi_k)^2
\right]\\
&\leq
\frac{1}{K_1^4}
\ssum_{(i,k)\in D_{2,N}}
\mathbb{E}_{m}
\left[
\frac{A_iB_k}{n^2}
\mathbb{E}_{d}
(\xi_i-\pi_i)^2(\xi_k-\pi_k)^2
\right]
\end{split}
\end{equation}
Straightforward computation shows that $\mathbb{E}_{d}(\xi_i-\pi_i)^2(\xi_k-\pi_k)^2$ equals
\[
(\pi_{ik}-\pi_i\pi_k)(1-2\pi_i)(1-2\pi_k)
+
\pi_i\pi_k(1-\pi_i)(1-\pi_k).
\]
Hence, with (C1)-(C2) we find that
\[
\mathbb{E}_{d}
(\xi_i-\pi_i)^2(\xi_k-\pi_k)^2
\leq
\left|
\mathbb{E}_{d}
(\xi_i-\pi_i)(\xi_k-\pi_k)
\right|
+
K_2^2
\frac{n^2}{N^2}
=
O\left(\frac{n^2}{N^2}\right),
\]
$\omega$-almost surely.
It follows that
\[
\frac{1}{N^4}
\sum_{i=1}^N\sum_{k=1}^N
\mathbb{E}_{d,m}\left[n^2\alpha_i^2\beta_k^2\right]
\leq
O\left(\frac{1}{N^2}\right)
\ssum_{(i,k)\in D_{2,N}}
\mathbb{E}_{m}
\left[
A_iB_k
\right].
\]
Since $D_{2,N}$ has $N(N-1)$ elements and $\mathbb{E}_{m}[A_iB_j]=p_1p_2$ for $(i,j)\in D_{2,N}$, it follows that
\begin{equation}
\label{eq:term1total FN}
\frac{1}{N^4}
\sum_{i=1}^N\sum_{j=1}^N
\mathbb{E}_{d,m}\left[n^2\alpha_i^2\beta_j^2\right]\leq Kp_1p_2.
\end{equation}
Consider the second (and third) summation on the right hand side of~\eqref{eq:decomposition}.
Similarly to~\eqref{eq:bound sum1}, we can then write
\[\begin{split}
&
\frac{1}{N^4}
\left|
\sum_{i=1}^N\sum_{j\ne i}\sum_{k=1}^N
\mathbb{E}_{d,m}\left[n^2\alpha_i\alpha_j\beta_k^2\right]
\right|
=
\frac{1}{N^4}
\left|
\sssum_{(i,j,k)\in D_{3,N}}
\mathbb{E}_{d,m}\left[n^2\alpha_i\alpha_j\beta_k^2\right]
\right|\\
&\leq
\frac{1}{N^4}
\sssum_{(i,j,k)\in D_{3,N}}
\left|
\mathbb{E}_{d,m}
\left[
n^2
\frac{A_iA_jB_k}{\pi_i\pi_j\pi_k^2}
(\xi_i-\pi_i)(\xi_j-\pi_j)(\xi_k-\pi_k)^2
\right]
\right|\\
&\leq
\frac{1}{N^4}
\sssum_{(i,j,k)\in D_{3,N}}
\mathbb{E}_{m}
\left[
n^2
\frac{A_iA_jB_k}{\pi_i\pi_j\pi_k^2}
\Big|
\mathbb{E}_{d}
(\xi_i-\pi_i)(\xi_j-\pi_j)(\xi_k-\pi_k)^2
\Big|
\right]\\
&\leq
\frac{1}{K_1^4}
\sssum_{(i,j,k)\in D_{3,N}}
\mathbb{E}_{m}
\left[
\frac{A_iA_jB_k}{n^2}
\Big|
\mathbb{E}_{d}
(\xi_i-\pi_i)(\xi_j-\pi_j)(\xi_k-\pi_k)^2
\Big|
\right].
\end{split}
\]
We find that $\mathbb{E}_{d}(\xi_i-\pi_i)(\xi_j-\pi_j)(\xi_k-\pi_k)^2$ equals
\[
(1-2\pi_k)
\mathbb{E}_{d}
(\xi_i-\pi_i)(\xi_j-\pi_j)(\xi_k-\pi_k)
+
\pi_k(1-\pi_k)
\mathbb{E}_{d}
(\xi_i-\pi_i)(\xi_j-\pi_j)
\]
With (C1)-(C3), this means
$|\mathbb{E}_{d}(\xi_i-\pi_i)(\xi_j-\pi_j)(\xi_k-\pi_k)^2|=O(n^2/N^3)$,
$\omega$-almost surely.
It follows that
\[
\frac{1}{N^4}
\left|
\sum_{i=1}^N\sum_{j\ne i}\sum_{k=1}^N
\mathbb{E}_{d,m}\left[n^2\alpha_i\alpha_j\beta_k^2\right]
\right|
=
O\left(\frac{1}{N^3}\right)
\sssum_{(i,j,k)\in D_{3,N}}
\mathbb{E}_{m}
\left[
A_iA_jB_k
\right].
\]
Since $D_{3,N}$ has $N(N-1)(N-2)$ elements and
$\mathbb{E}_{d,m}[A_iA_jB_k]=p_1^2p_2$,
for $(i,j,k)\in D_{3,N}$, we find
\begin{equation}
\label{eq:toprove sum2 FN}
\frac{1}{N^4}
\left|
\sum_{i=1}^N\sum_{j\ne i}\sum_{k=1}^N
\mathbb{E}_{d,m}\left[n^2\alpha_i\alpha_j\beta_k^2\right]
\right|
\leq
Kp_1p_2.
\end{equation}
The computations for the third summation in~\eqref{eq:decomposition} are completely similar.
Finally, consider the last summation in~\eqref{eq:decomposition}.
As before, this summation can be bounded by
\[
\frac{1}{K_1^4}
\sum_{(i,j,k,l)\in D_{4,N}}
\mathbb{E}_{m}
\left[
\frac{A_iA_jB_kB_l}{n^2}
\Big|
\mathbb{E}_{d}
(\xi_i-\pi_i)(\xi_j-\pi_j)(\xi_k-\pi_k)(\xi_l-\pi_l)
\Big|
\right].
\]
Since $D_{4,N}$ has $N(N-1)(N-2)(N-3)$ elements and
$\mathbb{E}_{m}[A_iA_jB_kB_l]=p_1^2p_2^2$, for $(i,j,k,l)\in D_{4,N}$,
with~(C4) we conclude that
\begin{equation}
\label{eq:toprove sum3 FN}
\frac{1}{N^4}
\left|
\sum_{i=1}^N\sum_{j\neq i}\sum_{k=1}^N\sum_{l\neq k}
\mathbb{E}_{d,m}\left[n^2\alpha_i\alpha_j\beta_k\beta_l\right]
\right|
\leq
Kp_1p_2.
\end{equation}
Together with~\eqref{eq:term1total FN}, \eqref{eq:toprove sum2 FN} and decomposition~\eqref{eq:decomposition},
this proves~\eqref{eq:toprove2}.
\end{proof}

\begin{lemma}
\label{lem:fidis HT FN}
Let $\mathbb{X}_N=\sqrt{n}(\FHT-\FN)$ and suppose that (C1)-(C2),(HT1)-(HT2) hold.
For any $k\in \{1,2,\ldots\}$, and $t_1,\ldots,t_k\in \mathbb{R}$,
$\big(\mathbb{X}_N(t_1),\ldots,\mathbb{X}_N(t_k)\big)$ converges
in distribution under~$\mathbb{P}_{d,m}$
to a $k$-variate mean zero normal random vector with covariance matrix
$\mathbf{\Sigma}_k^{\mathrm{HT}}$ given in~\eqref{eq:HT2 alternative moment}.
\end{lemma}
\begin{proof}
The proof can be found in the supplement Ai in~\cite{Boistard_2015}.
\end{proof}

\paragraph{Proof of Theorem~\ref{th:FCLT HT FN}}
We first consider $\mathbb{X}_N=\sqrt{n}(\FHT-\FN)$
for the case that the $Y_i$'s follow a uniform distribution on $[0,1]$.
We apply Theorem~13.5 from~\cite{Billingsley_1999}.
Lemma~\ref{lem:fidis HT FN} provides the limiting distribution of the finite dimensional projections
$(\mathbb{X}_N(t_1),\ldots,\mathbb{X}_N(t_k))$, which is the same as that of the vector
$(\mathbb{G}^{\mathrm{HT}}(t_1),\ldots,\mathbb{G}^{\mathrm{HT}}(t_k))$,
where $\mathbb{G}^{\mathrm{HT}}$ is a mean zero Gaussian process with covariance function
\[
\mathbb{E}_{m}\mathbb{G}^{\mathrm{HT}}(s)\mathbb{G}^{\mathrm{HT}}(t)
=
\lim_{N\to\infty}
\frac{1}{N^2}
\sum_{i=1}^N\sum_{j=1}^N
\mathbb{E}_{m}
\left[
n
\frac{\pi_{ij}-\pi_i\pi_j}{\pi_i\pi_j}
\mathds{1}_{\{Y_i\leq s\}}\mathds{1}_{\{Y_j\leq t\}}
\right],
\]
for all $s,t\in \mathbb{R}$.
Tightness condition (13.14) in~\cite{Billingsley_1999} is provided by Lemma~\ref{lem:tightness HT FN}.
Since $\mathbb{G}^{\mathrm{HT}}$ is continuous at 1,
the theorem now follows from Theorem~13.5 in~\cite{Billingsley_1999} for the case that the $Y_i$'s are uniformly distributed on~$[0,1]$.

To extend this to a functional CLT with i.i.d.~random variables $Y_1,Y_2,\ldots$ with a general c.d.f.~$F$,
we can follow the argument in the proof of Theorem~14.3 from~\cite{Billingsley_1999}.
First define the generalized inverse of~$F$:
\[
\varphi(s)
=
\inf\{t: s\leq F(t)\},
\]
that satisfies $s\leq F(t)$ if and only if $\varphi(s)\leq t$.
This means that if $U_1,U_2,\ldots$ are i.i.d.~uniformly distributed on $[0,1]$, $\varphi(U_i)$ has the same distribution as~$Y_i$, so that
$\mathds{1}_{\{Y_i\leq t\}}
\stackrel{d}{=}
\mathds{1}_{\{\varphi(U_i) \leq t\}}
=
\mathds{1}_{\{U_i\leq F(t)\}}$.
It follows that
\[
\mathbb{X}_N(t)
=
\sqrt{n}
\left\{
\frac{1}{N}
\sum_{i=1}^N
\frac{\xi_i\mathds{1}_{\{Y_i\leq t\}}}{\pi_i}
-
\frac{1}{N}
\sum_{i=1}^N
\mathds{1}_{\{Y_i\leq t\}}
\right\}
\stackrel{d}{=}
Z_N(F(t)),
\quad
t\in \mathbb{R},
\]
where
\begin{equation}
\label{eq:HT EP uniform}
Z_N(t)
=
\frac{\sqrt{n}}{N}
\sum_{i=1}^N
\left(\frac{\xi_i}{\pi_i}-1\right)\mathds{1}_{\{U_i\leq t\}},
\quad
t\in[0,1],
\end{equation}
Hence, the general HT empirical process $\mathbb{X}_N$ is the image of the HT uniform empirical process $Z_N$ under the mapping
$\psi:D[0,1]\mapsto D(\mathbb{R})$ given by
$\left[\psi x\right](t)=x(F(t))$.
Note that, if $x_N\to x$ in $D[0,1]$ in the Skorohod topology and $x$ has continuous sample paths,
then the convergence is uniform.
But then also $\psi x_N$ converges to $\psi x$ uniformly in $D(\mathbb{R})$.
This implies that $\psi x_N$ converges to $\psi x$ in the Skorohod topology.
We have established that $Z_N\Rightarrow Z$ weakly in $D[0,1]$ in the Skorohod topology, where~$Z$ has continuous sample paths.
Therefore, according to the continuous mapping theorem, e.g., Theorem~2.7 in~\cite{Billingsley_1999}, it follows that
$\psi(Z_N)\Rightarrow \psi(Z)$ weakly.
This proves the theorem for $Y_i$'s with a general c.d.f.~$F$.
\hfill\tqed

\bigskip

The proof of Proposition~\ref{prop:FCLT HT FN} is similar to that of Theorem~\ref{th:FCLT HT FN} and can be found
in the supplement A in~\cite{Boistard_2015}.

\bigskip

To establish tightness for the process $\sqrt{n}(\FHT-F)$ we use the following decomposition
\begin{equation}
\label{eq:tightness decomposition}
\sqrt{n}(\FHT-F)
=
\sqrt{n}(\FHT-\FN)
+
\frac{\sqrt{n}}{\sqrt{N}}\cdot\sqrt{N}(\FN-F).
\end{equation}
The first process on the right hand side converges weakly to Gaussian process, according to Theorem~\ref{th:FCLT HT FN}.
The process $\sqrt{N}(\FN-F)$ also converges weakly to a Gaussian process, due to the classical Donsker theorem.
In particular both processes on the right hand side are tight in $D(\mathbb{R})$ with the Skorohod metric.
In general the sum of two tight processes in $D(\mathbb{R})$ is not necessarily tight.
However, this will be the case if both processes converge weakly to continuous processes
(see Lemma~\ref{lem:tightness sum} in~\cite{Boistard_2015}).

\begin{lemma}
\label{lem:HT CLT}
Let $V_1,V_2,\ldots$ be a sequence of bounded i.i.d.~random variables on $(\Omega,\mathfrak{F},\mathbb{P}_{m})$
with mean $\mu_V$ and variance $\sigma_V^2$, and let $S_N^2$ be defined by~\eqref{def:variance HT}.
Suppose (HT1) and (HT3) hold and $nS_N^2\to \sigma_{\mathrm{HT}}^2>0$ in $\mathbb{P}_{m}$-probability.
Then,
\begin{equation}
\label{eq:HT estimator}
\sqrt{n}
\left(
\frac{1}{N}
\sum_{i=1}^N
\frac{\xi_iV_i}{\pi_i}
-
\mu_V
\right),
\end{equation}
converges in distribution under $\mathbb{P}_{d,m}$ to a mean zero normal random variable with variance
$\sigma_{\mathrm{HT}}^2+\lambda\sigma^2_V$.
\end{lemma}
\begin{proof}
The proof can be found in the supplement A in~\cite{Boistard_2015}.
\end{proof}

%Note that, in view of the expression for $\sigma_{\mathrm{HT}}^2$ obtained in Lemma~\ref{lem:variance HT},
%for simple random sampling without replacement, the condition~$\sigma^2_{\mathrm{HT}}>0$ implies that~$\lambda$ must differ from 1.

\begin{lemma}
\label{lem:fidis HT F}
Let $\mathbb{X}_N^F=\sqrt{n}(\FHT-F)$ and suppose that (C1)-(C2),(HT1)-(HT4) hold.
Then for any $k\in \{1,2,\ldots\}$, and $t_1,t_2,\ldots,t_k\in \mathbb{R}$,
the sequence
$\big(\mathbb{X}_N^F(t_1),\ldots,\mathbb{X}_N^F(t_k)\big)$ converges
in distribution under~$\mathbb{P}_{d,m}$
to a $k$-variate mean zero normal random vector with covariance matrix
$\SigmaHTF
=
\mathbf{\Sigma}_k^{\mathrm{HT}}+\lambda \mathbf{\Sigma}_F$,
where $\mathbf{\Sigma}_k^{\mathrm{HT}}$ is given in~\eqref{eq:HT2 alternative moment}
and
$\mathbf{\Sigma}_F$ is the $k\times k$ matrix with $(q,r)$-entry
$F(t_q\wedge t_r)-F(t_q)F(t_r)$, for $q,r=1,2,\ldots,k$.
\end{lemma}
\begin{proof}
The proof can be found in the supplement A in~\cite{Boistard_2015}.
\end{proof}

%\begin{proof}
%The proof is similar to the proof of Lemma \ref{lem:fidis HT FN}.
%The details can be found in the supplement A in~\cite{Boistard_2015}.
%\end{proof}

\paragraph{Proof of Theorem~\ref{th:FCLT HT F}}
The proof is completely similar to that of Theorem~\ref{th:FCLT HT FN}.
We first consider the process $\mathbb{X}_N^F=\sqrt{n}(\FHT-F)$
for the case that the $Y_i$'s follow a uniform distribution with $F(t)=t$.
Decompose~$\mathbb{X}_N^F$ as in~\eqref{eq:tightness decomposition}.
By Theorem~\ref{th:FCLT HT FN}, the first process on the right hand side of~\eqref{eq:tightness decomposition}
converges weakly to a process in $C[0,1]$.
Due to the classical Donsker theorem and~(HT3), the second process on the right hand side of~\eqref{eq:tightness decomposition}
also converges weakly to a process in $C[0,1]$.
Tightness of~$\mathbb{X}_N^F$ then follows from Lemma~\ref{lem:tightness sum} in~\cite{Boistard_2015}.
Convergence of the finite dimensional distributions is provided by Lemma~\ref{lem:fidis HT F}.
The theorem now follows from Theorem~13.5 in~\cite{Billingsley_1999} for the case that the~$Y_i$'s are uniformly distributed on~$[0,1]$.
Next, this is extended to~$Y_i$'s with a general c.d.f.~$F$ in the same way as in the proof of Theorem~\ref{th:FCLT HT FN}.
\hfill\tqed

\bigskip

To establish convergence in distribution of the finite dimensional distributions of $\sqrt{n}(\FHT-F)$
under the conditions of Proposition~\ref{prop:FCLT HT F},
as in the proof of Lemma~\ref{lem:fidis HT F}, we will use the Cram\'er-Wold device.
To ensure that~$nS_N^2$ still has a strictly positive limit without imposing (HT4),
we will need the following lemma.
Its proof can be found in the supplement A in~\cite{Boistard_2015}.
\begin{lemma}
\label{lem:positive definite}
Let $F$ be the c.d.f.~of the i.i.d.~$Y_1,\ldots,Y_N$.
For any $k$-tuple $(t_1,\ldots, t_k)\in\mathbb{R}^k$, suppose that the values $F(t_1),\ldots, F(t_k)$ are all distinct and such that $0<F(t_i)<1$.
Let $a,b\in \mathbb{R}$, such that $a\geq b$.
If $a>0$, then the $k\times k$ matrix $\mathbf{M}$ with $(i,j)$-th element $M_{ij}=aF(t_i\wedge t_j)-bF(t_i)F(t_j)$ is positive definite.
\end{lemma}

\begin{lemma}
\label{lem:fidis HT F deterministic}
Let $\mathbb{X}_N^F=\sqrt{n}(\FHT-F)$ and suppose that
$n$ and $\pi_i,\pi_{ij}$, for $i,j=1,2,\ldots,N$, are deterministic.
Suppose that (C1)-(C2), (HT1) and~(HT3) hold, as well as conditions (i)-(ii) of
Proposition~\ref{prop:FCLT HT F}.
Then, for any $k\in \{1,2,\ldots\}$, and $t_1,\ldots,t_k\in \mathbb{R}$,
$\big(\mathbb{X}_N^F(t_1),\ldots,\mathbb{X}_N^F(t_k)\big)$ converges
in distribution under~$\mathbb{P}_{d,m}$
to a $k$-variate mean zero normal random vector with covariance matrix
$\SigmaHTF$, with $(q,r)$-entry
$(\mu_{\pi1}+\lambda)F(t_q\wedge t_r)+(\mu_{\pi2}-\lambda)F(t_q)F(t_r)$,
for $q,r,=1,2,\ldots,k$.
\end{lemma}
\begin{proof}
The proof can be found in the supplement A in~\cite{Boistard_2015}.
\end{proof}

The proof of Proposition~\ref{prop:FCLT HT F} is similar to that of Theorem~\ref{th:FCLT HT F} and
can be found in the supplement A in~\cite{Boistard_2015}.

%\subsection{Proofs for Section~\ref{sec:Hajek}}
\paragraph{Proof of Theorem~\ref{th:FCLT Lumley}}
For part (i), note that with $S_N^2$ defined in~\eqref{def:variance HT} with $V_i=1$,
from (HT1) together with condition~\eqref{eq:cond CLT Nhat}, it follows that
\[
\sqrt{n}S_N\times
\frac{1}{S_N}
\left(
\frac{1}{N}
\sum_{i=1}^N
\frac{\xi_i}{\pi_i}
-
1
\right)
\to
N(0,\sigma_{\pi}^2),
\quad
\omega-\text{a.s.},
\]
in distribution under $\mathbb{P}_d$.
This implies
\begin{equation}
\label{eq:CLT Nhat}
\sqrt{n}
\left(
\frac{\widehat{N}}{N}-1
\right)
=
\sqrt{n}
\left(
\frac{1}{N}
\sum_{i=1}^N
\frac{\xi_i}{\pi_i}
-
1
\right)
\to
N(0,\sigma_{\pi}^2),
\end{equation}
in distribution under $\mathbb{P}_{d,m}$.
In particular, since $n\to\infty$, this proves part~(i).

The proof of part(ii) is along the same lines as the proof of Theorems~\ref{th:FCLT HT FN} and~\ref{th:FCLT HT F}.
First consider the case, where the $Y_i$'s are uniform, with $F(t)=t$ on $[0,1]$.
Then, with $\FHT$ defined in~\eqref{eq:HT cdf} and $\mathbb{X}_N^F=\sqrt{n}(\FHT-F)$, we can write
$\GL(t)=\mathbb{X}_N^F(t)-(\mathbb{X}_N^F(t)-\GL(t))$.
According to Theorem~\ref{th:FCLT HT F}, the process~$\mathbb{X}_N^F$ converges weakly to a continuous process.
As a consequence of~\eqref{eq:CLT Nhat},
the process
\[
\mathbb{X}_N^F(t)-\GL(t)
=
t
\sqrt{n}
\left(
\frac{1}{N}
\sum_{i=1}^N
\frac{\xi_i}{\pi_i}
-
1
\right),
\]
also converges weakly to a continuous process.
Hence, similar to the argument in the proof of Theorem~\ref{th:FCLT HT F},
we conclude that the process $\GL$ is tight.
Next, we establish weak convergence of the finite dimensional projections.
{Details can be found in the supplement A in~\cite{Boistard_2015}.}
\hfill\tqed

\paragraph{Proof of Theorem~\ref{th:FCLT HJ FN}}
We use~\eqref{eq:decompose HJ EP}.
From the proof of Theorem~\ref{th:FCLT Lumley}, we know that $\GL$ is tight.
Together with Theorem~\ref{th:FCLT Lumley}(i), it then follows that the limit behavior of $\sqrt{n}(\FHJ-\FN)$
is the same as that of the process $\mathbb{Y}_N$ defined in~\eqref{def:Y_N}.
This process can be written as
\[
\mathbb{Y}_N(t)
=
\frac{\sqrt{n}}{N}
\sum_{i=1}^N
\left(
\frac{\xi_i}{\pi_i}-1
\right)
\mathds{1}_{\{Y_i\leq t\}}
-
F(t)
\frac{\sqrt{n}}{N}
\sum_{i=1}^N
\left(
\frac{\xi_i}{\pi_i}-1
\right).
\]
As in the proofs of
Theorems~\ref{th:FCLT HT FN}, \ref{th:FCLT HT F}, and~\ref{th:FCLT Lumley},
we first consider the case of uniform $Y_i$'s.
The first process on the right hand side is $\sqrt{n}(\FHT-\FN)$, which converges weakly
to a continuous process, according to Theorem~\ref{th:FCLT HT FN},
whereas the second process also converges to a continuous process due to~\eqref{eq:CLT Nhat}.
As in the proof of Theorem~\ref{th:FCLT HT F}, one can then argue that~$\mathbb{Y}_N$, being the difference
of these processes, is tight.
Next, we prove weak convergence of the finite dimensional projections.
{Details can be found in the supplement A in~\cite{Boistard_2015}.}
\hfill\tqed

\bigskip

The proofs of Propositions~\ref{prop:FCLT HJ FN} and~\ref{prop:FCLT HJ F} are similar to those of
Theorems~\ref{th:FCLT HJ FN} and~\ref{th:FCLT Lumley}, respectively, and can be found in the supplement A in~\cite{Boistard_2015}.

{
\paragraph{Proof of Corollary~\ref{cor:high entropy}}
Similar to the approach followed in~\cite{berger1998b},
we first prove the results for a rejective sampling $R$ and then extend them to high entropy designs.

First note that $\mathbb{E}_d(\xi_i-\pi_i)(\xi_j-\pi_j)=\pi_{ij}-\pi_i\pi_j$.
According to Theorem~1 in~\cite{Boistard_2012}, which is an extension of Theorem~5.2 in~\cite{Hajek_1964}, together with (C1) and~(A2),
for sampling design $R$,
\begin{equation}
\label{eq:order2 expansion}
\begin{split}
\pi_{ij}-\pi_i\pi_j
&=
\pi_i\pi_j
\left\{
-\frac{1}{d_N}
(1-\pi_i)(1-\pi_j)
+
O(d_N^{-2})
\right\}\\
&=
O(n^2/(N^2d_N)),
\end{split}
\end{equation}
$\omega$-almost surely.
Therefore, together with (A3), condition (C2) follows, $\omega$-almost surely.
For condition (C3), according to Lemma 2 in~\cite{Boistard_2012}, the third order correlation
$\mathbb{E}_d(\xi_i-\pi_i)(\xi_j-\pi_j)(\xi_k-\pi_k)$
splits into terms of the form $(\pi_{ij}-\pi_i\pi_j)\pi_k$
and the term $\pi_{ijk}-\pi_i\pi_j\pi_k$.
Similar to~\eqref{eq:order2 expansion}, together with Theorem~1 in~\cite{Boistard_2012},
the latter term can be shown to be of the order $O(n^3/(N^3d_N))$,
whereas other terms are of the same order according to (C1)-(C2) and~(A2).
Again, together with (A3), condition (C3) follows, $\omega$-almost surely.
According to Proposition~1 in~\cite{Boistard_2012},
\[
|\mathbb{E}_d(\xi_i-\pi_i)(\xi_j-\pi_j)(\xi_k-\pi_k)(\xi_l-\pi_l)|
=
O(d_N^{-2}),
\quad
\text{a.s.}-\mathbb{P}_m.
\]
Hence, together with (A4), condition (C4) follows, $\omega$-almost surely.
Theorems~\ref{th:FCLT HT FN} and~\ref{th:FCLT HT F} are now immediate,
when either (HT2) holds or (HT2)-(HT4), respectively, which establishes parts~(i) and~(ii)
for the rejective sampling design $R$.
For parts~(iii) and~(iv), it can be seen that under design~$R$,
\[
\begin{split}
\frac{n}{N^2}
\ssum_{i\ne j}
\frac{\pi_{ij}-\pi_i\pi_j}{\pi_i\pi_j}
&=
-
\frac{n}{N^2}
\ssum_{i\ne j}
\frac{(1-\pi_i)(1-\pi_j)}{d_N}
+
O\left(n/d_N^2\right)\\
&=
-\frac{n}{N^2d_N}(N-n)^2
+
O(1/d_N)+O\left(n/d_N^2\right)\\
&\to\alpha,
\end{split}
\]
with (A2)-(A3) and~(A5).
Hence, Theorems~\ref{th:FCLT HJ FN} and~\ref{th:FCLT HJ F} are now immediate with $\mu_{\pi2}=-\alpha$,
when either (HT3) and (HJ2) hold or (HT3), (HJ2), and~(HJ4), respectively, which establishes parts~(iii) and~(iv)
for rejective sampling design $R$.

To extend these results to high entropy designs, we use the same approach as in~\cite{Bertail_2013}.
They use the bounded Lipschitz metric for random elements $X$ and~$Y$ on a metric space $\mathbb{D}$:
\[
d_{\mathrm{BL}}(X,Y)
=
\sup_{f\in \mathrm{BL}_1}
|\mathbb{\mathbb{E}}f(Y)-\mathbb{E}f(X)|,
\]
where $\mathrm{BL}_1$ is the class of Lipshitz functions with Lipshitz norm bounded by one.
See~\cite{van_1996}, page 73, who define the metric $d_{\mathrm{BL}}$ on the space of separable Borel measures.
Weak convergence is metrizable by this metric, i.e.,
\[
X_\alpha\rightsquigarrow X
\quad
\Leftrightarrow
\quad
\sup_{f\in \mathrm{BL}_1}
|\mathbb{\mathbb{E}}^*f(X_\alpha)-\mathbb{E}f(X)|
\to
0.
\]
{Now, consider part~(i) and let $P$ be a high entropy design.
Let $R$ be some rejective sampling design such that $D(P\|R)\to 0$.
Given the inclusion probabilities $\pi_1(P),\ldots,\pi_N(P)$, there exists a rejective sampling design $\widetilde{R}$ such that $\pi_i(\widetilde{R})=\pi_i(P)$. Note that $D(P\|\widetilde{R})\leq D(P\|R)\to 0$, according to Lemma~3 in~\cite{berger1998b}.

Consider the Horvitz-Thompson process for the design $P$
\[
\mathbb{G}_{P}^{\pi(P)}(t)
=
\sqrt{n}
\left(
\frac{1}{N}
\sum_{i=1}^N
\frac{\xi_i(P)\mathds{1}_{\{Y_i\leq t\}}}{\pi_i(P)}
-
\frac{1}{N}
\sum_{i=1}^N
\mathds{1}_{\{Y_i\leq t\}}
\right),
\]
and compare this with the same process for design $\tilde{R}$,
\[
\mathbb{G}_{\widetilde{R}}^{\pi(P)}(t)
=
\sqrt{n}
\left(
\frac{1}{N}
\sum_{i=1}^N
\frac{\xi_i(\widetilde{R})\mathds{1}_{\{Y_i\leq t\}}}{\pi_i(P)}
-
\frac{1}{N}
\sum_{i=1}^N
\mathds{1}_{\{Y_i\leq t\}}
\right).
\]
Then, because
$\mathbb{E}_d[\xi_i(P)]=\sum_{s\in \mathcal{S}_N} P(s)\delta_i(s)$,
where $\delta_i(s)=1$ when $i\in s$ and zero otherwise,
it follows that for $\mathbb{E}_df(\mathbb{G}_{P}^{\pi(P)})$, the argument inside $f$ is independent of the design $P$.
Hence, for any $f\in \mathrm{BL}_{1}$, one finds
\[
\left|
\mathbb{E}_df\left(\mathbb{G}_{P}^{\pi(P)}\right)
-
\mathbb{E}_df\left(\mathbb{G}_{\widetilde{R}}^{\pi(P)}\right)
\right|
\leq
\sum_{s\in \mathcal{P}(\mathcal{U}_N)}
|P(s)-\widetilde{R}(s)|
\leq
\sqrt{2D(P\|\widetilde{R})},
\]
using Lemma 2 in~\cite{berger1998b}.
As $|\mathbb{E}_{d,m}f(Y)-\mathbb{E}_{d,m}f(X)|
\leq
\mathbb{E}_{m}
\left|
\mathbb{E}_{d}f(Y)-\mathbb{E}_{d}f(X)
\right|$,
it follows that $d_{\mathrm{BL}_1}(\mathbb{G}_{P}^{\pi(P)},\mathbb{G}_{\widetilde{R}}^{\pi(P)})\to0$.
Because part~(i) has already been established for rejective sampling design $\widetilde{R}$, we obtain that $\mathbb{G}_{\widetilde{R}}^{\pi(P)}\to \mathbb{G}$ weakly.
Hence, $d_{\mathrm{BL}_1}(\mathbb{G}_{\widetilde{R}}^{\pi(P)},\mathbb{G})\to0$ and therefore
\[
d_{\mathrm{BL}_1}(\mathbb{G}_{P}^{\pi(P)},\mathbb{G})
\leq
d_{\mathrm{BL}_1}(\mathbb{G}_{P}^{\pi(P)},\mathbb{G}_{\widetilde{R}}^{\pi(P)})
+
d_{\mathrm{BL}_1}(\mathbb{G}_{\widetilde{R}}^{\pi(P)},\mathbb{G})
\to0
\]
which means that $\mathbb{G}_{P}^{\pi(P)}\to \mathbb{G}$ weakly.
This establishes part(i) for high entropy design~$P$.}
Parts (ii)-(iv) are obtained in the same way.
\hfill\tqed

\paragraph{Proof of Corollary~\ref{cor:fixed size}}
We first re-prove Lemma~\ref{lem:tightness HT FN} under conditions (C2$^*$)-(C4$^*$).
Because $n$ is deterministic, it can be taken out of the expectation~$\mathbb{E}_{d,m}$.
When also $\pi_1,\ldots,\pi_N$ are deterministic,
this means that the expectation $\mathbb{E}_d$ over the~$\xi_i$'s can be separated from the expectation~$\mathbb{E}_m$ over the $A_i$'s and $B_j$'s
in~\eqref{eq:decomposition}.
It follows that
\begin{equation}
\label{eq:bound sum1}
\begin{split}
&\frac{1}{N^4}
\sum_{i=1}^N\sum_{k=1}^N
\mathbb{E}_{d,m}\left[n^2\alpha_i^2\beta_k^2\right]\\
&=
\frac{n^2}{N^4}
\ssum_{(i,k)\in D_{2,N}}
\frac{\mathbb{E}_d\left[(\xi_i-\pi_i)^2(\xi_k-\pi_k)^2\right]}{\pi_i^2\pi_k^2}
p_1p_2.
\end{split}
\end{equation}
Straightforward computation shows that $\mathbb{E}_{d}(\xi_i-\pi_i)^2(\xi_k-\pi_k)^2$ equals
\[
(\pi_{ik}-\pi_i\pi_k)(1-2\pi_i)(1-2\pi_k)
+
\pi_i\pi_k(1-\pi_i)(1-\pi_k).
\]
The contribution of the last term is
\[
\frac{n^2}{N^4}
\ssum_{(i,k)\in D_{2,N}}
\frac{\pi_i\pi_k(1-\pi_i)(1-\pi_k)}{\pi_i^2\pi_k^2}
\leq
\left(\frac{n}{N^2}\sum_{i=1}^N\left(\frac{1}{\pi_i}-1\right)\right)^2
=
O(1),
\]
according to condition~(i) of Proposition~\ref{prop:FCLT HT FN}.
With (C1) and (C2$^*$), the contribution of the first term is
\begin{equation}
\label{eq:application C1*}
\begin{split}
&\frac{n^2}{N^4}
\ssum_{(i,k)\in D_{2,N}}
\frac{(\pi_{ik}-\pi_i\pi_k)(1-2\pi_i)(1-2\pi_k)}{\pi_i^2\pi_k^2}\\
&\leq
O\left(\frac{N^2}{n^2}\right)
\frac{n^2}{N^4}\cdot N
\sum_{i\ne k}
\left|\frac{\pi_{ik}-\pi_i\pi_k}{\pi_i\pi_k}\right|
=
O\left(\frac{1}{n}\right).
\end{split}
\end{equation}
This establishes~\eqref{eq:term1total FN}.

For the second (and third) summation on the right hand side of~\eqref{eq:decomposition}, we have
\[
\begin{split}
&
\frac{1}{N^4}
\left|
\sum_{i=1}^N\sum_{j\ne i}\sum_{k=1}^N
\mathbb{E}_{d,m}\left[n^2\alpha_i\alpha_j\beta_k^2\right]
\right|\\
&\leq
\frac{n^2}{N^4}
\left|
\sssum_{(i,j,k)\in D_{3,N}}
\mathbb{E}_{d}
\left[
\frac{(\xi_i-\pi_i)(\xi_j-\pi_j)(\xi_k-\pi_k)^2}{\pi_i\pi_j\pi_k^2}
\right]
\right|
p_1p_2
\end{split}
\]
We still have that $\mathbb{E}_{d}(\xi_i-\pi_i)(\xi_j-\pi_j)(\xi_k-\pi_k)^2$ equals
\[
(1-2\pi_k)
\mathbb{E}_{d}
(\xi_i-\pi_i)(\xi_j-\pi_j)(\xi_k-\pi_k)
+
\pi_k(1-\pi_k)
\mathbb{E}_{d}
(\xi_i-\pi_i)(\xi_j-\pi_j).
\]
The contribution of the last term is
\[
\begin{split}
&
\left|
\frac{n^2}{N^4}
\sssum_{(i,j,k)\in D_{3,N}}
\left(\frac{1}{\pi_k}-1\right)
\frac{\pi_{ij}-\pi_i\pi_j}{\pi_i\pi_j}
\right|\\
&\leq
\left|\frac{n}{N^2}
\ssum_{(i,j)\in D_{2,N}}
\frac{\pi_{ij}-\pi_i\pi_j}{\pi_i\pi_j}
\right|
\cdot
\frac{n}{N^2}
\sum_{k=1}^N
\left(\frac{1}{\pi_k}-1\right)=O(1),
\end{split}
\]
according to conditions (i)-(ii) of Proposition~\ref{prop:FCLT HT FN}.
From Lemma 2 in~\cite{Boistard_2012}, we have that $\mathbb{E}_{d}\xi_i-\pi_i)(\xi_j-\pi_j)(\xi_k-\pi_k)$ splits into
\begin{enumerate}
  \item
  $-(\pi_{ij}-\pi_i\pi_j)\pi_k-(\pi_{ik}-\pi_i\pi_k)\pi_j-(\pi_{jk}-\pi_j\pi_k)\pi_i$.
  \item
  $\pi_{ijk}-\pi_i\pi_j\pi_k$.
\end{enumerate}
According to (C1) and (C2$^*$), the contribution of the terms in the first case is of the order~$O(1)$
similarly to~\eqref{eq:application C1*},
whereas (C1) and (C3$^*$) yield that the contribution of the second case is also of the order~$O(1)$.
This establishes~\eqref{eq:toprove sum2 FN}.

Finally,
\[
\begin{split}
&
\frac1{N^4}\sum_{i=1}^N\sum_{j\neq i}\sum_{k=1}^N\sum_{l\neq k}
\mathbb{E}_{d,m}\left[n^2\alpha_i\alpha_j\beta_k\beta_l\right]\\
&=
\frac{n^2}{N^4}
\sum_{(i,j,k,l)\in D_{4,N}}
\frac{\mathbb{E}_{d}\left[(\xi_i-\pi_i)(\xi_j-\pi_j)(\xi_k-\pi_k)(\xi_l-\pi_l)\right]}{\pi_i\pi_j\pi_k\pi_l}p_1^2p_2^2.
\end{split}
\]
Because $0\leq p_1,p_2\leq 1$, together with (C4$^*$), we obtain~\eqref{eq:toprove sum3 FN}.
Together with~\eqref{eq:term1total FN}, \eqref{eq:toprove sum2 FN} and decomposition~\eqref{eq:decomposition},
this proves Lemma~\ref{lem:tightness HT FN}.

Furthermore, at the cost of some extra technicalities,
it can be seen that Lemma~\ref{lem:variance HT} in~\cite{Boistard_2015} holds with (C2$^*$) and conditions (i)-(ii)
from Proposition~\ref{prop:FCLT HT FN} instead of (C2).
Details can be found in the supplement B in~\cite{Boistard_2015}.
From here on, the proofs of Propositions~\ref{prop:FCLT HT FN}, \ref{prop:FCLT HT F}, \ref{prop:FCLT HJ FN},
and~\ref{prop:FCLT HJ F} remain the same.
\hfill\tqed

\bigskip

The proofs for Corollaries~\ref{cor:poverty rate HT} and~\ref{cor:poverty rate HJ}
are fairly straightforward and can be found in the supplement Ai in~\cite{Boistard_2015}.

%
%\begin{supplement}[id=suppA]
%  \sname{Supplement A}
%  \stitle{Proofs for results in the main text}
%  \slink[doi]{COMPLETED BY THE TYPESETTER}
%  \sdatatype{.pdf}
%  \sdescription{This supplement contains detailed proofs of Lemmas, Propositions and Corollaries for results in the main text, that are not present
%  in Section~\ref{sec:proofs}.}
%\end{supplement}
%
%\begin{supplement}[id=suppB]
%  \sname{Supplement B}
%  \stitle{Additional technicalities}
%  \slink[doi]{COMPLETED BY THE TYPESETTER}
%  \sdatatype{.pdf}
%  \sdescription{This supplement contains detailed proofs of some remarks and additional lemmas.}
%\end{supplement}

}

%\bibliographystyle{imsart-nameyear}
%\bibliographystyle{imsart-number}
%\bibliography{FTCL_bib}
\bibliography{FTCL_bib}

\newcommand{\etalchar}[1]{$^{#1}$}
\begin{thebibliography}{MRJM14}

\bibitem[BCC14]{Bertail_2013}
Patrice Bertail, Emilie Chautru, and St{\'e}phan Cl{\'e}men{\c{c}}on.
\newblock Empirical processes in survey sampling.
\newblock {\em Submitted, See
  also~\texttt{https://hal.archives-ouvertes.fr/hal-00989585}}, 2014.

\bibitem[BD09]{barrettdonald2009}
Garry~F. Barrett and Stephen~G. Donald.
\newblock Statistical inference with generalized {G}ini indices of inequality,
  poverty, and welfare.
\newblock {\em J. Bus. Econom. Statist.}, 27(1):1--17, 2009.

\bibitem[Ber98a]{berger1998a}
Yves~G. Berger.
\newblock Rate of convergence for asymptotic variance of the
  {H}orvitz-{T}hompson estimator.
\newblock {\em J. Statist. Plann. Inference}, 74(1):149--168, 1998.

\bibitem[Ber98b]{berger1998b}
Yves~G. Berger.
\newblock Rate of convergence to normal distribution for the
  {H}orvitz-{T}hompson estimator.
\newblock {\em J. Statist. Plann. Inference}, 67(2):209--226, 1998.

\bibitem[Ber11]{Berger_2011}
Yves~G Berger.
\newblock Asymptotic consistency under large entropy sampling designs with
  unequal probabilities.
\newblock {\em Pakistan Journal of Statistics}, 27(4):407--426, 2011.

\bibitem[BF84]{Bickel_1984}
Peter~J Bickel and David~A Freedman.
\newblock Asymptotic normality and the bootstrap in stratified sampling.
\newblock {\em The annals of statistics}, pages 470--482, 1984.

\bibitem[Bha07]{bhattacharya2007}
Debopam Bhattacharya.
\newblock Inference on inequality from household survey data.
\newblock {\em J. Econometrics}, 137(2):674--707, 2007.

\bibitem[Bil99]{Billingsley_1999}
Patrick Billingsley.
\newblock {\em Convergence of probability measures}.
\newblock Wiley Series in Probability and Statistics: Probability and
  Statistics. John Wiley \& Sons, Inc., New York, second edition, 1999.
\newblock A Wiley-Interscience Publication.

\bibitem[BLRG12]{Boistard_2012}
Hel{\`e}ne Boistard, Hendrik~P. Lopuha{\"a}, and Anne Ruiz-Gazen.
\newblock Approximation of rejective sampling inclusion probabilities and
  application to high order correlations.
\newblock {\em Electron. J. Stat.}, 6:1967--1983, 2012.

\bibitem[BLRG15]{Boistard_2015}
Hel{\`e}ne Boistard, Hendrik~P. Lopuha{\"a}, and Anne Ruiz-Gazen.
\newblock Supplement to "functional central limit theorems in survey sampling".
\newblock 2015.

\bibitem[BM11]{bhattacharyamazumder2011}
Debopam Bhattacharya and Bhaskhar Mazumder.
\newblock A nonparametric analysis of black�white differences in
  intergenerational income mobility in the united states.
\newblock {\em Quant. Econ.}, 2(3):335�379, 2011.

\bibitem[BO00]{BreidtOpsomer_2000}
F.~Jay Breidt and Jean~D. Opsomer.
\newblock Local polynomial regresssion estimators in survey sampling.
\newblock {\em Ann. Statist.}, 28(4):1026--1053, 2000.

\bibitem[BR09]{Binder_2009}
David Binder and Georgia Roberts.
\newblock {\em Handbook of Statistics 29B: Sample Surveys: Design, Methods and
  Applications.}, chapter Chapter 24: Design- and Model-Based Inference for
  Model Parameters, pages 33--54.
\newblock Elsevier, Amsterdam, 2009.

\bibitem[BS03]{bergerskinner2003}
Yves~G. Berger and Chris~J. Skinner.
\newblock Variance estimation for a low income proportion.
\newblock {\em J. Roy. Statist. Soc. Ser. C}, 52(4):457--468, 2003.

\bibitem[BS05]{bergerskinner2005}
Yves~G. Berger and Chris~J. Skinner.
\newblock A jackknife variance estimator for unequal probability sampling.
\newblock {\em J. R. Stat. Soc. Ser. B Stat. Methodol.}, 67(1):79--89, 2005.

\bibitem[BW07]{BreslowWellner_2007}
Norman~E. Breslow and Jon~A. Wellner.
\newblock Weighted likelihood for semiparametric models and two-phase
  stratified samples, with application to {C}ox regression.
\newblock {\em Scand. J. Statist.}, 34(1):86--102, 2007.

\bibitem[C{\etalchar{+}}15]{Chauvet_2015}
Guillaume Chauvet et~al.
\newblock Coupling methods for multistage sampling.
\newblock {\em The Annals of Statistics}, 43(6):2484--2506, 2015.

\bibitem[CCGL10]{Cardot_2010}
Herv{\'e} Cardot, Mohamed Chaouch, Camelia Goga, and Catherine Labru{\`e}re.
\newblock Properties of design-based functional principal components analysis.
\newblock {\em J. Statist. Plann. Inference}, 140(1):75--91, 2010.

\bibitem[CMM15]{conti2015}
Pier~Luigi Conti, Daniela Marella, and Fulvia Mecatti.
\newblock Recovering sampling distributions of statistics of finite populations
  via resampling: a predictive approach.
\newblock {\em Submitted}, 2015.

\bibitem[Con14]{conti2014}
Pier~Luigi Conti.
\newblock On the estimation of the distribution function of a finite population
  under high entropy sampling designs, with applications.
\newblock {\em Sankhya B}, 76(2):234--259, 2014.

\bibitem[Dav09]{davidson2009}
Russell Davidson.
\newblock Reliable inference for the {G}ini index.
\newblock {\em J. Econometrics}, 150(1):30--40, 2009.

\bibitem[Dd08]{dell2008}
Fabien Dell and Xavier d'Haultf{\oe}uille.
\newblock Measuring the evolution of complex indicators: Theory and application
  to the poverty rate in {F}rance.
\newblock {\em Ann. \'Econom. Statist.}, (90):259--290, 2008.

\bibitem[Dor09]{Dorfman_2009}
Alan~H Dorfman.
\newblock {\em Inference on distribution functions and quantiles}, chapter
  Chapter 36: Sample Surveys: Design, Methods and Applications. Chapter 36:
  Inference on distribution functions and quantiles, pages 371--395.
\newblock Elsevier, Amsterdam, 2009.

\bibitem[DS92]{Deville_1992}
Jean-Claude Deville and Carl-Erik S{\"a}rndal.
\newblock Calibration estimators in survey sampling.
\newblock {\em Journal of the American statistical Association},
  87(418):376--382, 1992.

\bibitem[Dud02]{dudley2002}
R.~M. Dudley.
\newblock {\em Real analysis and probability}, volume~74 of {\em Cambridge
  Studies in Advanced Mathematics}.
\newblock Cambridge University Press, Cambridge, 2002.
\newblock Revised reprint of the 1989 original.

\bibitem[Dup79]{dupacova1979}
Jitka Dupa{\v{c}}ov{\'a}.
\newblock A note on rejective sampling.
\newblock In {\em Contributions to statistics}, pages 71--78. Reidel,
  Dordrecht-Boston, Mass.-London, 1979.

\bibitem[EB13]{escobarberger2013}
Emilio~L. Escobar and Yves~G. Berger.
\newblock A jackknife variance estimator for self-weighted two-stage samples.
\newblock {\em Statist. Sinica}, 23(2):595--613, 2013.

\bibitem[FF91]{Francisco_1991}
Carol~A. Francisco and Wayne~A. Fuller.
\newblock Quantile estimation with a complex survey design.
\newblock {\em Ann. Statist.}, 19(1):454--469, 1991.

\bibitem[Ful09]{Fuller_2009}
W.A. Fuller.
\newblock {\em Sampling Statistics}.
\newblock Wiley Series in Survey Methodology. Wiley, New York, 2009.

\bibitem[GT14]{graftille2014}
Eric Graf and Yves Till\'e.
\newblock Variance estimation using linearization for poverty and social
  exclusion indicators.
\newblock {\em Survey Methodology}, 40(1):61--79, 2014.

\bibitem[H{\'a}j59]{hajek1959}
Jaroslav H{\'a}jek.
\newblock Optimum strategy and other problems in probability sampling.
\newblock {\em \v Casopis P\v est. Mat.}, 84:387--423, 1959.

\bibitem[H{\'a}j64]{Hajek_1964}
Jaroslav H{\'a}jek.
\newblock Asymptotic theory of rejective sampling with varying probabilities
  from a finite population.
\newblock {\em Ann. Math. Statist.}, 35:1491--1523, 1964.

\bibitem[IF82]{isakifuller1982}
Cary~T. Isaki and Wayne~A. Fuller.
\newblock Survey design under the regression superpopulation model.
\newblock {\em J. Amer. Statist. Assoc.}, 77(377):89--96, 1982.

\bibitem[KG98]{KornGraubard1998}
Edward~L. Korn and Barry~I. Graubard.
\newblock Variance estimation for superpopulation parameters.
\newblock {\em Statist. Sinica}, 8(4):1131--1151, 1998.

\bibitem[KR81]{Krewski_1981}
D.~Krewski and J.~N.~K. Rao.
\newblock Inference from stratified samples: properties of the linearization,
  jackknife and balanced repeated replication methods.
\newblock {\em Ann. Statist.}, 9(5):1010--1019, 1981.

\bibitem[Lin00]{lin2000}
D.~Y. Lin.
\newblock On fitting {C}ox's proportional hazards models to survey data.
\newblock {\em Biometrika}, 87(1):37--47, 2000.

\bibitem[MRJM14]{mirakhmedov2014}
S.~M. Mirakhmedov, S.~Rao~Jammalamadaka, and Ibrahim~B. Mohamed.
\newblock On {E}dgeworth expansions in generalized urn models.
\newblock {\em J. Theoret. Probab.}, 27(3):725--753, 2014.

\bibitem[OAB15]{oguzalperberger2015}
M.~Oguz-Alper and Y.~G. Berger.
\newblock Variance estimation of change of poverty based upon the turkish
  eu-silc survey.
\newblock {\em Journal of Official Statistics}, 31(2):155--175, 2015.

\bibitem[PS09a]{Pfeffermann_2009}
Danny Pfeffermann and Michail Sverchkov.
\newblock Inference under informative sampling.
\newblock {\em Handbook of Statistics}, 29:455--487, 2009.

\bibitem[PS09b]{praskovasen2009}
Z.~Pr\'askov\'a and P.K. Sen.
\newblock {\em Asymptotic in finite population sampling}, pages 489--522.
\newblock Elsevier, 2009.

\bibitem[PW93]{Praestgaard_1993}
Jens Pr{\ae}stgaard and Jon~A. Wellner.
\newblock Exchangeably weighted bootstraps of the general empirical process.
\newblock {\em Ann. Probab.}, 21(4):2053--2086, 1993.

\bibitem[RBSK05]{Rubin-Bleuer_Kratina_2005}
Susana Rubin-Bleuer and Ioana Schiopu~Kratina.
\newblock On the two-phase framework for joint model and design-based
  inference.
\newblock {\em Ann. Statist.}, 33(6):2789--2810, 2005.

\bibitem[RRZ94]{Robins_1994}
James~M Robins, Andrea Rotnitzky, and Lue~Ping Zhao.
\newblock Estimation of regression coefficients when some regressors are not
  always observed.
\newblock {\em Journal of the American statistical Association},
  89(427):846--866, 1994.

\bibitem[Sil86]{silverman1986}
B.~W. Silverman.
\newblock {\em Density estimation for statistics and data analysis}.
\newblock Monographs on Statistics and Applied Probability. Chapman \& Hall,
  London, 1986.

\bibitem[SW13]{Saegusa_2013}
Takumi Saegusa and Jon~A. Wellner.
\newblock Weighted likelihood estimation under two-phase sampling.
\newblock {\em Ann. Statist.}, 41(1):269--295, 2013.

\bibitem[Tho97]{Thompson_1997}
M.~E. Thompson.
\newblock {\em Theory of sample surveys}, volume~74 of {\em Monographs on
  Statistics and Applied Probability}.
\newblock Chapman \& Hall, London, 1997.

\bibitem[vdV98]{vandervaart1998}
A.~W. van~der Vaart.
\newblock {\em Asymptotic statistics}, volume~3 of {\em Cambridge Series in
  Statistical and Probabilistic Mathematics}.
\newblock Cambridge University Press, Cambridge, 1998.

\bibitem[vdVW96]{van_1996}
Aad~W. van~der Vaart and Jon~A. Wellner.
\newblock {\em Weak convergence and empirical processes}.
\newblock Springer Series in Statistics. Springer-Verlag, New York, 1996.
\newblock With applications to statistics.

\bibitem[V{\'{\i}}{\v{s}}79]{visek1979}
Jan~{\'A}mos V{\'{\i}}{\v{s}}ek.
\newblock Asymptotic distribution of simple estimate for rejective, {S}ampford
  and successive sampling.
\newblock In {\em Contributions to statistics}, pages 263--275. Reidel,
  Dordrecht-Boston, Mass.-London, 1979.

\bibitem[Wan12]{Wang_2012}
Jianqiang~C. Wang.
\newblock Sample distribution function based goodness-of-fit test for complex
  surveys.
\newblock {\em Comput. Statist. Data Anal.}, 56(3):664--679, 2012.

\end{thebibliography}
\noindent
H\'el\`ene Boistard\\
Toulouse School of Economics\\
21 all\'ee de Brienne\\
31000 Toulouse, France\\
e-mail: helene@boistard.fr\\

\noindent
Hendrik P. Lopuha\"a\\
Delft Institute of Applied Mathematics\\
Delft University of Technology\\
Delft, The Netherlands\\
e-mail: h.p.lopuhaa@tudelft.nl\\

\noindent
Anne Ruiz-Gazen\\
Toulouse School of Economics\\
21 all\'ee de Brienne\\
31000 Toulouse, France\\
e-mail: anne.ruiz-gazen@tse-fr.eu\\

\newpage

\newpage

\section*{Supplement}

\appendix
\section{Proofs for results in the main text}
\label{sec:supplemental material}

{\paragraph{Proof of Lemma~\ref{lem:fidis HT FN}}
We will use the Cram\'er-Wold device.
Note that any linear combination
\begin{equation}
\label{eq:lincomb HT FN}
a_1\sqrt{n}
\left\{
\FHT(t_1)-\FN(t_1)
\right\}
+
\cdots
+
a_k
\sqrt{n}
\left\{
\FHT(t_k)-\FN(t_k)
\right\}
\end{equation}
can be written as
\begin{equation}
\label{eq:fidis HT FN}
\sqrt{n}
\left\{
\frac1N\sum_{i=1}^N
\frac{\xi_i}{\pi_i}V_{ik}
-
\frac1N\sum_{i=1}^N
V_{ik}
\right\},
\end{equation}
where
\begin{equation}
\label{eq:def Vik}
V_{ik}
=
a_1\mathds{1}_{\{Y_i\leq t_1\}}
+\cdots+
a_k\mathds{1}_{\{Y_i\leq t_k\}}
=
\mathbf{a}_k^t\mathbf{Y}_{ik}
\end{equation}
with $\mathbf{Y}_{ik}^t=(\mathds{1}_{\{Y_i\leq t_1\}},\ldots,\mathds{1}_{\{Y_i\leq t_k\}})$ and $\mathbf{a}_k^t=(a_1, \dots, a_k)$.
For the corresponding design-based variance, we have
\begin{equation}
\label{eq:conv SN}
\begin{split}
nS_N^2
&=
\frac{n}{N^2}
\sum_{i=1}^N\sum_{j=1}^N
\frac{\pi_{ij}-\pi_i\pi_j}{\pi_i\pi_j}
V_{ik}V_{jk}\\
&=
\mathbf{a}_k^t
\left(
\frac{n}{N^2}
\sum_{i=1}^N\sum_{j=1}^N
\frac{\pi_{ij}-\pi_i\pi_j}{\pi_i\pi_j}
\mathbf{Y}_{ik}\mathbf{Y}_{jk}^t
\right)
\mathbf{a}_k
\to
\mathbf{a}_k^t\mathbf{\Sigma}_k^{\mathrm{HT}}\mathbf{a}_k,
\end{split}
\end{equation}
$\omega$-almost surely, according to (HT2),
where $\mathbf{\Sigma}_k^{\mathrm{HT}}$ can obtained from~\eqref{eq:HT2 alternative moment}.
Together with (HT1), it follows that~\eqref{eq:lincomb HT FN} converges in distribution to a mean zero normal random variable
with variance $\mathbf{a}_k^t\mathbf{\Sigma}_k^{\mathrm{HT}}\mathbf{a}_k$.
We conclude that~\eqref{eq:lincomb HT FN} converges in distribution to
$a_1N_1+\cdots+a_kN_k$, where $(N_1,\ldots,N_k)$ has a $k$-variate mean zero normal distribution
with covariance matrix $\mathbf{\Sigma}_k^{\mathrm{HT}}$.
According to the Cram\'er-Wold device this proves the lemma.
\hfill\tqed}

\paragraph{Proof of Proposition~\ref{prop:FCLT HT FN}}
The proof is similar to that of Theorem~\ref{th:FCLT HT FN}.
First consider the case of uniform $Y_i$'s with $F(t)=t$.
We only have to verify the weak convergence of the finite dimensional projections
of the process $\mathbb{X}_N=\sqrt{n}(\FHT-\FN)$.
Consider~\eqref{eq:lincomb HT FN} represented as in~\eqref{eq:fidis HT FN}.
From~(HT1) and Lemma~\ref{lem:variance HT}(ii) in~\cite{Boistard_2015} we conclude that~\eqref{eq:lincomb HT FN} converges in distribution to a mean zero normal random variable
with variance
\[
\begin{split}
\sigma_{\mathrm{HT}}^2
&=
\mu_{\pi1}
\mathbb{E}_m\left[V_{1k}^2\right]
+
\mu_{\pi2}
\left(\mathbb{E}_m\left[V_{1k}\right]\right)^2\\
&=
\mu_{\pi1}
\mathbf{a}_k^t\mathbb{E}_m\left[\mathbf{Y}_{1k}\mathbf{Y}_{1k}^t\right]\mathbf{a}_k
+
\mu_{\pi2}
\mathbf{a}_k^t\left(\mathbb{E}_m\mathbf{Y}_{1k}\right)\left(\mathbb{E}_m\mathbf{Y}_{1k}\right)^t\mathbf{a}_k
=
\mathbf{a}_k^t\mathbf{\Sigma}_k\mathbf{a}_k,
\end{split}
\]
where $\mathbf{\Sigma}_k$ is the $k\times k$-matrix
with $(q,r)$-element equal to $\mu_{\pi_1}(t_q\wedge t_r)+\mu_{\pi2}t_qt_r$.
%\[
%\mu_{\pi_1}(t_q\wedge t_r)+\mu_{\pi2}t_qt_r.
%\]
We conclude that~\eqref{eq:lincomb HT FN} converges in distribution to
$a_1N_1+\cdots+a_kN_k$, where $(N_1,\ldots,N_k)$ has a $k$-variate mean zero normal distribution
with covariance matrix $\mathbf{\Sigma}_k$.
As in the proof of Lemma~\ref{lem:fidis HT FN}, by means of the Cram\'er-Wold device this establishes the limit distribution
of $(\mathbb{X}_N(t_1),\ldots,\mathbb{X}_N(t_k))$, which is the same that of the vector
$(\mathbb{G}^{\mathrm{HT}}(t_1),\ldots,\mathbb{G}^{\mathrm{HT}}(t_k))$,
where $\mathbb{G}^{\mathrm{HT}}$ is a mean zero Gaussian process with covariance function
$
\mathbb{E}_{d,m}\mathbb{G}^{\mathrm{HT}}(s)\mathbb{G}^{\mathrm{HT}}(t)
=
\mu_{\pi1}(s\wedge t)+\mu_{\pi2}st.
$
%\[
%\mathbb{E}_{d,m}\mathbb{G}^{\mathrm{HT}}(s)\mathbb{G}^{\mathrm{HT}}(t)
%=
%\mu_{\pi1}(s\wedge t)+\mu_{\pi2}st.
%\]
From here on, the proof is completely the same as that of Theorem~\ref{th:FCLT HT FN}.
\hfill
\tqed

{\paragraph{Proof of Lemma~\ref{lem:HT CLT}}
We decompose as follows
\[
\begin{split}
\frac{1}{S_N}
\left(
\frac{1}{N}
\sum_{i=1}^N
\frac{\xi_iV_i}{\pi_i}
-
\mu_{V}
\right)
=
&\frac{1}{S_N}
\left(
\frac{1}{N}
\sum_{i=1}^N
\frac{\xi_iV_i}{\pi_i}
-
\frac{1}{N}
\sum_{i=1}^N
V_i
\right)\\
&
+
\frac{1}{\sqrt{n}S_N}
\times
\frac{\sqrt{n}}{\sqrt{N}}\times \sqrt{N}
\left(
\frac{1}{N}
\sum_{i=1}^N
V_i
-
\mu_{V}
\right).
\end{split}
\]
According to (HT3), the central limit theorem, Slutsky's theorem, and the fact that $nS_N^2\to\sigma_{\mathrm{HT}}^2>0$ in probability,
\begin{equation}
\label{eq:RBKterm2}
\frac{1}{\sqrt{n}S_N}
\times
\frac{\sqrt{n}}{\sqrt{N}}\times \sqrt{N}
\left(
\frac{1}{N}
\sum_{i=1}^N
V_i
-
\mu_{V}
\right)
\to
N(0,\lambda\sigma^2_V/\sigma^2_{\mathrm{HT}}),
\end{equation}
in distribution under $\mathbb{P}_m$,
whereas, thanks to (HT1),
\begin{equation}
\label{eq:RBKterm1}
\frac{1}{S_N}
\left(
\frac{1}{N}
\sum_{i=1}^N
\frac{\xi_iV_i}{\pi_i}
-
\frac{1}{N}
\sum_{i=1}^N
V_i
\right)
\to
N(0,1),
\quad
\omega-\text{a.s.,}
\end{equation}
in distribution under $\mathbb{P}_d$.
Since the latter limit distribution does not depend on $\omega$, we can apply Theorem~5.1(iii)
from~\cite{Rubin-Bleuer_Kratina_2005}.
It follows that
\[
\frac{1}{S_N}
\left(
\frac{1}{N}
\sum_{i=1}^N
\frac{\xi_iV_i}{\pi_i}
-
\mu_{V}
\right)
\to
N\left(0,1+\lambda\sigma^2_V/\sigma^2_{\mathrm{HT}}\right),
\]
in distribution under $\mathbb{P}_{d,m}$.
Together with $nS_N^2\to\sigma_{HT}^2$ in probability, this implies that the random variable in~\eqref{eq:HT estimator}
converges to a mean zero normal random variable with variance
$\sigma_{\mathrm{HT}}^2+ \lambda \sigma_V^2$.
\hfill\tqed}

%%%%
\paragraph{Proof of Lemma~\ref{lem:fidis HT F}}
We will use the Cram\'er-Wold device.
To this end, we determine the limit distribution of
$a_1\mathbb{X}_N^F(t_1)+\cdots+a_k\mathbb{X}_N^F(t_k)$, for $a_1,\ldots,a_k\in \mathbb{R}$ fixed
and $\mathbf{a}_k^t=(a_1,\ldots,a_k)\ne (0,\ldots,0)$.
As in the proof of Lemma~\ref{lem:fidis HT FN}, we consider
\begin{equation}
\label{eq:lincomb HT F}
a_1\mathbb{X}_N^F(t_1)+\cdots+a_k\mathbb{X}_N^F(t_k)
=
\sqrt{n}
\left(
\frac{1}{N}
\sum_{i=1}^N
\frac{\xi_i}{\pi_i}
V_{ik}-\mu_k
\right),
\end{equation}
where $V_{ik}$ is defined in~\eqref{eq:def Vik}.
We want to apply Lemma~\ref{lem:HT CLT}.
As in~\eqref{eq:conv SN},
\begin{equation}
\label{eq:sigmatilde}
nS_N^2
\to
\mathbf{a}_k^t\mathbf{\Sigma}_k^{\mathrm{HT}}\mathbf{a}_k,
\quad
\omega-\text{a.s.},
\end{equation}
where $\mathbf{a}_k^t\mathbf{\Sigma}_k^{\mathrm{HT}}\mathbf{a}_k>0$, thanks to (HT4).
This means that, according to Lemma~\ref{lem:HT CLT}, the right hand side of~\eqref{eq:lincomb HT F}
converges in distribution under $\mathbb{P}_{d,m}$ to a mean zero normal random variable with variance
\[
\mathbf{a}_k^t\mathbf{\Sigma}_k^{\mathrm{HT}}\mathbf{a}_k+\lambda\left\{\mathbb{E}_{m}[V_{1k}^2]-\left(\mathbb{E}_{m}[V_{1k}]\right)^2\right\}
=
\mathbf{a}_k^t\SigmaHTF\mathbf{a}_k,
\]
where
\begin{equation}
\label{eq:SigmaHTF}
\SigmaHTF
=
\mathbf{\Sigma}_k^{\mathrm{HT}}
+
\lambda
\mathbf{\Sigma}_F.
\end{equation}
We conclude that~\eqref{eq:lincomb HT F} converges in distribution to $a_1N_1+\cdots+a_kN_k$, where $(N_1,\ldots,N_k)$ has a
mean zero $k$-variate
normal distribution with covariance matrix $\SigmaHTF$.
By the Cram\'er-Wold device, this proves the lemma.
\hfill
\tqed

\paragraph{Proof of Lemma~\ref{lem:positive definite}}
Without loss of generality we may assume $0<F(t_1)<\cdots<F(t_k)< 1$,
since we can permute the rows and columns of $\mathbf{M}$ without changing the determinant.
For the entries of $\mathbf{M}$ we can distinguish three situations:
\begin{enumerate}
\item
if $1\leq j<i\leq k$, then $M_{ij}=aF(t_j)-bF(t_i)F(t_j)$
\item
if $1\leq i=j\leq k$, then $M_{ij}=aF(t_i)-bF(t_i)^2$
\item
if $1\leq i<j\leq k$, then $M_{ij}=aF(t_i)-bF(t_i)F(t_j)$.
\end{enumerate}
Now, for $2\leq i\leq k$, multiply the $i$-th row by $F(t_1)/F(t_i)$.
This changes the determinant with a factor $F(t_1)^{k-1}/F(t_2)\cdots F(t_k)>0$,
and as a result, all entries in column $j$, at positions $1\leq i\leq j\leq k$, are the same: $aF(t_1)-bF(t_1)F(t_j)$.
Hence, if we subtract row-2 from row-1, then row-3 from row-2, \ldots, and then row-$k$ from row-$(k-1)$,
we get a new matrix $\mathbf{M}'$ with a right-upper triangle consisting of zero's and a main diagonal
with elements $M'_{ii}=aF(t_1)-aF(t_1)F(t_i)/F(t_{i+1})$, if $1\leq i\leq k-1$, and $M'_{kk}=aF(t_1)-bF(t_1)F(t_k)$.
It follows that
\[
\begin{split}
&\text{det}(\mathbf{M})
=
\frac{F(t_2)\cdots F(t_k)}{F(t_1)^{k-1}}
\text{det}(\mathbf{M}')\\
&=
a^{k-1}F(t_1)
(F(t_2)-F(t_1))\cdots(F(t_k)-F(t_{k-1}))
(a-bF(t_k))>0,
\end{split}
\]
since $a>0$, $0<F(t_1)<\cdots<F(t_k)< 1$, and $a-bF(t_k)>a-b\ge 0$.
\hfill
\tqed

\paragraph{Proof of Lemma~\ref{lem:fidis HT F deterministic}}
The proof is similar to that of Lemma~\ref{lem:fidis HT F}.
We determine the limit distribution of~\eqref{eq:lincomb HT F}.
Note that without loss of generality we can assume that
$0\leq F(t_1)\leq \cdots\leq F(t_k)\leq 1$.
In contrast with the proof of Lemma~\ref{lem:fidis HT F},
we now have to distinguish between several cases.

We first consider the situation where all $F(t_i)$'s are distinct and such that $0<F(t_i)<1$.
From (HT1) and Lemma~\ref{lem:variance HT}(ii) we conclude that
\[
nS_N^2
\to
\sigma_{\mathrm{HT}}^2
=
\mu_{\pi1}\mathbb{E}_{m}[V_{1k}^2]+\mu_{\pi2}\left(\mathbb{E}_{m}[V_{1k}]\right)^2
=
\mathbf{a}_k^t\mathbf{\Sigma}_k\mathbf{a}_k,
\]
where
\begin{equation}
\label{eq:SigmaHTFN}
\mathbf{\Sigma}_k
=
\Big(
\mu_{\pi1}F(t_q\wedge t_r)+\mu_{\pi2}F(t_q)F(t_r)
\Big)_{q,r=1}^k.
\end{equation}
First note that
\[
\begin{split}
\mu_{\pi1}+\mu_{\pi2}
&=
\lim_{N\to\infty}
\frac{n}{N^2}\sum_{i=1}^N\sum_{j=1}^N\frac{\pi_{ij}-\pi_i\pi_j}{\pi_i\pi_j}
=
\lim_{N\to\infty}
\frac{n}{N^2}
\text{Var}
\left(
\sum_{i=1}^N
\frac{\xi_i}{\pi_i}
\right)
\geq0.
\end{split}
\]
Therefore, together with condition~(i) we can apply Lemma~\ref{lem:positive definite} with $a=\mu_{\pi1}$ and $b=-\mu_{\pi2}$.
It follows that~$\mathbf{\Sigma}_k$ is positive definite,
so that $\sigma_{\mathrm{HT}}^2>0$.
This means that, according to Lemma~\ref{lem:HT CLT}, the right hand side of~\eqref{eq:lincomb HT F}
converges in distribution under $\mathbb{P}_{d,m}$ to a mean zero normal random variable with variance
$
(\mu_{\pi1}+\lambda)\mathbb{E}_{m}[V_{1k}^2]+(\mu_{\pi2}-\lambda)\left(\mathbb{E}_{m}[V_{1k}]\right)^2
=
\mathbf{a}_k^t\SigmaHTF\mathbf{a}_k$,
where
\begin{equation}
\label{eq:SigmaHTFprop}
\begin{split}
\SigmaHTF
=
\Big(
(\mu_{\pi1}+\lambda)F(t_q\wedge t_r)+(\mu_{\pi2}-\lambda)F(t_q)F(t_r)
\Big)_{q,r=1}^k.
\end{split}
\end{equation}
We conclude that~\eqref{eq:lincomb HT F} converges in distribution to $a_1N_1+\cdots+a_kN_k$, where $(N_1,\ldots,N_k)$ has a
mean zero $k$-variate
normal distribution with covariance matrix $\SigmaHTF$.
By means of the Cram\'er-Wold device, this proves the lemma for the case that $0<F(t_1)<\cdots<F(t_k)<1$.

The case that the $F(t_i)$'s are not all distinct, but still satisfy $0<F(t_i)<1$,
can be reduced to the case where all $F(t_i)$'s are distinct.
This can be seen as follows.
For simplicity, suppose $F(t_1)=\cdots=F(t_m)=F(t_0)$, with $0<F(t_0)<F(t_{m+1})<\cdots<F(t_k)<1$.
Then we can write~\eqref{eq:lincomb HT F} as
\begin{equation}
\label{eq:lin comb a0}
a_0\mathbb{X}_N^F(t_0)+a_{m+1}\mathbb{X}_N^F(t_{m+1})+\cdots+a_k\mathbb{X}_N^F(t_k),
\end{equation}
where $a_0=a_1+\cdots+a_m$.
As before, with (HT4) and Lemma~\ref{lem:positive definite},
it follows from Lemma~\ref{lem:HT CLT} that~\eqref{eq:lin comb a0} converges in distribution
to a mean zero normal random variable with variance $\mathbf{a}_0^t\mathbf{\Sigma}_0^F\mathbf{a}_0$, where
$\mathbf{a}_0=(a_0,a_{m+1},\ldots,a_k)^t$ and
\[
\mathbf{\Sigma}_0^F
=
\gamma_{\pi1}\mathbb{E}_{m}[\mathbf{Y}_{0}\mathbf{Y}_{0}^t]
+
(\gamma_{\pi2}-\lambda)\left(\mathbb{E}_{m}[\mathbf{Y}_{0}]\right)\left(\mathbb{E}_{m}[\mathbf{Y}_{0}]\right)^t,
\]
with $\mathbf{Y}_0=(\mathds{1}_{\{Y_i\leq t_0\}},\mathds{1}_{\{Y_i\leq t_{m+1}\}},\ldots,\mathds{1}_{\{Y_i\leq t_k\}})^t$.
However, note that
\[
\begin{split}
\mathbf{a}_0^t\mathbf{Y}_{0}
&=
(a_1+\cdots+a_m)\mathds{1}_{\{Y_i\leq t_0\}}
+
a_{m+1}\mathds{1}_{\{Y_i\leq t_{m+1}\}}+\cdots+a_k\mathds{1}_{\{Y_i\leq t_k\}}\\
&=
a_1\mathds{1}_{\{Y_i\leq t_1\}}+\cdots+a_k\mathds{1}_{\{Y_i\leq t_k\}}
=
\mathbf{a}_k^t\mathbf{Y}_{1k},
\end{split},
\]
where $\mathbf{a}_k=(a_1,\ldots,a_k)^t$ and
$\mathbf{Y}_{1k}=(\mathds{1}_{\{Y_i\leq t_1\}},\ldots,\mathds{1}_{\{Y_i\leq t_k\}})^t$, as before.
This means that $\mathbf{a}_0^t\mathbf{\Sigma}_0^F\mathbf{a}_0=\mathbf{a}_k^t\SigmaHTF\mathbf{a}_k$,
with $\SigmaHTF$ from~\eqref{eq:SigmaHTF}.
It follows that~\eqref{eq:lincomb HT F} converges in distribution to $a_1N_1+\cdots+a_kN_k$, where $(N_1,\ldots,N_k)$ has a
mean zero $k$-variate
normal distribution with covariance matrix $\SigmaHTF$.
By means of the Cram\'er-Wold device, this proves the lemma for the case
$F(t_1)=\cdots=F(t_m)=F(t_0)<F(t_{m+1})<\cdots<F(t_k)<1$.
The argument is the same for other cases with multiple $F(t_i)\in(0,1)$ being equal to each other.

Next, consider the case $F(t_1)=0$.
In this case, $\mathds{1}_{\{Y_i\leq t_1\}}=0$ with probability one.
This means that the summation on the left hand side of~\eqref{eq:lincomb HT F} reduces to
$a_2\mathbb{X}_N^F(t_2)+\cdots+a_k\mathbb{X}_N^F(t_k)$ and
\begin{equation}
\label{eq:SigmaHTFN=0}
\SigmaHTFN
=
\left(
  \begin{array}{cc}
    0 & \begin{array}{ccc}
                      0 & \cdots & 0
                    \end{array}
     \\
    \begin{array}{c}
      0 \\
      \vdots \\
      0
    \end{array}
    & \mathbf{\Sigma}_{\mathrm{HT},k-1}
 \\
  \end{array}
\right),
\end{equation}
where $\mathbf{\Sigma}_{\mathrm{HT},k-1}$ is the matrix in~\eqref{eq:SigmaHTFN}
based on $0<F(t_2)<\cdots<F(t_k)<1$.
When $\mathbf{a}_{k-1}^t=(a_2,\ldots,a_k)\ne (0,\ldots,0)$, then
\[
\sigma_{\mathrm{HT}}^2
=
\mathbf{a}_k^t\SigmaHTF\mathbf{a}_k
=
\mathbf{a}_{k-1}^t\mathbf{\Sigma}_{\mathrm{HT},k-1}\mathbf{a}_{k-1}
>0,
\]
because $\mathbf{\Sigma}_{\mathrm{HT},k-1}$ is positive definite,
due to (HT4) and Lemma~\ref{lem:positive definite}.
This allows application of Lemma~\ref{lem:HT CLT} to~\eqref{eq:lincomb HT F}.
As in the previous cases, we conclude that~\eqref{eq:lincomb HT F}
converges in distribution to $a_1N_1+\cdots+a_kN_k$, where $(N_1,\ldots,N_k)$ has a
mean zero $k$-variate
normal distribution with covariance matrix $\SigmaHTF$ given by~\eqref{eq:SigmaHTF}.
When $\mathbf{a}_{k}^t=(a_1,0,\ldots,0)$, with $a_1\ne0$, then both~\eqref{eq:lincomb HT F}
and $a_1N_1+\cdots+a_kN_k$ are equal to zero.
According to the Cram\'er-Wold device, this proves the lemma for the case $F(t_k)=0$.

It remains to consider the case $F(t_k)=1$.
In this case, the $(k,k)$-th element of the matrix $\SigmaHTFN$
in~\eqref{eq:SigmaHTFN} is equal to $\mu_{\pi1}+\mu_{\pi2}$.
We distinguish between $\mu_{\pi1}+\mu_{\pi2}=0$ and $\mu_{\pi1}+\mu_{\pi2}>0$.
In the latter case, from the proof of Lemma~\ref{lem:positive definite} we find that
$\SigmaHTFN$ has determinant
\[
\mu_{\pi1}^{k-1}F(t_1)
\prod_{i=2}^k
(F(t_i)-F(t_{i-1}))
(\mu_{\pi1}+\mu_{\pi2})>0,
\]
using~(HT4) and $0<F(t_{1})<\cdots<F(t_{k-1})<F(t_k)=1$.
This allows application of Lemma~\ref{lem:HT CLT} to~\eqref{eq:lincomb HT F}.
As before, we conclude that~\eqref{eq:lincomb HT F} converges in distribution to $a_1N_1+\cdots+a_kN_k$, where $(N_1,\ldots,N_k)$ has a
$k$-variate mean zero
normal distribution with covariance matrix $\SigmaHTF$
from~\eqref{eq:SigmaHTF}.
According to the Cram\'er-Wold device, this proves the lemma for the case
$F(t_k)=1$ and $\mu_{\pi1}+\mu_{\pi2}>0$.

Next, consider the case $F(t_k)=1$ and $\mu_{\pi1}+\mu_{\pi2}=0$.
This means
\begin{equation}
\label{eq:SigmaHTF=1}
\SigmaHTFN
=
\left(
  \begin{array}{cc}
    \mathbf{\Sigma}_{\mathrm{HT},k-1} & \begin{array}{c}
      0 \\
      \vdots \\
      0
    \end{array}\\

    \begin{array}{ccc}0 & \cdots & 0\end{array}
    & 0
 \\
  \end{array}
\right),
\end{equation}
where $\mathbf{\Sigma}_{\mathrm{HT},k-1}$
is the matrix in~\eqref{eq:SigmaHTFN} corresponding to $0<F(t_1)<\cdots<F(t_{k-1})<1$.
When $\mathbf{a}_{k-1}^t=(a_1,\ldots,a_{k-1})\ne (0,\ldots,0)$, then
\[
\sigma_{\mathrm{HT}}^2
=
\mathbf{a}_k^t\SigmaHTFN\mathbf{a}_k
=
\mathbf{a}_{k-1}^t \mathbf{\Sigma}_{\mathrm{HT},k-1}\mathbf{a}_{k-1}
>0,
\]
because $ \mathbf{\Sigma}_{\mathrm{HT},k-1}$ is positive definite,
due to (HT4) and Lemma~\ref{lem:positive definite}.
This allows application of Lemma~\ref{lem:HT CLT} to~\eqref{eq:lincomb HT F}.
As in the previous cases, we conclude that~\eqref{eq:lincomb HT F}
converges in distribution to $a_1N_1+\cdots+a_kN_k$, where $(N_1,\ldots,N_k)$ has a $k$-variate
mean zero normal distribution with covariance matrix $\SigmaHTF$ given by~\eqref{eq:SigmaHTF}.
When $\mathbf{a}_{k}^t=(0,\ldots,0,a_k)$, with $a_k\ne0$, then
$a_1N_1+\cdots+a_kN_k=0$ and
\[
a_1\mathbb{X}_N^F(t_1)+\cdots+a_k\mathbb{X}_N^F(t_k)
=
a_k
\sqrt{n}
\left(
\frac{1}{N}
\sum_{i=1}^N
\frac{\xi_i}{\pi_i}
-
1\right).
\]
converges to zero in probability.
The latter follows from the fact that according to (HT1) and Lemma~\ref{lem:variance HT}, we have that
\begin{equation}
\label{eq:CLT Nhat F2}
\sqrt{n}
\left(
\frac{1}{N}
\sum_{i=1}^N
\frac{\xi_i}{\pi_i}
-
1\right)
\to
N(0,\mu_{\pi1}+\mu_{\pi2}),
\end{equation}
in distribution under $\mathbb{P}_{d,m}$.
According to the Cram\'er-Wold device, this proves the lemma for the case $F(t_k)=1$ and $\mu_{\pi1}+\mu_{\pi2}=0$.
Finally, the argument for the case that $F(t_1)=0$ and $F(t_k)=1$ simultaneously,
either with or without repeated among the $F(t_i)$'s,
is completely similar.
This finishes the proof.
\hfill
\tqed

\paragraph{Proof of Proposition~\ref{prop:FCLT HT F}}
The proof is similar to that of Theorem~\ref{th:FCLT HT F}.
Tightness is obtained in the same way and the convergence of finite dimensional projections
is provided by Lemma~\ref{lem:fidis HT F deterministic}.
The theorem now follows from Theorem~13.5 in~\cite{Billingsley_1999} for the case that the $Y_i$'s are uniformly distributed on~$[0,1]$.
Next, this is extended to $Y_i$'s with a general c.d.f.~$F$ in the same way as in the proof of Theorem~\ref{th:FCLT HT FN}.
\hfill\tqed

\paragraph{Proof of Proposition~\ref{prop:FCLT HJ FN}}
The proof is similar to that of Theorem~\ref{th:FCLT HJ FN}.
We find that the limit behavior of $\sqrt{n}(\FHJ-\FN)$ is the same as that of the process $\mathbb{Y}_N$ defined in~\eqref{def:Y_N}.
When we first consider the case of uniform $Y_i$'s with $F(t)=t$, tightness of the process $\mathbb{Y}_N$ follows in the same way as in the proof of
Theorem~\ref{th:FCLT HJ FN}.
It remains to establish weak convergence of the finite dimensional projections~\eqref{eq:fidis Y_N}.
This can be done in the same way as in the proof of Proposition~\ref{prop:FCLT HT FN}, but this time with
\[
V_{ik}=a_1\left(\mathds{1}_{\{Y_i\leq t_1\}}-t_1)+\cdots+a_k(\mathds{1}_{\{Y_i\leq t_k\}})-t_k\right).
\]
From~(HT1) and Lemma~\ref{lem:variance HT}(i) we conclude that~\eqref{eq:lincomb Y_N} converges in distribution to a mean zero normal random variable
with variance
\[
\sigma_{\mathrm{HT}}^2
=
\mu_{\pi1}
\mathbb{E}_m\left[V_{1k}^2\right]
=
\mathbf{a}_k^t\widetilde{\mathbf{\Sigma}}_k\mathbf{a}_k,
\]
where $\widetilde{\mathbf{\Sigma}}_k$ is the $k\times k$-matrix
with $(q,r)$-element equal to $\mu_{\pi_1}(t_q\wedge t_r-t_qt_r)$.
We conclude that~\eqref{eq:lincomb Y_N} converges in distribution to
$a_1N_1+\cdots+a_kN_k$, where $(N_1,\ldots,N_k)$ has a $k$-variate mean zero normal distribution
with covariance matrix $\widetilde{\mathbf{\Sigma}}_k$.
By means of the Cram\'er-Wold device this establishes the limit distribution
of~\eqref{eq:fidis Y_N}, which is the same as that of the vector
$(\mathbb{G}^{\mathrm{HJ}}(t_1),\ldots,\mathbb{G}^{\mathrm{HJ}}(t_k))$,
where $\mathbb{G}^{\mathrm{HJ}}$ is a mean zero Gaussian process with covariance function
\[
\mathbb{E}_{d,m}\mathbb{G}^{\mathrm{HJ}}(s)\mathbb{G}^{\mathrm{HJ}}(t)
=
\mu_{\pi1}
\left(
s\wedge t-st
\right).
\]
From here on, the proof is completely the same as that of Theorem~\ref{th:FCLT HJ FN}.
\hfill
\tqed

{\paragraph{Remainder of the proof of Theorem~\ref{th:FCLT Lumley}}
It remains to prove weak convergence of the finite dimensional projections
\begin{equation}
\label{eq:fidis lumley}
\big(\mathbb{G}^{\pi}_N(t_1),\ldots,\mathbb{G}^{\pi}_N(t_k)\big).
\end{equation}
To this end we apply the Cram\'er-Wold device and consider
linear combinations
\begin{equation}
\label{eq:lincomb Lumley}
a_1\GL(t_1)+\cdots+a_k\GL(t_k)
=
\frac{\sqrt{n}}{N}
\sum_{i=1}^N
\frac{\xi_i}{\pi_i}
V_{ik}.
\end{equation}
Convergence of~\eqref{eq:lincomb Lumley}, is obtained completely similar to that of~\eqref{eq:lincomb HT F} in Lemma~\ref{lem:fidis HT F}, but this time with
\[
V_{ik}=a_1\left(\mathds{1}_{\{Y_i\leq t_1\}}-t_1)+\cdots+a_k(\mathds{1}_{\{Y_i\leq t_k\}})-t_k\right),
\]
and $\mu_k=0$.
Using the fact that (HJ4) allows the use of Lemma~\ref{lem:HT CLT},
one can deduce that~\eqref{eq:lincomb Lumley} converges in distribution under~$\mathbb{P}_{d,m}$
to $a_1N_1+\cdots+a_kN_k$, where $(N_1,\ldots,N_k)$ has a $k$-variate
normal distribution with covariance matrix
$\mathbf{\Sigma}^{\mathrm{\pi}}
=
\mathbf{\Sigma}_k^{\mathrm{HJ}}
+
\lambda \mathbf{\Sigma}_F$,
where $\mathbf{\Sigma}_k^{\mathrm{HJ}}$ and $\mathbf{\Sigma}_F$ are given in~\eqref{eq:HJ2 alternative moment} and Lemma~\ref{lem:fidis HT F}, respectively.
By means of the Cram\'er-Wold device, this proves that~\eqref{eq:fidis lumley}
converges in distribution under $\mathbb{P}_{d,m}$ to a mean zero $k$-variate normal random vector with covariance matrix
$\mathbf{\Sigma}^{\mathrm{\pi}}$.
This distribution is the same as that of $\big(\mathbb{G}^{\pi}(t_1),\ldots,\mathbb{G}^{\pi}(t_k)\big)$,
where $\mathbb{G}^{\pi}$ is a mean zero Gaussian process with covariance function
\[
\begin{split}
\lim_{N\to\infty}
\frac{1}{N^2}
\sum_{i=1}^N\sum_{j=1}^N
\mathbb{E}_m
&\left[n
\frac{\pi_{ij}-\pi_i\pi_j}{\pi_i\pi_j}
\left(\mathds{1}_{\{Y_i\leq s\}}-s\right)
\left(\mathds{1}_{\{Y_i\leq t\}}-t\right)
\right]\\
&+
\lambda
\left(
s\wedge t-st
\right),
\quad
s,t\in \mathbb{R}.
\end{split}
\]
Since $\mathbb{G}^{\pi}$ is continuous at 1, the theorem then follows from
Theorem~13.5 in~\cite{Billingsley_1999} for the case of uniform $Y_i$'s.
Extension to $Y_i$'s with a general c.d.f.~$F$ is completely similar
to the proof of Theorem~\ref{th:FCLT HT FN}.}

{\paragraph{Remainder of the proof of Theorem~\ref{th:FCLT HJ FN}}
It remains to prove weak convergence of the finite dimensional projections
\begin{equation}
\label{eq:fidis Y_N}
\big(\mathbb{Y}_N(t_1),\ldots,\mathbb{Y}_N(t_k)\big).
\end{equation}
As before, we apply the Cram\'er-Wold device and consider
\begin{equation}
\label{eq:lincomb Y_N}
a_1\mathbb{Y}_N(t_1)+\cdots+a_k\mathbb{Y}_N(t_k)
=
\sqrt{n}
\left\{
\frac1N\sum_{i=1}^N
\frac{\xi_i}{\pi_i}V_{ik}
-
\frac1N\sum_{i=1}^N
V_{ik}
\right\},
\end{equation}
with
\[
V_{ik}=a_1\left(\mathds{1}_{\{Y_i\leq t_1\}}-t_1)+\cdots+a_k(\mathds{1}_{\{Y_i\leq t_k\}})-t_k\right).
\]
Convergence of~\eqref{eq:lincomb Y_N} is obtained completely similar to that of~\eqref{eq:fidis HT FN} in the proof of Lemma~\ref{lem:fidis HT FN}.
From (HT1) and (HJ2), it follows that~\eqref{eq:lincomb Y_N}
converges in distribution under~$\mathbb{P}_{d,m}$
to $a_1N_1+\cdots+a_kN_k$, where $(N_1,\ldots,N_k)$ has a $k$-variate
normal distribution with covariance matrix~$\mathbf{\Sigma}_k^{\mathrm{HJ}}$
given in~\eqref{eq:HJ2 alternative moment}.
By means of the Cram\'er-Wold device, this proves that~\eqref{eq:fidis Y_N}
converges in distribution under~$\mathbb{P}_{d,m}$ to a mean zero $k$-variate normal random vector with covariance matrix~$\mathbf{\Sigma}_k^{\mathrm{HJ}}$.
This distribution is the same as that of $\big(\mathbb{G}^{\mathrm{HJ}}(t_1),\ldots,\mathbb{G}^{\mathrm{HJ}}(t_k)\big)$,
where $\mathbb{G}^{\mathrm{HJ}}$ is a mean zero Gaussian process with covariance function
\[
\lim_{N\to\infty}
\frac{1}{N^2}
\sum_{i=1}^N\sum_{j=1}^N
\mathbb{E}_m
\left[n
\frac{\pi_{ij}-\pi_i\pi_j}{\pi_i\pi_j}
\left(\mathds{1}_{\{Y_i\leq s\}}-s\right)
\left(\mathds{1}_{\{Y_i\leq t\}}-t\right)
\right],
\]
for $s,t\in \mathbb{R}$.
As before, the theorem now follows from
Theorem~13.5 in~\cite{Billingsley_1999} for the case of uniform $Y_i$'s,
and is then extended to $Y_i$'s with a general c.d.f.~$F$.
\hfill\tqed}

{\paragraph{Proof of Theorem~\ref{th:FCLT HJ F}}
The theorem follows directly from relation~\eqref{eq:relation GL and GHJ} and Theorem~\ref{th:FCLT Lumley}.
\hfill\tqed}

\paragraph{Proof of Proposition~\ref{prop:FCLT HJ F}}
From relation~\eqref{eq:relation GL and GHJ} and Theorem~\ref{th:FCLT Lumley} we
know that the limit behavior of $\sqrt{n}(\FHJ-F)$ is the same as that of $\GL$.
Tightness of $\GL$ has been obtained in the proof of Theorem~\ref{th:FCLT Lumley}.
It remains to establish weak convergence of~\eqref{eq:fidis lumley}.
This can be done in the same way as in the proof of Lemma~\ref{lem:fidis HT F deterministic},
but this time with
\[
V_{ik}=a_1\left(\mathds{1}_{\{Y_i\leq t_1\}}-F(t_1)\right)+\cdots+a_k\left(\mathds{1}_{\{Y_i\leq t_k\}}-F(t_k)\right)
\]
and $\mu_k=0$.
When $0<F(t_1)<\cdots<F(t_k)<1$, from (HT1) and Lemma~\ref{lem:variance HT} we find that
$nS_N^2
\to
\mu_{\pi1}\mathbb{E}_{m}[V_{1k}^2]
=
\mathbf{a}_k^t\mathbf{\Sigma}_k\mathbf{a}_k$,
where
\begin{equation}
\label{eq:SigmaG}
\mathbf{\Sigma}_k
=
\mu_{\pi1}
\Big(
F(t_q\wedge t_r)-F(t_q)F(t_r)
\Big)_{q,r=1}^k.
\end{equation}
From condition~(i) of Proposition~\ref{prop:FCLT HT F} and Lemma~\ref{lem:positive definite},
it follows that $\mathbf{\Sigma}_k$ is positive definite, so that $\mathbf{a}_k^t\mathbf{\Sigma}_k\mathbf{a}_k>0$.
Hence, according to Lemma~\ref{lem:HT CLT}, the right hand side of~\eqref{eq:lincomb Lumley}
converges in distribution under $\mathbb{P}_{d,m}$ to a mean zero normal random variable with variance
$(\mu_{\pi1}+\lambda)\mathbb{E}_{m}[V_{1k}^2]=\mathbf{a}_k^t\SigmaHJF\mathbf{a}_k$,
where
\begin{equation}
\label{eq:SigmaHJF}
\begin{split}
\SigmaHJF
=
\Big(
(\mu_{\pi1}+\lambda)F(t_q\wedge t_r)
\Big)_{q,r=1}^k.
\end{split}
\end{equation}
We conclude that the right hand side of~\eqref{eq:lincomb Lumley} converges in distribution to $a_1N_1+\cdots+a_kN_k$, where $(N_1,\ldots,N_k)$ has a
mean zero $k$-variate
normal distribution with covariance matrix $\SigmaHJF$.
By means of the Cram\'er-Wold device, this proves weak convergence
of $\big(\mathbb{G}^{\pi}_N(t_1),\ldots,\mathbb{G}^{\pi}_N(t_k)\big)$ for the case that $0<F(t_1)<\cdots<F(t_k)<1$.
As in the proof of Lemma~\ref{lem:fidis HT F deterministic}, the case where the $F(t_i)$'s are not all distinct, but satisfy $0<F(t_i)<1$,
the case $F(t_1)=0$, and the case $F(t_k)=1$, can be reduced to the previous case.
From here on, the proof is completely the same as that of Theorem~\ref{th:FCLT Lumley}.
\hfill
\tqed

{\paragraph{Proof of Proposition~\ref{prop:HT1}}
The proposition only needs to be established for the rejective sampling design, as it can be extended to high entropy designs
by means of Theorem~5 in~\cite{berger1998b}.
Since the rejective sampling design can be represented as a Poisson sampling design conditionally on the sample size being equal to $n$,
the proof is along the lines of the arguments used in the proof of Theorem~3.2 in~\cite{Bertail_2013}.
It applies results from~\cite{mirakhmedov2014} on a central limit theorem for sums of functions of independent random variables
$\xi_1,\ldots,\xi_N$, conditional on $\xi_1+\cdots+\xi_N=n$.
Details are provided in the supplement B in~\cite{Boistard_2015}.
\hfill\tqed}

{\paragraph{Proof of Corollary~\ref{cor:high entropy deterministic}}
As in the proof of Corollary~\ref{cor:high entropy}, we first prove the results for rejective sampling and then extend them to high entropy designs.
Completely similar to the proof of Corollary~\ref{cor:high entropy}, conditions (A2)-(A4) imply (C2)-(C4).
Furthermore, condition~(ii) of Proposition~\ref{prop:FCLT HT FN} is obtained in the same way as in the proof of Corollary~\ref{cor:high entropy}, with $\mu_{\pi2}=-\alpha$,
from conditions (A2)-(A3) and (A5).
This proves parts~(i)-(iv).
\hfill\tqed}

\paragraph{Proof of Corollary~\ref{cor:poverty rate HT}}
The mapping $\phi:\mathbb{D}_{\phi}\subset D(\mathbb{R})\mapsto \mathbb{R}$ is Hadamard-differentiable
at $F$ tangentially to the set $\mathbb{D}_0$ consisting
of functions $h\in D(\mathbb{R})$ that are continuous at $F^{-1}(\alpha)$.
According to Theorem~\ref{th:FCLT HT F}, the sequence
$\sqrt{n}(\FHT-F)$ converges weakly to a mean zero Gaussian process
$\mathbb{G}_F^{\mathrm{HT}}$ with covariance structure
\begin{equation}
\label{eq:cov GHT poverty}
\mathbb{E}_{d,m}\mathbb{G}_F^{\mathrm{HT}}(s)\mathbb{G}_F^{\mathrm{HT}}(t)
=
(\mu_{\pi1}+\lambda)F(s\wedge t)+(\mu_{\pi2}-\lambda)F(s)F(t),
\end{equation}
for $s,t\in \mathbb{R}$.
It then follows from Theorem~3.9.4 in~\cite{van_1996},
that the random variable $\sqrt{n}(\phi(\FHT)-\phi(F))$
converges weakly to
\[
-\beta\frac{f(\beta F^{-1}(\alpha))}{f(F^{-1}(\alpha))}
\mathbb{G}_F^{\mathrm{HT}}(F^{-1}(\alpha))
+
\mathbb{G}_F^{\mathrm{HT}}(\beta F^{-1}(\alpha)),
\]
which has a normal distribution with mean zero and variance
\[
\begin{split}
\sigma_{\mathrm{HT},\alpha,\beta}^2
&=
\beta^2\frac{f(\beta F^{-1}(\alpha))^2}{f(F^{-1}(\alpha))^2}
\mathbb{E}\left[\mathbb{G}^{\mathrm{HT}}_F(F^{-1}(\alpha))^2\right]\\
&\qquad+
\mathbb{E}\left[\mathbb{G}^{\mathrm{HT}}_F(\beta F^{-1}(\alpha))^2\right]\\
&\qquad
-2\beta\frac{f(\beta F^{-1}(\alpha))}{f(F^{-1}(\alpha))}
\mathbb{E}\left[\mathbb{G}^{\mathrm{HT}}_F(F^{-1}(\alpha))\mathbb{G}^{\mathrm{HT}}_F(\beta F^{-1}(\alpha))\right].
\end{split}
\]
The precise expression can then be derived from~\eqref{eq:cov GHT poverty},
which proves part one.
For part two, write
\[
\sqrt{n}\left(\phi(\FHT)-\phi(\FN)\right)
=
\sqrt{n}\left(\phi(\FHT)-\phi(F)\right)
+
\frac{\sqrt{n}}{\sqrt{N}}
\sqrt{N}\left(\phi(\FN)-\phi(F)\right).
\]
The process $\sqrt{N}(\FN-F)$ converges weakly to a mean zero Gaussian process~$\mathbb{G}_F$.
Then, Hadamard-differentiability of $\phi$ together with Theorem~3.9.4 in~\cite{van_1996}
yields that the sequence $\sqrt{N}(\phi(\FN)-\phi(F))$ converges weakly to $\phi_F'(\mathbb{G}_F)$.
As $n/N\to0$, the theorem follows from part one.
\hfill
\tqed

\paragraph{Proof of Corollary~\ref{cor:poverty rate HJ}}
The proof is completely the same as that of Corollary~\ref{cor:poverty rate HT},
with the only difference that the covariance structure of the limiting process
$\sqrt{n}(\phi(\FHJ)-\phi(F))$ is now given in Theorem~\ref{th:FCLT HJ F}.
\hfill
\tqed

\section{Additional technicalities}
\paragraph{Comment about (C1) on page~\pageref{page:C1}}
Condition A3 in~\cite{conti2015} requires that
\begin{equation}
\label{A3}
\lim_{N,n\to\infty}
\mathbb{E}[\pi_i(1-\pi_i)]=d>0,
\end{equation}
where $0<d\leq 1/4$.
The parabola $x\mapsto x(1-x)-d$ is strictly positive for
\[
0<\frac{1-\sqrt{1-4d}}{2}<x<\frac{1+\sqrt{1-4d}}{2}<1.
\]
According to condition A4 in~\cite{conti2015}, it holds that $n/N\to\lambda>0$.
Suppose that the lower bound in (C1) does not hold, so that $N\pi_i/n$ can be arbitrarily small, say
\[
\frac{N\pi_i}n< \frac{1-\sqrt{1-4d}}{4\lambda}.
\]
In that case
\[
\lim_{N\to\infty}\pi_i
=
\lim_{N\to\infty}
\frac{n}{N}\cdot \frac{N\pi_i}n
<
\lambda\cdot\frac{1-\sqrt{1-4d}}{4\lambda}
=
\frac{1-\sqrt{1-4d}}{4},
\]
which lies left of the smallest zero of the parabola $x(1-x)-d$.
As a consequence
\[
\lim_{N,n\to\infty}
\mathbb{E}[\pi_i(1-\pi_i)]<d,
\]
which is in contradiction with~\eqref{A3}.
\hfill\tqed

\begin{lemma}
\label{lem:variance HT}
Let $S_N^2$ be defined by~\eqref{def:variance HT}, where $V_1,V_2,\ldots$ is a sequence of i.i.d.~ random variables on $(\Omega,\mathfrak{F},\mathbb{P}_{m})$ with $\mathbb{E}_{m}[V_1^4]<\infty$.
Suppose that $n$ and $\pi_i,\pi_{ij}$, for $i,j=1,2,\ldots,N$ are deterministic and let $\mathbb{V}_m(S_N^2)$ denote the variance of $S_N^2$.
If (C1)-(C2) hold, then $n^2\mathbb{V}_m[S_N^2]=O(1/N)$.
Then,
\begin{itemize}
\item[(i)]
if $\mathbb{E}_m[V_1]=0$ and condition (i) in Proposition~\ref{prop:FCLT HT FN} holds,
\[
nS_N^2\to\sigma_{\mathrm{HT}}^2=\mu_{\pi1}\mathbb{E}_{m}[V_1^2],
\quad\text{in $\mathbb{P}_{m}$-probability.}
\]
\item[(ii)]
if $\mathbb{E}_m[V_1]\ne0$ and conditions (i)-(ii) in Proposition~\ref{prop:FCLT HT FN} hold,
\[
nS_N^2\to\sigma_{\mathrm{HT}}^2
=
\mu_{\pi1}\mathbb{E}_{m}[V_1^2]+\mu_{\pi2}\left(\mathbb{E}_{m}[V_1]\right)^2,
\quad\text{in $\mathbb{P}_{m}$-probability.}
\]
\end{itemize}
\end{lemma}
\begin{proof}
For any $\epsilon >0$, by Markov inequality we have
\begin{equation}
\label{eq:markov}
\mathbb{P}_m\left\{ |nS_N^2 -\mathbb{E}_m[nS_N^2]| >\epsilon \right\}
<
\frac{n^2\mathbb{V}_m[S_N^2]}{\epsilon^2},
\end{equation}
where $\mathbb{V}_m$ denotes the variance of $S_N^2$ under the super-population model.
In order to compute $\mathbb{V}_m[S_N^2]$, we first have
\begin{equation}
\label{eq:ES}
\begin{split}
\mathbb{E}_m[S_N^2]
&=
\frac{1}{N^2}
\sum_{i=1}^N\sum_{j=1}^N
\frac{\pi_{ij}-\pi_i\pi_j}{\pi_i\pi_j}
\mathbb{E}_m(V_iV_j)\\
&=
\frac{\mathbb{E}_m[V_1^2]}{N^2}
\sum_{i=1}^N
\frac{1-\pi_i}{\pi_i}
+
\frac{\left(\mathbb{E}_{m}[V_1]\right)^2}{N^2}
\ssum_{i\ne j}
\frac{\pi_{ij}-\pi_i\pi_j}{\pi_i\pi_j}.
\end{split}
\end{equation}
From this, tedious but straightforward calculus leads to the expression for $(\mathbb{E}_m[S_N^2])^2$
and $\mathbb{E}_m[S_N^4]$.
One finds
\[
N^4\left(\mathbb{E}_m\left[S_N^2\right]\right)^2
=
a_1\left(\mathbb{E}_m[V_1]\right)^4
+
a_2\mathbb{E}_m\left[V_1^2\right]\left(\mathbb{E}_m\left[V_1\right]\right)^2
+
a_3\left(\mathbb{E}_m\left[V_1^2\right]\right)^2,
\]
where, according to (C1)-(C2):
\[
\begin{split}
a_1
&=
\ssssum_{(i,j,k,l)\in D_{4,N}}\frac{\pi_{ij}-\pi_i\pi_j}{\pi_i\pi_j}\frac{\pi_{kl}-\pi_k\pi_l}{\pi_k\pi_l}\\
&\qquad+
4\sssum_{(i,j,l)\in D_{3,N}}\frac{\pi_{ij}-\pi_i\pi_j}{\pi_i\pi_j}\frac{\pi_{il}-\pi_i\pi_l}{\pi_i\pi_l}
+
2\ssum_{(i,j)\in D_{2,N}}\left(\frac{\pi_{ij}-\pi_i\pi_j}{\pi_i\pi_j}\right)^2\\
&=
\ssssum_{(i,j,k,l)\in D_{4,N}}\frac{\pi_{ij}-\pi_i\pi_j}{\pi_i\pi_j}\frac{\pi_{kl}-\pi_k\pi_l}{\pi_k\pi_l}
+
O(N^3/n^2)
+
O(N^2/n^2)\\
a_2
&=
2\sssum_{(i,k,l)\in D_{3,N}}\frac{1-\pi_i}{\pi_i}\frac{\pi_{kl}-\pi_k\pi_l}{\pi_k\pi_l}
+
4\ssum_{(i,k)\in D_{2,N}}\frac{1-\pi_i}{\pi_i}\frac{\pi_{ik}-\pi_i\pi_k}{\pi_i\pi_k}\\
&=
2\sssum_{(i,k,l)\in D_{3,N}}\frac{1-\pi_i}{\pi_i}\frac{\pi_{kl}-\pi_k\pi_l}{\pi_k\pi_l}
+
O(N^3/n^2)\\
a_3
&=
\ssum_{(i,j)\in D_{2,N}}\frac{1-\pi_i}{\pi_i}\frac{1-\pi_j}{\pi_j}
+
\sum_{i=1}^N\left(\frac{1-\pi_i}{\pi_i}\right)^2\\
&=
\ssum_{(i,j)\in D_{2,N}}\frac{1-\pi_i}{\pi_i}\frac{1-\pi_j}{\pi_j}
+
O(N^3/n^2).
\end{split}
\]
Furthermore,
\[
\begin{split}
N^4\mathbb{E}_m\left[S_N^4\right]
=
b_1\left(\mathbb{E}_m[V_1]\right)^4
&+
b_2\mathbb{E}_m\left[V_1^2\right]\left(\mathbb{E}_m[V_1]\right)^2\\
&+
b_3\left(\mathbb{E}_m\left[V_1^2\right]\right)^2
+
b_4
\mathbb{E}_m[V_1]\mathbb{E}_m\left[V_1^3\right]
\end{split}
\]
where
\[
\begin{split}
b_1
&=
\ssssum_{(i,j,k,l)\in D_{4,N}}\frac{\pi_{ij}-\pi_i\pi_j}{\pi_i\pi_j}\frac{\pi_{kl}-\pi_k\pi_l}{\pi_k\pi_l}
+
\sum_{i=1}^N\left(\frac{1-\pi_i}{\pi_i}\right)^2\\
&=
\ssssum_{(i,j,k,l)\in D_{4,N}}\frac{\pi_{ij}-\pi_i\pi_j}{\pi_i\pi_j}\frac{\pi_{kl}-\pi_k\pi_l}{\pi_k\pi_l}
+
O(N^3/n^2)\\
b_2
&=
2\sssum_{(i,k,l)\in D_{3,N}}\frac{1-\pi_i}{\pi_i}\frac{\pi_{kl}-\pi_k\pi_l}{\pi_k\pi_l}
+
4\sssum_{(i,j,l)\in D_{3,N}}\frac{\pi_{ij}-\pi_i\pi_j}{\pi_i\pi_j}\frac{\pi_{il}-\pi_i\pi_l}{\pi_i\pi_l}\\
&=
2\sssum_{(i,k,l)\in D_{3,N}}\frac{1-\pi_i}{\pi_i}\frac{\pi_{kl}-\pi_k\pi_l}{\pi_k\pi_l}
+
O(N^3/n^2)
\end{split}
\]
\[
\begin{split}
b_3
&=
\ssum_{(i,k)\in D_{2,N}}\frac{1-\pi_i}{\pi_i}\frac{1-\pi_k}{\pi_k}
+
2\ssum_{(i,j)\in D_{2,N}}\left(\frac{\pi_{ij}-\pi_i\pi_j}{\pi_i\pi_j}\right)^2\\
&=
\ssum_{(i,k)\in D_{2,N}}\frac{1-\pi_i}{\pi_i}\frac{1-\pi_k}{\pi_k}
+
O(N^2/n^2)\\
b_4
&=
4\ssum_{(i,j)\in D_{2,N}}\frac{\pi_{ij}-\pi_i\pi_j}{\pi_i\pi_j}\frac{1-\pi_j}{\pi_j}
=
O(N^3/n^2).
\end{split}
\]
The variance expression for $S_N^2$ is deduced easily from the previous computations.
From the expression derived in~\cite{Boistard_2015}, we find that $a_i-b_i=O(N^3/n^2)$, for $i=1,2,3$,
and $b_4=O(N^3/n^2)$, so that
\begin{equation}
\label{eq:variance HT}
n^2\mathbb{V}_m[S_N^2]
=
n^2\mathbb{E}_m[S_N^4]-n^2\left(\mathbb{E}_m[S_N^2]\right)^2
=
O(1/N).
\end{equation}
From~\eqref{eq:markov} we conclude that
$nS_N^2 -\mathbb{E}_m[nS_N^2]$ tends to zero in $\mathbb{P}_m$-probability.
As a consequence, statements (i) and (ii) follow from~\eqref{eq:ES}.
\end{proof}

{
\paragraph{Proof of Lemma~\ref{lem:variance HT} under (C2$^*$)}
We used (C2) to bound remainder terms in the coefficients $a_i$ and $b_i$,
but this can also be achieved with~(C2$^*$).
For the second term in $a_1$ we get
\[
\begin{split}
\left|
\sssum_{(i,j,l)\in D_{3,N}}\frac{\pi_{ij}-\pi_i\pi_j}{\pi_i\pi_j}\frac{\pi_{il}-\pi_i\pi_l}{\pi_i\pi_l}
\right|
&\leq
\ssum_{(i,j)\in D_{2,N}}
\left|
\frac{\pi_{ij}-\pi_i\pi_j}{\pi_i\pi_j}
\right|
\cdot
\sum_{l\ne i,j}
\left|
\frac{\pi_{il}-\pi_i\pi_l}{\pi_i\pi_l}
\right|\\
&=
N\cdot O\left(\frac{N}{n}\right)\cdot O\left(\frac{N}{n}\right)
=
O\left(\frac{N^3}{n^2}\right),
\end{split}
\]
by means of~(C2$^*$).
For the third term in $a_1$, we have
\[
\begin{split}
\ssum_{(i,j)\in D_{2,N}}\left(\frac{\pi_{ij}-\pi_i\pi_j}{\pi_i\pi_j}\right)^2
&\leq
\ssum_{(i,j)\in D_{2,N}}\frac{\left|\pi_{ij}-\pi_i\pi_j\right|}{\pi_i\pi_j}\cdot\frac{\pi_{ij}}{\pi_i\pi_j}\\
&=
N\cdot O\left(\frac{N}{n}\right)\cdot O\left(\frac{N}{n}\right)
=O\left(\frac{N^3}{n^2}\right),
\end{split}
\]
by means of~(C2$^*$) and (C1) and the fact that $\pi_{ij}\leq\pi_i$.
For the second term in $a_2$ we have
\[
\begin{split}
\left|
\ssum_{(i,k)\in D_{2,N}}\frac{1-\pi_i}{\pi_i}\frac{\pi_{ik}-\pi_i\pi_k}{\pi_i\pi_k}
\right|
&\leq
\sum_{i=1}^N
\left(
\frac{1}{\pi_i}-1
\right)
\sum_{k\ne i}
\frac{|\pi_{ik}-\pi_i\pi_k|}{\pi_i\pi_k}\\
&=
O\left(\frac{N^2}{n}\right)\cdot O\left(\frac{N}{n}\right)=O\left(\frac{N^3}{n^2}\right),
\end{split}
\]
by means of condition~(i) and~(C2$^*$).
For the remainder terms in $b_2,b_3,b_4$ we obtain bounds for the same quantities, as the previous three.
The rest of the proof of Lemma~\ref{lem:variance HT} remains the same.
\hfill\tqed
}
\begin{lemma}
\label{lem:tightness sum}
If $x_N\rightsquigarrow x$ and $y_N\rightsquigarrow y$ in $D[0,1]$ with the Skorohod metric,
and $x,y\in C[0,1]$, then the sequence $\{x_N+y_N\}$ is also tight in $D[0,1]$.
\end{lemma}
\begin{proof}
We can use Theorem~13.2 from~\cite{Billingsley_1999}.
The first condition follows easily since
\[
\sup_{t\in[0,1]}
|x_N(t)+y_N(t)|
\leq
\sup_{t\in[0,1]}|x_N(t)|
+
\sup_{t\in[0,1]}
|y_N(t)|.
\]
Because $x_N\rightsquigarrow x$ and $y_N\rightsquigarrow y$ both sequences $\{x_N\}$ and $\{y_N\}$ are tight,
so that they satisfy the first condition of Theorem~13.2 individually.
For condition (ii) of Theorem~13.2 in~\cite{Billingsley_1999}, choose $\epsilon>0$.
According to (12.7) in~\cite{Billingsley_1999}, for any $0<\delta<1/2$,
\[
w'_x(\delta)\leq w_x(2\delta).
\]
This means that
\[
\begin{split}
\mathbb{P}
\left\{
w_{x_N+y_N}'(\delta)\geq \epsilon
\right\}
&\leq
\mathbb{P}
\left\{
w_{x_N+y_N}(2\delta)\geq \epsilon
\right\}\\
&\leq
\mathbb{P}
\left\{
w_{x_N}(2\delta)\geq \epsilon/2
\right\}
+
\mathbb{P}
\left\{
w_{y_N}(2\delta)\geq \epsilon/2
\right\}.
\end{split}
\]
Consider the first probability.
Since $x_N\rightsquigarrow x$ in $D[0,1]$ with the Skorohod metric,
according to the almost sure representation theorem (see, e.g., Theorem~11.7.2 in~\cite{dudley2002}), there exist
$\widetilde{x}_n$ and~$\widetilde{x}$,
having the same distribution as $x_N$ and $x$, respectively,
such that
$\widetilde{x}_N\to \widetilde{x}$,
with probability one,
in the Skorohod metric.
Because $\widetilde{x}\stackrel{d}{=} x$ and $x\in C[0,1]$, also $\widetilde{x}\in C[0,1]$.
Hence, since $\widetilde{x}$ is continuous, it follows that
\begin{equation}
\label{eq:conv as tilde}
\sup_{t\in[0,1]}
|\widetilde{x}_N(t)-\widetilde{x}(t)|
\to
0,
\qquad
\text{with probability one.}
\end{equation}
We then find that
\[
\begin{split}
&
\mathbb{P}
\left\{
w_{x_N}(2\delta)\geq \epsilon/2
\right\}
=
\mathbb{P}
\left\{
\sup_{|s-t|<2\delta}|x_N(s)-x_N(t)|\geq \epsilon/2
\right\}\\
&=
\mathbb{P}
\left\{
\sup_{|s-t|<2\delta}|\widetilde{x}_N(s)-\widetilde{x}_N(t)|\geq \epsilon/2
\right\}\\
&\leq
\mathbb{P}
\left\{
\sup_{|s-t|<2\delta}|\widetilde{x}(s)-\widetilde{x}(t)|\geq \epsilon/4
\right\}\\
&\quad+
\mathbb{P}
\left\{
\sup_{s\in[0,1]}|\widetilde{x}_N(s)-\widetilde{x}(s)|\geq \epsilon/8
\right\}
+
\mathbb{P}
\left\{
\sup_{t\in[0,1]}|\widetilde{x}_N(t)-\widetilde{x}(t)|\geq \epsilon/8
\right\}.
\end{split}
\]
The latter two probabilities tend to zero due to to~\eqref{eq:conv as tilde}.
For the first probability on the right hand side, note that $C[0,1]$ is separable and complete.
This means that each random element in $C[0,1]$ is tight.
Hence, $\widetilde{x}\in C[0,1]$ is tight, so that according to Theorem 7.3 in~\cite{Billingsley_1999},
there exists a $0<\delta<1/2$, such that
\[
\mathbb{P}
\left\{
\sup_{|s-t|<2\delta}|x(s)-x(t)|\geq \epsilon/4
\right\}
=
\mathbb{P}
\left\{
w_{x}(2\delta)\geq \epsilon/4
\right\}
\leq
\eta.
\]
We conclude that
$\mathbb{P}
\left\{
w_{x_N}(2\delta)\geq \epsilon/2
\right\}
\to
0$,
and the same result for $y_N$ can be obtained similarly.
This proves the lemma.
\end{proof}

{
\paragraph{Proof of Proposition~\ref{prop:HT1}}
It suffices to prove (HT1) for rejective sampling.
The proof is along the lines of the proof of Theorem~3.2 in~\cite{Bertail_2013} and uses results from~\cite{mirakhmedov2014}.
To adapt to the notation used in~\cite{mirakhmedov2014}, we will show that
\begin{equation}
\label{eq:HT1mirakhmedov}
\frac1{S_N}
\left(
\frac{1}{N}
\sum_{i=1}^N
\frac{\eta_iV_i}{\pi_i}
-
\frac{1}{N}
\sum_{i=1}^N
V_i
\right)
%\stackrel{d}{\to}
\to
N(0,1),
\qquad
\omega-\text{a.s.},
\end{equation}
in distribution under $\mathbb{P}_d$,
where
\[
S_N^2
=
\text{Var}_d\left[\frac1N\sum_{i=1}^N \left(\frac{\eta_i}{\pi_i}-1\right)V_i\right]
=
\frac{1}{N^2}
\sum_{i=1}^N\sum_{j=1}^N
\frac{\pi_{ij}-\pi_i\pi_j}{\pi_i\pi_j}V_iV_j.
\]
Here, the $\eta_1,\ldots,\eta_N$ represent the inclusion indicators corresponding to the rejective sampling design.
The rejective sampling design can be represented by a Poisson design conditional on the sample size being equal to~$n$
(e.g., see~\cite{Hajek_1964})
Let $\xi_1,\ldots,\xi_N$ denote the indicators of the corresponding Poisson design.
Note that $\mathbb{E}_d[\eta_i]=\pi_i$ and $\mathbb{E}_d[\xi_i]=p_i$, where the $p_i$'s can be chosen such that~$\sum_{i=1}^Np_i=n$,
and that $d_N=\sum_{i=1}^N\pi_i(1-\pi_i)\to\infty$, as a consequence of~(B2).

In order to obtain~\eqref{eq:HT1mirakhmedov}, it is more convenient to rewrite the left hand side.
To this end, note that by means of Theorem~5.1 in~\cite{Hajek_1964} and
the fact that $\sum_{i=1}^N\eta_i=\sum_{i=1}^Np_i=n$, we can write
\begin{equation}\label{eq:HT1equivalence}
\frac1{NS_N}
\sum_{i=1}^N
\left(
\eta_i-\pi_i
\right)
\frac{V_i}{\pi_i}\\
=
(1+o(1))
\frac1{NS_N}
\sum_{i=1}^N
\left(
\eta_i-p_i
\right)
\left(
\frac{V_i}{p_i}-\theta_N
\right)
\end{equation}
where
\begin{equation}
\label{def:BN}
\begin{split}
\theta_N
&=
\frac{1}{B_N^2}\sum_{i=1}^NV_i(1-p_i),\\
B_N^2
&=
\sum_{i=1}^N p_i(1-p_i)
=
(1+o(1))d_N,
\end{split}
\end{equation}
according to Theorem~5.1 in~\cite{Hajek_1964}.
The summation on the right hand side of~\eqref{eq:HT1equivalence} is of the form
\[
R_N(\eta)
=
\sum_{m=1}^N
f_{m,N}(\eta_m),
\quad
\text{where }
f_{m,N}(y)=\frac{1}{NS_N}
\left(
y-p_m
\right)
\left(
\frac{V_m}{p_m}
-
\theta_N
\right),
\]
which is of the type considered in~\cite{mirakhmedov2014}.
Furthermore, note that
\[
\begin{split}
\Lambda_N
&=
\sum_{m=1}^N \mathbb{E}_d\left[f_{m,N}(\xi_m)\right]=0\\
\gamma_N
&=
\frac{1}{B_N^2}
\sum_{m=1}^N
\text{cov}\left(f_{m,N}(\xi_m),\xi_m\right)=0.
\end{split}
\]
Under suitable conditions on that we specify below
\[
g_m(y)
=
f_{m,N}(y)-\mathbb{E}_df_{m,N}(\xi_m)-\gamma_N(y-\mathbb{E}_d\xi_m)
=
f_{m,N}(y),
\]
according to Theorem~3.1 in~\cite{mirakhmedov2014},
\begin{equation}
\label{eq:CLTmirakhmedov}
\frac{R_N(\eta)}{\sigma_N}\to N(0,1)
\end{equation}
in distribution, where
\[
\sigma^2_N
=
\sum_{m=1}^N
\text{Var}
\left[g_m(\xi_m)\right]
=
\frac{1}{N^2S_N^2}
\sum_{m=1}^N
\left(
\frac{V_m}{p_m}
-
\theta_N
\right)^2
p_m(1-p_m).
\]
From Theorem~5.1 and~6.1 in~\cite{Hajek_1964}, it follows that
\begin{equation}
\label{eq:conv sigmaN}
\sigma^2_N
=
(1+o(1))
\frac{1}{N^2S_N^2}
\sum_{m=1}^N
\left(
\frac{V_m}{\pi_m}
-
R\right)^2
\pi_m(1-\pi_m)
=
1+o(1),
\end{equation}
where $R=d_N^{-1}\sum_{i=1}^N\pi_i(1-\pi_i)$.
Therefore, \eqref{eq:HT1mirakhmedov} is equivalent with~\eqref{eq:CLTmirakhmedov} and it remains to check the conditions
of Theorem~3.1 in~\cite{mirakhmedov2014}.

Define (as mentioned in~\cite{Bertail_2013}, a factor $\sqrt{N}$ after $\epsilon$ is missing in~\cite{mirakhmedov2014})
\[
\begin{split}
\mathcal{L}_{1,N}(\epsilon)
&=
\frac{1}{B_N^{3}}
\sum_{m=1}^N \mathbb{E}_d\left|\xi_m-p_m\right|^3\mathds{1}\left\{|\xi_m-p_m|\leq \epsilon B_N\right\}
\\
\mathcal{L}_{2,N}(\epsilon)
&=
\frac{1}{B_N^2}
\sum_{m=1}^N \mathbb{E}_d\left|\xi_m-p_m\right|^2\mathds{1}\left\{|\xi_m-p_m|>\epsilon B_N\right\}\\
L_{2,N}(\epsilon)
&=
\frac{1}{\sigma_N^2}
\sum_{m=1}^N \mathbb{E}_d g_m(\xi_m)^2\mathds{1}\left\{|g_m(\xi_m)|>\epsilon \sigma_N\right\}\\
M_N(T)
&=
\inf_{T\leq \tau\leq \pi}
\sum_{m=1}^N
\left(
1-|\mathbb{E}_d\exp(i\tau \xi_m)|^2
\right),
\end{split}
\]
if $T\leq \pi$ else $M_N(T)=\infty$.
If for arbitrary $\epsilon>0$,
\begin{itemize}
  \item[(i)]
  $L_{2,N}(\epsilon)\to 0$,
  \item[(ii)]
  $\mathcal{L}_{2,N}(\epsilon)\to 0$,
  \item[(iii)]
  $M_N\left(\pi(4B_N\mathcal{L}_{1,N}(\epsilon))^{-1}\right)\to\infty$
  \item[(iv)]
  $\min\left(B_N,\sqrt{N}\right)=o\left(M_N\left(\pi(4B_N\mathcal{L}_{1,N}(\epsilon))^{-1}\right)\right)$
\end{itemize}
then~\eqref{eq:CLTmirakhmedov} holds, according to Theorem~3.1 in~\cite{mirakhmedov2014}.

\emph{ad(i).}
Since $|V_i|\leq K$ and $p_m/\pi_m=1+o(1)$, according to Theorem~5.1 in~\cite{Hajek_1964}, together with (C1) it follows that
for $N$ sufficiently large
\[
|g_m(\xi_m)|
\leq
\frac{2K}{NS_N}\left(\frac{N}{nK_1}+\frac{N}{B_N^2}\right)
\leq
\frac{2K}{nS_N}\left(\frac{1}{K_1}+\frac{n}{B_N^2}\right).
\]
Together, with condition (B1), there exists $C>0$, such that
\[
L_{2,N}(\epsilon)
\leq
\frac{C}{n^2S_N^2}
\frac{1}{\sigma_N^2}
\sum_{m=1}^N
\frac{\mathbb{E}_d g_m(\xi_m)^2}{\epsilon^2\sigma_N^2}
=
\frac{C}{n^2S_N^2}
\frac{1}{\epsilon^2\sigma_N^2}\to0
\]
according to (B3) and~\eqref{eq:conv sigmaN}.
This proves condition~(i) in~\cite{mirakhmedov2014}.

\emph{ad(ii).}
Since $B_N^2=d_N\to\infty$, for $N$ sufficiently large,  $\left\{|\xi_m-p_m|>\epsilon B_N\right\}\subset \left\{2>\epsilon B_N\right\}=\emptyset$,
which means that for $N$ sufficiently large $\mathcal{L}_{2,N}(\epsilon)=0$.
This proves condition~(ii) in~\cite{mirakhmedov2014}.

\emph{ad(iii)}
First note that (see also~\cite{Bertail_2013})
\[
|\mathbb{E}_d\exp(i\tau \xi_m)|^2
=
1+2p_m(1-p_m)\left(\cos \tau -1\right)
\]
so that for $T\in[0,\pi]$,
\[
M_N(T)=2\inf_{T\leq \tau\leq \pi}(1-\cos\tau)\sum_{m=1}^Np_m(1-p_m)=2B_N^2(1-\cos T).
\]
Because $B_N^2=d_N\to\infty$, for $N$ sufficiently large $\mathds{1}\left\{|\xi_m-p_m|\leq \epsilon B_N\right\}=1$.
This means that for $N$ sufficiently large
\[
\mathcal{L}_{1,N}(\epsilon)
=
\frac{1}{B_N^{3}}\sum_{m=1}^N \mathbb{E}_d\left|\xi_m-p_m\right|^3,
\]
where
\[
\mathbb{E}_d\left|\xi_m-p_m\right|^3
=
p_m(1-p_m)\left\{1-2p_m+2p_m^2\right\}.
\]
It follows that
\[
\frac12p_m(1-p_m)\leq \mathbb{E}_d\left|\xi_m-p_m\right|^3\leq p_m(1-p_m),
\]
so that for $N$ sufficiently large,
$2\leq 4B_N\mathcal{L}_{1,N}(\epsilon)\leq 4$,
and therefore
\[
\begin{split}
M_N\left(\pi(4B_N\mathcal{L}_{1,N}(\epsilon))^{-1}\right)
&=
2B_N^2(1-\cos \left(\pi(4B_N\mathcal{L}_{1,N}(\epsilon))^{-1}\right))\\
&\geq
2B_N^2(1-\cos \left(\pi/4\right))
\to\infty.
\end{split}
\]
This proves condition~(iii) in~\cite{mirakhmedov2014}.

\emph{ad(iv).}
From the previous computations it follows that
\[
\begin{split}
\frac{\min(B_N,\sqrt{N})}{M_N\left(\pi(4B_N\mathcal{L}_{1,N}(\epsilon))^{-1}\right)}
&\leq
\frac{\min(B_N,\sqrt{N})}{2B_N^2(1-\cos(\pi/4))}\\
&=
\frac{1}{2(1-\cos(\pi/4))}
\min\left(1/B_N,\sqrt{N}/B_N^2\right)
\to0
\end{split}
\]
according to (B2) and the fact that $B_N^2=d_N\to\infty$.
This proves condition~(iv) in~\cite{mirakhmedov2014}.
\hfill\tqed

}

\paragraph{Proof of~\eqref{eq:derivative phi}}
Following~\cite{dell2008}, one can write $\phi=\psi_2\circ \psi_1$, where
\[
\begin{split}
\psi_1(F)
&=
\left(F,\beta F^{-1}(\alpha)\right)\\
\psi_2(F,x)
&=
F(x).
\end{split}
\]
The Hadamard-derivative of $\phi$ can then be obtained from the chain rule,
e.g., see Lemma~3.9.3 in~\cite{van_1996}.
According to Lemma~3.9.20 in~\cite{van_1996}, for $0<\alpha<1$
and $F\in \mathbb{D}_{\phi}$ that have a positive derivative at $F^{-1}(\alpha)$, the map
$\psi_1$ is Hadamard-differentiable at $F$ tangentially to the set of functions
$h\in D(\mathbb{R})$ that are continuous at~$F^{-1}(\alpha)$ with derivative
\[
\psi_{1,F}'(h)
=
\left(
h,-\beta\frac{h(F^{-1}(\alpha))}{f(F^{-1}(\alpha))}
\right).
\]
It is fairly straightforward to show that for $F$ that are differentiable at $x$,
the mapping $\psi_2$ is Hadamard-differentiable at $(F,x)$
tangentially to the set of pairs $(h,\epsilon)$, such that
$h$ is continuous at $x$ and $\epsilon\in \mathbb{R}$, with derivative
\[
\psi_{2,(F,x)}'(h,\epsilon)
=
\epsilon f(x)+h(x).
\]
Then for $F\in \mathbb{D}_{\phi}$ that are differentiable at $\beta F^{-1}(\alpha)$,
the mapping $\psi_2$ is Hadamard-differentiable at $\psi_{1}(F)=\left(F,\beta F^{-1}(\alpha)\right)$.
It follows
from the chain rule that $\phi(F)=F\left(\beta F^{-1}(\alpha)\right)=\psi_2\circ \psi_1(F)$ is
Hadamard-differentiable at~$F$ tangentially to the set $\mathbb{D}_0$ consisting
of functions $h\in D(\mathbb{R})$ that are continuous at $F^{-1}(\alpha)$ with derivative
\[
\phi_F'(h)
=
-\beta\frac{f(\beta F^{-1}(\alpha))}{f(F^{-1}(\alpha))}
h(F^{-1}(\alpha))
+
h(\beta F^{-1}(\alpha)).
\]
\hfill\tqed

\end{document}